\documentclass[11pt, reqno]{amsart}
\usepackage{amsmath, amssymb, graphicx, enumerate, amsthm, amscd}
\usepackage[pdfencoding=auto, psdextra]{hyperref}

\usepackage{xcolor}
\hypersetup{
    colorlinks,
    linkcolor={red!50!black},
    citecolor={blue!50!black},
    urlcolor={blue!80!black}
}

\title{Unimodal measurable pseudo-Anosov maps}
\date{March 2025}
\author{Philip Boyland}
\address{Department of Mathematics\\University of Florida\\372 Little
    Hall\\Gainesville\\ FL 32611-8105, USA}
\email{boyland@ufl.edu}
\author{Andr\'e de Carvalho}
\address{Departamento de Matem\'atica Aplicada\\ IME-USP\\ Rua Do Mat\~ao
    1010\\ Cidade Universit\'aria\\ 05508-090 S\~ao Paulo SP\\ Brazil}
\email{andre@ime.usp.br}
\author{Toby Hall}
\address{Department of Mathematical Sciences\\ University of Liverpool\\
    Liverpool L69 7ZL, UK}
\email{tobyhall@liverpool.ac.uk}

\thanks{The authors gratefully acknowledge the support of S\~ao Paulo Research
Foundation (FAPESP) grant number 2016/25053-8.}

\graphicspath{{./}{figures/}{../figures}}

\newtheorem{thm}{Theorem}[section]
\newtheorem{lem}[thm]{Lemma}

\newtheorem*{unnumthm}{Theorem}

\theoremstyle{definition}
\newtheorem{defn}[thm]{Definition}
\newtheorem{defns}[thm]{Definitions}
\newtheorem{notn}[thm]{Notation}

\theoremstyle{remark}
\newtheorem{rem}[thm]{Remark}
\newtheorem{rems}[thm]{Remarks}

\newcommand{\cA}{{{\mathcal{A}}}}
\newcommand{\cB}{{{\mathcal{B}}}}
\newcommand{\cC}{{{\mathcal{C}}}}
\newcommand{\cE}{{{\mathcal{E}}}}
\newcommand{\cF}{{{\mathcal{F}}}}
\newcommand{\cI}{{{\mathcal{I}}}}
\newcommand{\cJ}{{{\mathcal{J}}}}
\newcommand{\cK}{{{\mathcal{K}}}}
\newcommand{\cL}{{{\mathcal{L}}}}
\newcommand{\cS}{{{\mathcal{S}}}}
\newcommand{\cT}{{{\mathcal{T}}}}
\newcommand{\cU}{{{\mathcal{U}}}}
\newcommand{\cV}{{{\mathcal{V}}}}
\newcommand{\cZ}{{{\mathcal{Z}}}}

\newcommand{\fs}{{{\mathfrak{s}}}}
\newcommand{\fu}{{{\mathfrak{u}}}}
\newcommand{\fsp}{{{\mathfrak{s}'}}}
\newcommand{\fup}{{{\mathfrak{u}'}}}

\newcommand{\N}{{{\mathbb{N}}}}

\newcommand{\R}{{{\mathbb{R}}}}
\newcommand{\Z}{{{\mathbb{Z}}}}

\newcommand{\be}{{{\mathbf{e}}}}

\newcommand{\bx}{{{\mathbf{x}}}}
\newcommand{\by}{{{\mathbf{y}}}}
\newcommand{\bz}{{{\mathbf{z}}}}
\newcommand{\bP}{{{\mathbf{P}}}}
\newcommand{\bQ}{{{\mathbf{Q}}}}

\newcommand{\hf}{{{\widehat{f}}}}
\newcommand{\hI}{{{\widehat{I}}}}

\newcommand{\ha}{{{\hat{a}}}}

\newcommand{\hx}{{{\hat{x}}}}
\newcommand{\hz}{{{\hat{z}}}}
\newcommand{\hmu}{{{\hat{\mu}}}}

\newcommand{\tB}{\Theta}

\DeclareMathOperator{\Int}{Int}
\DeclareMathOperator{\orb}{orb}
\DeclareMathOperator{\GO}{GO}
\DeclareMathOperator{\diam}{diam}
\DeclareMathOperator{\pA}{pA}

\DeclareMathOperator{\Homeo}{Homeo}

\newcommand{\PC}{{\text{PC}}}

\newcommand{\thr}[1]{\left<#1\right>}
\newcommand{\thrn}[1]{\thr{#1_0, #1_1, #1_2, \dots}}

\newcommand{\ve}{\varepsilon}
\newcommand{\ssp}{{\{0,1\}^\N}}
\newcommand{\up}[1]{^{(#1)}}
\newcommand{\upi}{\up{i}}
\newcommand{\re}{\ha}
\newcommand{\fib}[1]{\left[#1\right]}

\newcommand{\ingam}{{\mathring\gamma}}
\newcommand{\toplus}{\overset{\scriptscriptstyle{+}}{\to}}
\newcommand{\tominus}{\overset{\scriptscriptstyle{-}}{\to}}
\newcommand{\ec}[1]{\left<#1\right>}

\begin{document}

\begin{abstract} 
    We exhibit a continuously varying family $F_\lambda$ of homeomorphisms of the
    sphere~$S^2$, for which each~$F_\lambda$ is a measurable pseudo-Anosov map.
    Measurable pseudo-Anosov maps are generalizations of Thurston's pseudo-Anosov
    maps, and also of the generalized pseudo-Anosov maps of~\cite{gpa}. They have a
    transverse pair of invariant full measure turbulations, consisting of
    streamlines which are dense injectively immersed lines: these turbulations are
    equipped with measures which are expanded and contracted uniformly by the
    homeomorphism. The turbulations need not have a good product structure
    anywhere, but have some local structure imposed by the existence of tartans:
    bundles of unstable and stable streamline segments which intersect regularly,
    and on whose intersections the product of the measures on the turbulations
    agrees with the ambient measure.

    Each map $F_\lambda$ is semi-conjugate to the inverse limit of the core tent map
    with slope $\lambda$: it is topologically transitive, ergodic with respect to a
    background Oxtoby-Ulam measure, has dense periodic points, and has topological
    entropy \mbox{$h(F_\lambda) = \log \lambda$} (so that no two $F_\lambda$ are
    topologically conjugate). For a full measure, dense $G_\delta$ set of parameters,
    $F_\lambda$ is a measurable pseudo-Anosov map but not a generalized pseudo-Anosov
    map, and its turbulations are nowhere locally regular.

\end{abstract}

\maketitle

\section{Introduction}

Since their introduction by Thurston in the 1970s, pseudo-Anosov maps have played a
central r\^ole in low-dimensional geometry and topology, as well as in
low-dimensional dynamics. They first appeared in Thurston's Classification Theorem
for isotopy classes of surface homeomorphisms, and are also fundamental for the
statement and the proof of Thurston's Hyperbolization Theorem for fibered
3-manifolds. In both of these discussions, there is a finiteness hypothesis: the
surfaces involved are of finite topological type (compact surfaces from which
finitely many points have been removed, or marked).

In surface dynamics, the way that pseudo-Anosov maps most frequently appear is as
follows: given a homeomorphism $f$ of a surface $S$, find a periodic orbit $O$ and
apply the Classification Theorem to $f$ on $S\setminus O$; if this isotopy class is
pseudo-Anosov then one can conclude, amongst other things, that $f$ has infinitely
many other periodic orbits, of infinitely many distinct periods, and has positive
topological entropy. If we adopt the point of view that the periodic orbit $O$
consists of marked points instead of having been removed (in which case only
isotopies relative to $O$ are allowed), we can consider the collection $\pA(S)$ of
all pseudo-Anosov homeomorphisms on a given surface $S$, relative to all possible
finite invariant sets and up to topological changes of coordinates. This set $\pA(S)$
is countable, since there is at most one pseudo-Anosov map in any marked isotopy
class up to change of coordinates.

At this point some questions arise naturally. What kind of set is $\pA(S)$? How does
it sit inside the set~$\Homeo(S)$ of all homeomorphisms of $S$? Can we complete it
somehow, and what sort of maps would be necessary to do this? To the best of our
knowledge the nature of~$\pA(S)$ as a subset of~$\Homeo(S)$ remains a puzzle, and
none of these questions has a satisfactory answer.

Our approach to this problem is to restrict attention to a subset of $\pA(S^2)$ which
naturally lies within an arc in $\Homeo(S^2)$ --- a parametrized family
$\{F_\lambda\}$ of sphere homeomorphisms --- and to understand what sort of
homeomorphisms the~$F_\lambda$ are: this is expected to illuminate the general case.
We show that the~$F_\lambda$ have enough structure to make them worthy of being
called \emph{measurable pseudo-Anosov maps}. These are analogs of pseudo-Anosov maps
in which the transverse pair of invariant measured foliations is replaced with a
transverse pair of invariant measured \emph{turbulations}: full measure subsets of
the surface decomposed into a disjoint union of \emph{streamlines}, which are
injectively immersed lines equipped with measures. As in the pseudo-Anosov case,
there is a dilatation~$\lambda>1$ such that one of the turbulations is uniformly
expanded by a factor~$\lambda$ and the other is uniformly contracted by a
factor~$1/\lambda$. In contrast to the pseudo-Anosov case, they need not have a good
product structure anywhere. Instead, some structure is imposed by the existence of
\emph{tartans}: bundles of arcs of stable and unstable streamlines which intersect
regularly, and in the context of which the measures on the streamlines are holonomy
invariant and have a product agreeing with the ambient measure on the sphere.

By construction the family contains countably many pseudo-Anosov maps. It also
contains uncountably many \emph{generalized pseudo-Anosov maps}, as introduced
in~\cite{gpa}: these are defined in the same way as pseudo-Anosov maps, except that
their invariant foliations are permitted to have a countably infinite number of
singularities, provided that they only accumulate on a finite set
(Definitions~\ref{defn:pAgpA}).

Our main focus here is on defining measurable pseudo-Anosov maps, and showing that
they are abundant in the sense that they form a continuously varying family. In
another paper~\cite{mpa-dynamics}, we show that the definition has dynamical weight
by proving that any measurable pseudo-Anosov map is Devaney chaotic (it has dense
periodic points and is topologically transitive); and, under an additional
hypothesis, is ergodic. (For the particular family discussed in this paper, these
dynamical properties follow more directly from the construction.)

The definition of measurable pseudo-Anosov maps is involved, but has three
fundamental properties: it is analogous to Thurston's definition of pseudo-Anosov
maps, and generalizes it; it is strong enough to have dynamical weight as explained
above; and it is broad enough to contain the maps of the family $F_\lambda$ (which
was originally constructed with different goals in mind).  We were also guided by the
higher-dimensional holomorphic analogs of~\cite{BLS,LM, Su} --- in particular, the
term `turbulation' is borrowed from comments in~\cite{LM}.

The following statement summarizes the main results of the paper (the fourth item depends in part on results of~\cite{BrM}).

\begin{unnumthm} 
    There exists a continuously varying family of sphere homeomorphisms
    $F_\lambda:S^2\to S^2$ for $\lambda\in (\sqrt{2},2]$ such that
    \begin{itemize}

        \item $F_\lambda$ is a measurable pseudo-Anosov map with dilatation
        $\lambda$.

        \item For a {countable discrete} set of parameters,
          $F_\lambda$ is {pseudo-Anosov}.

        \item For an {uncountable dense} set of parameters,
          $F_\lambda$ is generalized pseudo-Anosov.

        \item For a {dense $G_\delta$, full measure} set of
          parameters, $F_\lambda$ is {not generalized pseudo-Anosov},
          its invariant turbulations are nowhere locally regular, and
          the complement of its unstable turbulation contains a
          dense $G_\delta$ set. 

        \item Each $F_\lambda$ is topologically transitive, ergodic
          with respect to a background Oxtoby-Ulam measure, has dense periodic
          points, and has topological entropy $h(F_\lambda) =
          \log \lambda$ (so that no two $F_\lambda$ are topologically
          conjugate).
    \end{itemize}
\end{unnumthm}

In the remainder of the introduction we explain these statements in a bit more detail
and put them in a broader context.

\subsection*{Connections with surface dynamics}

A result of Bonatti and Jeandenans (see the appendix of~\cite{BJ}) relates
pseudo-Anosov maps to differentiable dynamics:  there is a natural quotient of a
Smale (Axiom A + strong transversality) surface diffeomorphism without impasses
(bigons bounded by a stable segment and an unstable segment) which is pseudo-Anosov.
The semiconjugacies between the Smale diffeomorphisms and their pseudo-Anosov
quotients are mild in the sense that their fibers carry no entropy. This result has
recently been generalized by Mello~\cite{dM}, who shows that, when impasses are
allowed, the quotient is a generalized pseudo-Anosov map. 

These results indicate that pseudo-Anosov and generalized pseudo-Anosov maps can be
viewed as linear models, with constant expansion and contraction, for Smale surface
diffeomorphisms. This places pseudo-Anosov and generalized pseudo-Anosov maps in a
broader context in the study of surface dynamics. However, the results only apply to
a collection of hyperbolic maps which is countable up to changes of coordinates, so
that this semi-conjugacy from the differentiable world to the pseudo-Anosov world
cannot apply to all maps in a parametrized family, such as the H\'enon family. The
existence of a parametrized family of measurable pseudo-Anosovs makes them candidates
to fill this gap: we conjecture that the $0$-entropy quotient~\cite{zeroent} of any
surface diffeomorphism with a single homoclinic class is a measurable pseudo-Anosov
map. Note also that the \emph{tartans} mentioned above, which are the key part of the
definition of measurable pseudo-Anosov maps, are reminiscent of \emph{Pesin blocks}
from the theory of non-uniformly hyperbolic surface diffeomorphisms.

\subsection*{A path in \texorpdfstring{$\Homeo(S^2)$}{Homeo(S2)}}

The family~$\{F_\lambda\}$ was originally constructed in~\cite{prime}. Each
homeomorphism~$F_\lambda$ is obtained as a quotient of the natural extension of the
core tent map~$f_\lambda$ of slope~$\lambda$: the quotient maps are mild (all but at
most one fiber contains at most~3 points) and are described explicitly. The family
intersects $\pA(S^2)$ at a particular collection of pseudo-Anosov maps which are
isotopic to Smale's horseshoe map relative to single periodic orbits: those which are
\emph{quasi-one-dimensional}, in the sense that their invariant train tracks are
intervals, and the train track maps are core tent maps. These were identified
in~\cite{HS}: they are parametrized by an invariant called \emph{height}, which is a
rational in~$(0,1/2)$.

In~\cite{gpa} it was shown that to \emph{every} horseshoe periodic orbit satisfying
an irreducibility condition there is associated a generalized pseudo-Anosov map of
the sphere whose dynamics mimics that of a core tent map. The family~$\{F_\lambda\}$
is the closure of this countable collection of \emph{unimodal generalized
pseudo-Anosov maps} --- thus our present results greatly extend the main application
of~\cite{origami}, which was to obtain a limit of a \emph{single} sequence of
unimodal generalized pseudo-Anosov maps. One of the main subsidiary results of this
paper (Theorem~\ref{thm:irrat-gpa}) is that there are uncountably many other
parameters~$\lambda$ for which $F_\lambda$ is a generalized pseudo-Anosov map.

\subsection*{Connections with one-dimensional dynamics and inverse limits}

To explain the theorem statement above in more detail, we first provide some
preliminary definitions and notation. These preliminaries are expanded on in
Sections~\ref{sec:tent} (tent maps),~\ref{sec:il} (inverse limits),
and~\ref{sec:measures} (invariant measures).

The \emph{core tent map} $f_\lambda\colon I_\lambda\to I_\lambda$ with
slope $\lambda\in(\sqrt{2}, 2]$ is the restriction of the full tent
map $T_\lambda\colon[0,1]\to[0,1]$ defined by $T_\lambda(x) =
\min(\lambda x, \lambda(1-x))$ to the interval $I_\lambda =
\left[T_\lambda^2(1/2), T_\lambda(1/2)\right]$. It has \emph{turning}
or \emph{critical point} $c=1/2$. For the vast majority of the paper we will only
be concerned with single tent maps, so, except where absolutely necessary, the
dependence of objects on the parameter~$\lambda$ will be suppressed. The natural
extension of the tent map, acting on its inverse limit, is denoted
$\hf\colon\hI\to\hI$. We write $\pi_0\colon\hI\to I$ for the projection of the
inverse limit onto its base using the $0^\text{th}$ coordinate, and $\fib{x} :=
\pi_0^{-1}(x)$ for the $\pi_0$-fiber of~$\hI$ above a point $x\in I$.

As discussed above, it is shown in~\cite{prime} that there is a map $g\colon \hI\to
S^2$ which semi-conjugates $\hf\colon\hI\to\hI$ to a sphere homeomorphism $F\colon
S^2\to S^2$ (and, momentarily restoring the dependence on the parameter, the sphere
homeomorphisms $F_\lambda$ vary continuously with~$\lambda$). The tent map~$f$ has a
unique ergodic invariant measure $\mu$ of maximal entropy, which is absolutely
continuous with respect to Lebesgue measure. This measure gives rise to an ergodic
$\hf$-invariant measure on the inverse limit, and in turn to a fully supported
ergodic $F$-invariant one (non-atomic and positive on non-empty open sets), also
denoted~$\mu$, on $S^2$. When we speak of full measure on~$S^2$, it is always with
respect to this measure (which depends, of course, on the parameter~$\lambda$).

The behavior of the sphere homeomorphism~$F$ depends on two pieces of
information about the tent map~$f$: its \emph{height}, and the nature
of the orbit of its turning point.

\begin{itemize}

    \item The height can be defined dynamically either as the rotation number of the
      outside map, a circle endomorphism associated with~$f$
      (Section~\ref{sec:outside}); or as the prime ends rotation number of the basin
      of infinity when $\hI$ is embedded in the sphere as an attractor of a sphere
      homeomorphism which extends $\hf$~\cite{prime}. It can be determined
      algorithmically from the kneading sequence of~$f$ when it is rational, and
      approximated arbitrarily closely when it is irrational~\cite{HS}. It is a
      decreasing function of the parameter~$\lambda$, with an interval of parameters
      corresponding to each rational height, and a single parameter corresponding to
      each irrational.

    \item As far as the orbit of the turning point is concerned,
      the main distinction is between the post-critically finite
      (where the orbit of~$c$ is periodic or pre-periodic) and the
      post-critically infinite cases.

\end{itemize}

For a typical parameter~$\lambda$, the tent map~$f$ is post-critically infinite
with rational height. In this case, the sphere homeomorphism is measurable
pseudo-Anosov, but is neither pseudo-Anosov nor generalized pseudo-Anosov. It is
important to understand that, even though the turbulations are of full measure and
transverse to each other, their local structure can be very complicated: for
example, when the critical orbit is dense, the streamlines which locally give a
product structure are never locally of full measure, so that there are no local
foliated charts anywhere on the sphere. Informally, the orbit of the critical
fiber~$\fib{c}$ carries `kinks' which become dense and preclude the existence of
local product structures anywhere on the sphere.

The construction of the measurable pseudo-Anosov map proceeds as follows. The
unstable streamlines correspond to the full measure collection of path components in
the inverse limit which are dense injectively immersed lines (see~\cite{typical}).
The measure is locally the pull back of Lebesgue measure on the interval under the
projection~$\pi_0$. The stable turbulation is built using the fibers $\pi_0^{-1}(x)$
and measures~$\alpha_x$ supported on these fibers, which were constructed
in~\cite{typical} using the density of the tent map absolutely continuous invariant
measure of maximal entropy. Each fiber is a Cantor set and, with countably many
exceptions, projects to an arc in the sphere carrying the push forward of $\alpha_x$.
The ends of all but countably many of these arcs are connected in pairs by the
semi-conjugacy~$g$. The countable collection of non-conforming fibers/arcs are
omitted from the stable turbulation.

While we obtain a nice global structure on a set of full measure, the inverse limits
in this case are known to be exquisitely complicated topological objects. For
example, when the parameter $\lambda$ is such that critical orbit is dense (these
parameters are contained in the rational height, post-critically infinite case),
theorems of Bruin and of Raines imply that the inverse limit $\hI_\lambda$ is nowhere
locally the product of a Cantor set and an interval~\cite{bruin,raines}. Even more
striking is that, for a perhaps smaller, but still dense $G_\delta$ set of
parameters, Barge, Brucks and Diamond~\cite{BBD} (see also~\cite{AABC}) show that the
inverse limit has a strong self-similarity: every open subset of~$\hI_\lambda$
contains a homeomorphic copy of $\hI_t$ for every $t\in(\sqrt{2}, 2]$.

In the other cases of the height and critical orbit data, the
construction follows the same general outline, but there is more
regularity and the sphere homeomorphism~$F_\lambda$ is generalized
pseudo-Anosov.

\begin{itemize}
    \item When the height is irrational (which implies that the tent
      map is post-critically infinite); or when~$\lambda$ is an endpoint of a
      rational height interval (which implies that the tent map is post-critically
      finite), there is a single bad point (accumulation of
      singularities)~$\infty$, and a single orbit of one-pronged singularities
      which is homoclinic to it: all other points are regular points of the
      invariant foliations.

    \item When~$\lambda$ is in the interior of a rational height
      interval and the tent map is post-critically finite, then, with
      the exception of a single parameter~$\lambda$ in each rational
      height interval, there is a fixed bad point~$\infty$ and a
      further periodic orbit~$Q$ of bad points having the same period
      as the post-critical periodic orbit. Heteroclinic orbits of
      connected one-pronged and three-pronged singularities emerge
      from~$\infty$ and converge to~$Q$.

    For the exceptional parameter~$\lambda$ in each rational height
    interval, the sphere homeomorphism~$F_\lambda$ is pseudo-Anosov.
\end{itemize}

\subsection*{Open questions}

Our results relate to a specific family of measurable pseudo-Anosov maps, but they
raise a number of questions of a more general nature. In~\cite{gpa} it was shown
that the generalized pseudo-Anosov maps considered here are quasi-conformal with
respect to a complex structure on the sphere, and that the invariant foliations are
trajectories of an integrable quadratic differential. Does any of this structure
exist for the more general measurable pseudo-Anosov maps and, if it does, is the
quasi-conformal distortion of $F_\lambda$ equal to $\lambda$? A related question is
whether or not measurable pseudo-Anosov maps are smoothable: topologically
conjugate to, say, a $C^k$-diffeomorphism, for some $k\ge 1$? Another natural
question concerns their prevalence amongst surface homeomorphisms. In our setting
the measurable pseudo-Anosov maps occur in a continuous one-parameter family
containing a countable family of pseudo-Anosov maps. How common is this: for
example, given two isotopic pseudo-Anosov maps (necessarily relative to finite
invariant sets), is there a continuous generic one-parameter family connecting them
which consists solely of measurable pseudo-Anosov maps? More generally, is the
collection of measurable pseudo-Anosov maps in some sense the completion of the
collection of generalized pseudo-Anosov maps?

\subsection*{Structure of the paper}

In Section~\ref{sec:gmpa-intro} we define measurable pseudo-Anosov maps and prove
some basic properties. Because the invariant turbulations of a measurable
pseudo-Anosov map need not have a local product structure anywhere, their
associated measures are defined on streamlines rather than on transverse arcs. For
consistency we take the same approach with the measures on (generalized)
pseudo-Anosov foliations, and these alternative definitions are also given in
Section~\ref{sec:gmpa-intro}. Section~\ref{sec:background} is a review of some
necessary background material on tent maps, their invariant measures, and their
inverse limits. Section~\ref{sec:extreme-outside} contains the main theoretical
content of the paper: we discuss the structure of $\pi_0$-fibers of inverse limits
of tent maps, identify their \emph{extreme} elements, and model the action of the
natural extension on these extreme elements by means of a circle map called the
\emph{outside} map. In Sections~\ref{sec:rat_PCI} and~\ref{sec:irrat} we apply this
first to the post-critically infinite rational height case and then to the
irrational height case, proving that the associated sphere homeomorphisms are
measurable pseudo-Anosov maps and generalized pseudo-Anosov maps respectively. The
paper ends with a brief summary of how the post-critically finite rational case
could be treated using the same approach: we omit the details because the fact that
the sphere homeomorphisms associated to these tent maps are generalized
pseudo-Anosov maps was proved, using quite different techniques, in~\cite{prime}.

\section{Generalized and measurable pseudo-Anosov maps}
    \label{sec:gmpa-intro}

We will show that the sphere homeomorphisms $F_\lambda$ are either examples of
Thurston's pseudo-Anosov maps~\cite{Thurston}, or one of two successive
generalizations of them: the \emph{generalized pseudo Anosov maps}
of~\cite{gpa}, which are permitted to have infinitely many pronged singularities,
provided that these only accumulate at finitely many points; and \emph{measurable
pseudo-Anosov maps}, which have invariant measured \emph{turbulations},
rather than foliations. A measured turbulation (Definitions~\ref{defn:turbulation})
is a partition of a full measure subset of the sphere into \emph{streamlines}:
immersed lines equipped with measures, which play the r\^ole of the leaves of a
foliation.

Because the invariant turbulations of a measurable pseudo-Anosov map may not have
a local product structure in any open set, it is necessary to define the associated
measures on the streamlines, rather than on arcs transverse to them. For this
reason, and for the sake of consistency, we take the same approach with the
measures for pseudo-Anosov and generalized pseudo-Anosov foliations.

We start by defining local models for regular and singular points of pseudo-Anosov
and generalized pseudo-Anosov maps. These models come equipped with measures not
only on the domain, but also on each stable and unstable leaf segment. Note that
the invariant measured foliations will be defined globally, with the local models
serving only to ensure that they have the desired local structure: thus no
transition conditions are required where the local models overlap.

Because measurable pseudo-Anosov maps cannot be defined using local models for
their invariant turbulations, properties such as holonomy invariance of the
measures and their compatibility with the ambient measure are expressed instead in
terms of \emph{tartans} (Definitions~\ref{defn:tartan}): unions of stable and
unstable arcs which intersect regularly, each stable arc intersecting each unstable
arc exactly once.

\begin{rem}
    When we say that a subset of a topological space is measurable, it will always
    mean with respect to the Borel $\sigma$-algebra.
\end{rem}

\medskip\medskip 

\subsection{Pseudo-Anosov and Generalized pseudo-Anosov maps}

\begin{defn}[regular model] 
    \label{defn:regular_model}
    Let $I$ and $J$ be intervals in~$\R$. We denote by $R_{I, J}$ the rectangle
    $I\times J$ in~$\R^2$, equipped with two-dimensional Lebesgue measure~$M$, and
    refer to it as a \emph{regular model}. Each horizontal $I\times\{\beta\}$ in
    $R_{I, J}$ is called an \emph{unstable leaf segment}, and is equipped with
    one-dimensional Lebesgue measure~$m_u$; similarly, each vertical
    $\{\alpha\}\times J$, with one-dimensional Lebesgue measure~$m_s$, is called a
    \emph{stable leaf segment}.
\end{defn}

\begin{defn}[$p$-pronged model, see Figure~\ref{fig:1_3-prong}]
    \label{defn:pronged_model}
     Given $a,b>0$ and $p\in\Z_+\setminus\{2\}$, we denote by $Q_{p, a, b}$ the
    space obtained from~$p$ copies $R_0,\dots,R_{p-1}$ of $(-a, a) \times [0,b)$ by
    identifying $(\alpha, 0)\in R_i$ with $(-\alpha, 0)\in R_{i+1\bmod{p}}$ for
    each $\alpha\in[0,a)$ and $i\in\{0,\dots,p-1\}$, and refer to it as a
    \emph{$p$-pronged model}. We equip $Q_{p, a, b}$ with the measure~$M$ induced
    by Lebesgue measure on each $R_i$. We
    say that the point obtained by identification of $(0, 0)\in R_i$ for all~$i$ is
    the \emph{singular point} of the model.

    An \emph{unstable leaf segment} in~$Q_{p, a, b}$ is either a segment $(-a,
    a)\times\{\beta\}$ in some~$R_i$ for $\beta\in(0, b)$; or the union of the
    segments $[0, a)\times\{0\}$ in each~$R_i$ after identifications. A
    \emph{stable leaf segment} in~$Q_{p, a, b}$ is either the union (after
    identifications) of $\{\alpha\}\times[0, b)$ in~$R_i$ and
    $\{-\alpha\}\times[0, b)$ in $R_{i+1\bmod{p}}$ for some $\alpha\in(0, a)$;
    or the union of $\{0\}\times [0, b)$ in $R_i$ over all~$i$. Unstable and
    stable leaf segments are equipped with the measures $m_u$ and~$m_s$ 
    induced by disintegration of~$M$.
\end{defn}

\begin{figure}[htbp]
        \begin{center}
            \includegraphics[width=0.75\textwidth]{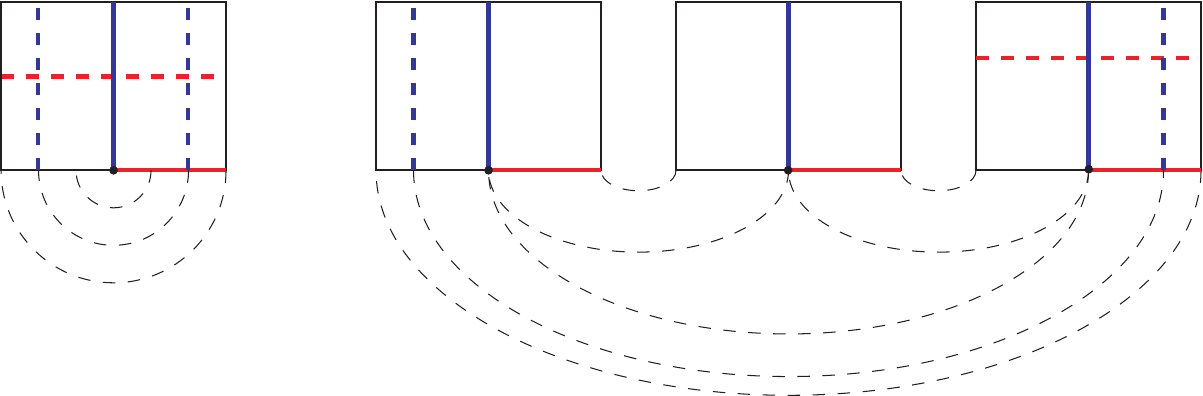}
        \end{center}
    \caption{A 1-pronged model and a 3-pronged model, each with two unstable
    (red) and two stable (blue) leaf segments. The dashed leaf segments are
    arcs, while the solid ones, which pass through the singular point of the
    model, are 1-ods and 3-ods. The curved dashed lines denote identifications}
    \label{fig:1_3-prong}
    \end{figure}

\begin{defn}[OU] \label{defn:OU}
    We say that a measure~$\mu$ on a topological space~$X$ is \emph{OU} (for
    Oxtoby-Ulam) if it is Borel, non-atomic, and positive on open subsets of~$X$.
    If~$X$ is a compact manifold, then it is further required that $\mu(\partial
    X)=0$.
\end{defn}

Throughout the remainder of this section, we let~$\Sigma$ be a surface, and $\mu$
be an OU probability measure on~$\Sigma$.

Note that the following definition of a measured foliation on~$\Sigma$ makes little
sense in isolation (since there is no requirement for measures on distinct leaves
to be compatible), and will only take on a proper meaning when constraints on the
local structure are introduced in Definitions~\ref{defn:pa-gpa-fols}. Note also
that it is not the standard definition: as explained above, we put measures on the
leaves rather than on arcs transverse to them, for consistency with the definition
of measurable pseudo-Anosov turbulations.

\begin{defn}[measured foliation, image foliation] \label{defn:measured_fol}
    A \emph{measured foliation} $(\cF, \nu)$ on~$\Sigma$ is a partition~$\cF$
    of~$\Sigma$ into subsets called \emph{leaves}, together with an OU measure
    $\nu_\ell$ on each non-singleton leaf~$\ell$. If $F\colon
    \Sigma\to\Sigma$ is a homeomorphism, we write $F(\cF, \nu)$ for the measured
    foliation $(\cF', \nu')$ whose leaves are $\{F(\ell)\,:\,\ell\in\cF\}$, with
    measures $\nu'_{F(\ell)} = F_*(\nu_\ell)$ on non-singleton leaves.
\end{defn}

    Note that, while~$\mu$ is a probability measure, there is no requirement
    for the leaf measures~$\nu_\ell$ to be finite, and they will generally not
    be.

\begin{defns}[regular and $p$-pronged charts] \label{defn:charts}
    Let $(\cF^s, \nu^s)$ and $(\cF^u, \nu^u)$ be measured foliations
    on~$\Sigma$, and let $x\in \Sigma$. We say that the foliations have a
    \emph{regular chart} at~$x$ if there is a neighborhood~$N$ of~$x$, a
    regular model~$R=R_{I, J}$, and a measure-preserving homeomorphism
    $\Phi\colon N\to R$, such that for each $\ell\in\cF^s$ (respectively
    $\ell\in\cF^u$), and each path component~$L$ of $\ell\cap N$, $\Phi|_L$ is
    a measure-preserving homeomorphism onto a stable (respectively unstable)
    leaf segment of~$R$. (To be clear, that $\Phi$ is measure-preserving means
    that $\Phi_*(\mu) = M$; and that $\Phi|_L$ is measure-preserving means that
    $(\Phi|_L)_*(\nu_\ell^{s/u}) = m_{s/u}$.)

    We say that the foliations have a \emph{$p$-pronged chart} at~$x$ if there
    is a neighborhood~$N$ of~$x$, a $p$-pronged model~$Q=Q_{p, a, b}$, and a
    measure-preserving homeomorphism \mbox{$\Phi\colon N\to Q$}, with $\Phi(x)$
    the singular point of~$Q$, such that for each $\ell\in\cF^s$ (respectively
    $\ell\in\cF^u$), and each path component~$L$ of $\ell\cap N$, $\Phi|_L$ is
    a measure-preserving homeomorphism onto a stable (respectively unstable)
    leaf segment of~$Q$.
\end{defns}

\begin{defns}[pseudo-Anosov and generalized pseudo-Anosov foliations]
    \label{defn:pa-gpa-fols} Let $(\cF^s, \nu^s)$ and $(\cF^u, \nu^u)$ be measured
    foliations on~$\Sigma$. We say that they are \emph{pseudo-Anosov foliations} if
    they have regular charts at all but finitely many points of~$\Sigma$, and at
    each other point they have a pronged chart. We say that they are
    \emph{generalized pseudo-Anosov foliations} if there is a countable $Y\subset
    \Sigma$ and a finite $Z\subset
    \Sigma\setminus Y$ such that the foliations have regular charts at every
    point of $\Sigma\setminus(Y\cup Z)$; pronged charts at every point of~$Y$;
    and each point of~$Z$ is a singleton leaf of both foliations.
\end{defns}

\begin{defns}[pseudo-Anosov and generalized pseudo-Anosov maps] \label{defn:pAgpA}
    Let $F\colon \Sigma\to \Sigma$ be a homeomorphism. We say that~$F$ is
    \emph{(generalized) pseudo-Anosov} if there are (generalized) pseudo-Anosov
    foliations $(\cF^s, \nu^s)$ and $(\cF^u, \nu^u)$ on~$\Sigma$ and a number
    $\lambda>1$ such that $F(\cF^s, \nu^s) = (\cF^s, \lambda\nu^s)$ and $F(\cF^u,
    \nu^u) = (\cF^u, \lambda^{-1}\nu^u)$. (Note that the $\lambda$ and the
    $\lambda^{-1}$ are the other way round in the standard definitions, in
    which the measures are on arcs transverse to the foliations rather than on
    the leaves themselves.)
\end{defns}

\begin{rem} \label{rem:zero_measure_leaf}
    Let~$\ell$ be a leaf of one of the invariant foliations of a pseudo-Anosov or
    generalized pseudo-Anosov map. Since~$\ell$ is either a singleton, or is a
    union of countably many leaf segments of regular or $p$-pronged charts, it
    follows from Definitions~\ref{defn:charts} that~\mbox{$\mu(\ell)=0$.}
\end{rem}

\subsection{Measured turbulations and tartans}\mbox{}

\begin{defn}[Immersed  line]
    \label{defn:immersed}
    An \emph{immersed  line} in~$\Sigma$ is a continuous injective image
    of~$\R$. 
\end{defn}

\begin{defns}[Measured turbulations, streamlines, dense streamlines, stream measure, transversality]
    \label{defn:turbulation}
    A \emph{measured turbulation $(\cT, \nu)$} on~$\Sigma$ is a partition of a full
    $\mu$-measure Borel subset of~$\Sigma$ into immersed lines called
    \emph{streamlines}, together with an OU measure $\nu_\ell$ on each
    streamline~$\ell$, which assigns finite measure to any closed arc in~$\ell$. We
    refer to the measures $\nu_\ell$ as \emph{stream measures} to distinguish them
    from the ambient measure~$\mu$ on~$\Sigma$. We say that the measured
    turbulation \emph{has dense streamlines} if every streamline is dense
    in~$\Sigma$.

    Two such turbulations are \emph{transverse} if they are topologically
    transverse on a full $\mu$-measure subset of~$\Sigma$ (to be precise, there
    is a full measure subset of~$\Sigma$, every point~$x$ of which is contained
    in streamlines~$\ell$ and~$\ell'$ of the two turbulations, which intersect 
    transversely at~$x$).
\end{defns}

\begin{defns}[Stream arc, measure of stream arc] 
    Let $(\cT, \nu)$ be a measured turbulation on~$\Sigma$. Given distinct points
    $x$ and $y$ of a streamline~$\ell$ of~$\cT$, we write $[x,y]_\ell$, or just
    $[x,y]$ if the streamline is irrelevant or clear from the context, for the
    (unoriented) closed arc in~$\ell$ with endpoints~$x$ and~$y$; and $[x,y)_\ell$
    and $(x,y)_\ell$ for the arcs obtained by omitting one or both endpoints of
    $[x,y]_\ell$. We refer to these as \emph{stream arcs} of the turbulation. The
    \emph{measure} of a stream arc is its stream measure.
\end{defns}

We will impose a regularity condition on turbulations which requires that stream
arcs of small measure are small. Note that the following definition is independent
of the choice of metric on~$\Sigma$: since $\Sigma$ is compact, it is equivalent to
the topological condition that for every neighborhood~$N$ of the diagonal~$\{(x,
x)\,:\,x\in\Sigma\}$ in~$\Sigma\times\Sigma$, there is some $\delta>0$ such that if
$[x, y]$ is a stream arc with $\nu([x, y])<\delta$, then $(x,y)\in N$. However the
metric formulation is usually easier to apply directly.

\begin{defn}[Tame turbulation] \label{defn:tame}
    Let~$d$ be any metric on~$\Sigma$ compatible with its topology. A measured turbulation $(\cT, \nu)$ on~$\Sigma$ is \emph{tame} if for
    every~$\epsilon>0$ there is some $\delta>0$ such that if $[x, y]$ is a stream
    arc with $\nu([x,y])<\delta$, then $d(x,y)<\epsilon$.
\end{defn}

\begin{defn}[Image turbulation]
    If $F\colon \Sigma\to \Sigma$ is a $\mu$-preserving homeomorphism, we write
    $F(\cT,\nu)$ for the measured turbulation whose streamlines are
    $\{F(\ell)\,:\,\ell\in\cT\}$, with measures $\nu_{F(\ell)} =
    F_*(\nu_\ell)$.
\end{defn}

    Let $(\cT^s, \nu^s)$ and $(\cT^u, \nu^u)$ be a transverse pair of
    measured turbulations on~$\Sigma$. With a view to what is to come, we refer
    to them as \emph{stable} and \emph{unstable} turbulations, and similarly
    apply the adjectives stable and unstable to their streamlines, stream
    measures, stream arcs, etc.\ --- although dynamics will not enter the
    picture until Section~\ref{sec:mpa}.

\begin{defns}[Tartan $R$, fibers, $R^\pitchfork$, positive measure tartan]
    \label{defn:tartan} A \emph{tartan} $R=(R^s, R^u)$ consists of Borel
    subsets $R^s$ and $R^u$ of~$\Sigma$, which are disjoint unions of stable
    and unstable stream arcs respectively, having the following properties:
    \begin{enumerate}[(a)]
        \item Every arc of $R^s$ intersects every arc of $R^u$ exactly once,
        either transversely or at an endpoint of one or both arcs.

        \item There is a consistent orientation of the arcs of $R^s$ and $R^u$:
        that is, the arcs can be oriented so that every arc of $R^s$
        (respectively~$R^u$) crosses the arcs of $R^u$ (respectively~$R^s$) in the
        same order. 

        \item The measures of the arcs of $R^s$ and $R^u$ are bounded above.

        \item There is an open topological disk $U\subseteq \Sigma$ which
        contains $R^s\cup R^u$. 
    \end{enumerate}

    \medskip

    We refer to the arcs of $R^s$ and $R^u$ as the stable and unstable
    \emph{fibers} of~$R$, to distinguish them from other stream arcs. We write
    $R^\pitchfork$ for the set of \emph{tartan intersection points}: $R^\pitchfork
    =  R^s \cap R^u$. We say that~$R$ is a \emph{positive measure tartan} if
    $\mu(R^\pitchfork)>0$. 
\end{defns}

While $R^s$ and $R^u$ are defined as subsets of~$\Sigma$,  they decompose
uniquely as disjoint unions of stable/unstable stream arcs, so where necessary
or helpful can be viewed as the collection of their fibers.

\begin{notn}[$\fs(x)$, $\fu(x)$, $\fs\pitchfork\fu$, $E^\fs$, $E^\fu$, $\psi^{\fs,
    \fu}$, $\nu_\fs$ and $\nu_\fu$, product measure~$\nu_{\fs, \fu}$, holonomy maps
    $h_{\fs, \fs'}$, $h_{\fu, \fu'}$] \label{notn:tartan}
    \mbox{}

    Let~$R$ be a tartan; $x,y\in R^\pitchfork$; $\fs$ and $\fs'$ be stable
    fibers of~$R$; and $\fu$ and $\fu'$ be unstable fibers of~$R$. We write (see Figure~\ref{fig:defs}):
    \begin{itemize}
        \item $\nu^s_x$ and $\nu^u_x$ for the stream measures on the stable and
        unstable streamlines through~$x$. 

        \item $\fs(x)$ and $\fu(x)$ for the stable and unstable fibers of~$R$
        containing~$x$. 

        \item $\fs\pitchfork\fu = \fu\pitchfork\fs \in R^\pitchfork$ for the unique intersection point of $\fs$ and $\fu$. 

        \item $E^\fs=\fs\cap R^u = \fs\cap R^\pitchfork$, and $E^\fu=\fu\cap R^s
        = \fu\cap R^\pitchfork$;

        \item $\psi^{\fs, \fu}\colon E^\fs\times E^\fu \to R^\pitchfork$ for the map $(x,y)\mapsto \fu(x) \pitchfork \fs(y)$;

        \item $\nu_\fs$ and $\nu_\fu$ for the (restriction of the) stream measures
        on $E^\fs$ and $E^\fu$;

        \item $\nu_{\fs, \fu}$ for the product measure $\nu_\fs\times\nu_\fu$ on
        $E^\fs\times E^\fu$;

        \item $h_{\fs, \fs'}\colon E^\fs\to E^{\fs'}$ (respectively $h_{\fu,
        \fu'}\colon E^\fu\to E^{\fu'}$) for the holonomy map $x\mapsto \fu(x)
        \pitchfork \fs'$ (respectively $y\mapsto \fs(x) \pitchfork \fu'$).
    \end{itemize}

\end{notn}

\begin{figure}[htbp]
        \begin{center}
            \includegraphics[width=0.55\textwidth]{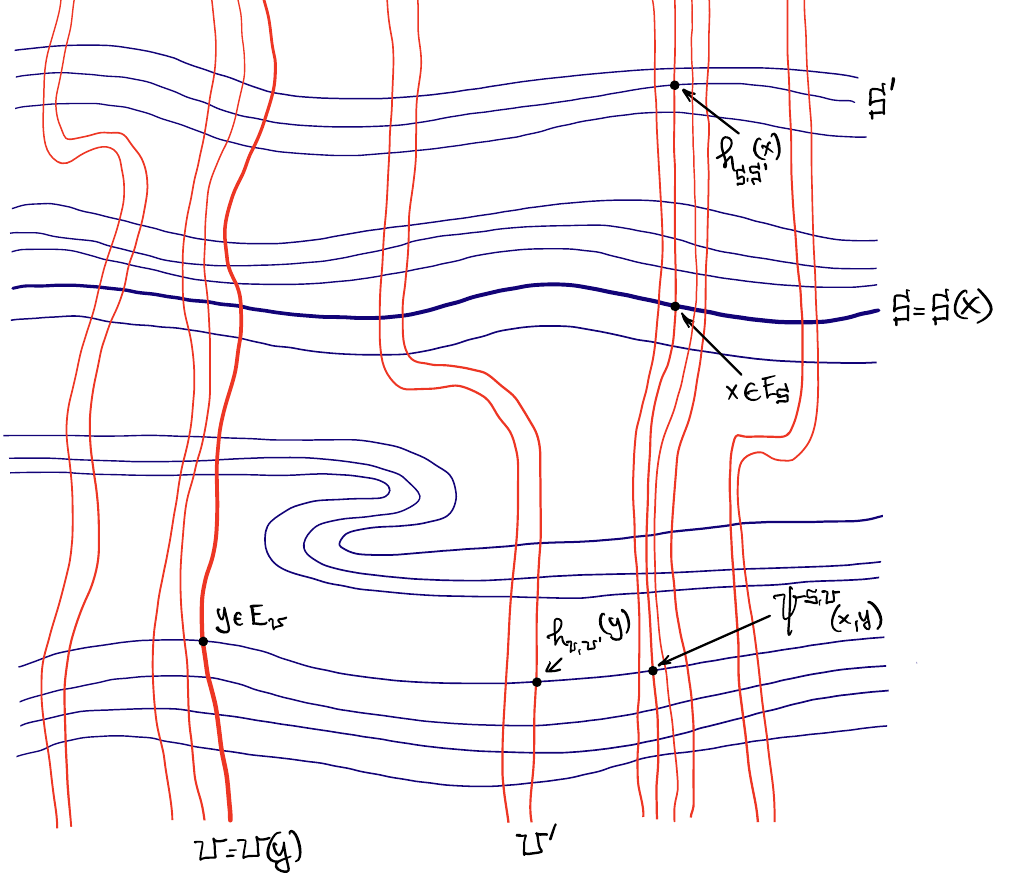}
        \end{center}
    \caption{Illustration of Notation~\ref{notn:tartan}}
    \label{fig:defs}
\end{figure}

The following is the key condition on tartans, which guarantees that the product
measure on the set of tartan intersection points induced by the measures on the
fibers agrees with the ambient measure.

\begin{defn}[Compatible tartan]
    We say that the tartan $R$ is \emph{compatible} (with the
    ambient measure~$\mu$) if, for all stable and unstable fibers~$\fs$
    and~$\fu$, the bijection $\psi^{\fs, \fu}\colon E^\fs\times E^\fu \to
    R^\pitchfork$ is bi-measurable, and $\psi^{\fs, \fu}_* \nu_{\fs, \fu} =
    \mu|_{R^\pitchfork}$.
\end{defn}

Important consequences of compatibility, proved in~\cite{mpa-dynamics}, include:
\begin{enumerate}[(a)]
    \item The stream measures are the disintegration of the ambient measure onto
    fibers: if~$A\subset R^\pitchfork$ is Borel, then
    \[
        \mu(A) = \int_{E^\fs} \nu_{\fu(x)}(A\cap \fu(x))\,d\nu_\fs(x) =
        \int_{E^\fu} \nu_{\fs(y)}(A \cap \fs(y))\,d\nu_\fu(y).
    \]
    \item The stream measures are holonomy invariant in tartans. Let~$\fs$ and
    $\fs'$ be stable fibers of~$R$, and $A \subset E^\fs$ be $\nu_\fs$-measurable.
    Then $A' = h_{\fs, \fs'}(A)$ is $\nu_{\fs'}$-measurable, and $\nu_{\fs'}(A') =
    \nu_\fs(A)$. The analogous statement holds for holonomies~$h_{\fu,
    \fu'}$.
\end{enumerate}

The final condition on the turbulations is that the compatible tartans see a full
measure subset of~$\Sigma$. 

\begin{defn}[Full turbulations]  \label{defn:full-turb}
    We say that the transverse pair $(\cT^s, \nu^s)$ and $(\cT^u, \nu^u)$ is
    \emph{full} if 
    \begin{enumerate}[(a)]
        \item There is a countable collection~$R_i$ of (positive measure) compatible tartans with $\mu\left(\bigcup_i R_i^\pitchfork\right)=1$; and

        \item for every non-empty open subset~$U$ of~$\Sigma$, there is a
            positive measure compatible tartan $R=(R^s, R^u)$ with $R^s\cup
            R^u\subset U$.
    \end{enumerate}
\end{defn}

Part~(b) of the definition is an additional regularity condition, which prevents
the fibers of tartans from behaving too wildly between intersection points.
In~\cite{mpa-dynamics}, a variety of conditions is presented, each of which,
together with part~(a) of the definition, implies part~(b). Here we only state the
condition which we will use later.

\begin{defn}[Regular tartan] \label{defn:regular-tartan}
    We say that a tartan~$R$ is \emph{regular} if for all $x\in R^\pitchfork$ and
    all neighborhoods~$U$ of~$x$, there is some~$\delta>0$ such that if $y\in
    \fs(x)\cap R^\pitchfork$ and $z\in \fu(x)\cap R^\pitchfork$ with
    $\nu_{\fs(x)}([x,y]_s)<\delta$ and $\nu_{\fu(x)}([x,z]_u)<\delta$, then the
    stream arcs $[x,y]_s$, $[y, \fs(z)\pitchfork\fu(y)]_u$,
    $[\fs(z)\pitchfork\fu(y), z]_s$, and $[z,x]_u$ are all contained in~$U$.
\end{defn}

\begin{lem}\label{lem:part_b_from_a}
    Suppose that part~(a) of Definition~\ref{defn:full-turb} is satisfied, and in
    addition each of the tartans~$R_i$ is regular.  
    Then part~(b) of the definition is also satisfied.
\end{lem}

\begin{rem}\label{lem:regular-propagate}
    If a tartan~$R$ is regular, then so is~$F(R)$.
\end{rem}

\subsection{Measurable pseudo-Anosov maps}
\label{sec:mpa}
\begin{defns}[measurable pseudo-Anosov turbulations, measurable pseudo-Anosov map,
dilatation]
    \label{defn:mpa} 
    A pair $(\cT^s, \nu^s), (\cT^u, \nu^u)$ of measured turbulations on~$\Sigma$
    are said to be \emph{measurable pseudo-Anosov turbulations} if they are
    transverse, tame, full, and have dense streamlines.

    A $\mu$-preserving homeomorphism $F\colon \Sigma\to \Sigma$ is \emph{measurable
    pseudo-Anosov} if there is a pair $(\cT^s, \nu^s), (\cT^u, \nu^u)$ of
    measurable pseudo-Anosov turbulations and a number $\lambda>1$, called the
    \emph{dilatation} of~$F$, such that $F(\cT^s, \nu^s) = (\cT^s,
    \lambda\nu^s)$ and $F(\cT^u, \nu^u) = (\cT^u, \lambda^{-1}\nu^u)$.
\end{defns}

In the hypotheses of the following lemma, note that we are not assuming that~$F$ is
a measurable pseudo-Anosov map: rather, the lemma will be used to prove that the
homeomorphisms~$F_\lambda$ are measurable pseudo-Anosov maps, by propagating a
single positive measure compatible tartan over a full measure subset of~$\Sigma$
using the ergodicity of~$F_\lambda$.

\begin{lem}\label{lem:image_tartan}
    Let $(\cT^s, \nu^s), (\cT^u, \nu^u)$ be a transverse pair of measured
    turbulations on~$\Sigma$, and suppose that $F\colon\Sigma\to\Sigma$ is
    $\mu$-preserving and satisfies $F(\cT^s, \nu^s) = (\cT^s, \lambda\nu^s)$
    and $F(\cT^u, \nu^u) = (\cT^u, \lambda^{-1}\nu^u)$ for some $\lambda>1$.

    If $R=(R^s, R^u)$ is a compatible tartan, then so is~$F(R)$ (the tartan whose
    fibers are the $F$-images of the fibers of~$R$).
\end{lem}

\begin{proof}
    It is immediate from Definitions~\ref{defn:tartan} that $F(R)$ is a tartan, and
    that $F(R)^\pitchfork = F(R^\pitchfork)$.

    For compatibility, let $\fs$ and $\fu$ be any stable and unstable fibers
    of~$R$, and write $\fs'$ and~$\fu'$ for the image fibers of~$F(R)$. Then there
    is a commutative diagram of bijections
    \[
        \begin{CD}
            E^\fs\times E^\fu @>\psi^{\fs, \fu}>> R^\pitchfork \\
            @VVF\times FV @VVFV \\
            E^\fsp\times E^\fup @>\psi^{\fsp, \fup}>> F(R)^\pitchfork.
        \end{CD}
    \]
    The bi-measurability of $\psi^{\fsp, \fup}$ follows from that of $\psi^{\fs,
    \fu}$ since $F$ and $F\times F$ are homeomorphisms. Since $F_*(\nu_\fs)=\lambda
    \nu_\fsp$ and $F_*(\nu_\fu)=\lambda^{-1} \nu_\fup$, we have $(F\times F)_*
    \nu_{\fs, \fu} = \nu_{\fs', \fu'}$, and hence $F_*(\psi_*^{\fs, \fu}\nu_{\fs,
    \fu}) = \psi_*^{\fsp, \fup}\nu_{\fsp, \fup}$. However $\psi_*^{\fs,
    \fu}\nu_{\fs, \fu} = \mu\vert_{R^\pitchfork}$ by compatibility of~$R$, so that
    $\psi_*^{\fsp, \fup}\nu_{\fsp, \fup} = \mu|_{F(R)^\pitchfork}$ by
    $F$-invariance of~$\mu$ as required.
\end{proof}

It is immediate from the definitions that any pseudo-Anosov map is a generalized
pseudo-Anosov map. We now show that any generalized pseudo-Anosov map
$F\colon\Sigma\to\Sigma$ with dense non-singular leaves is a measurable
pseudo-Anosov map.

\begin{lem}
    Let~$F\colon\Sigma\to\Sigma$ be generalized pseudo-Anosov, and suppose that the
    non-singular leaves of its invariant foliations are dense in~$\Sigma$. Then~$F$
    is also measurable pseudo-Anosov, with the streamlines of its (un)stable
    invariant turbulation given by the non-singular leaves~$\ell$ of the (un)stable
    foliation $\cF^{s/u}$, carrying the measures~$\nu_\ell^{s/u}$.
\end{lem}

\begin{proof}
    The stable and unstable streamlines are obtained from the stable and unstable
    foliations by omitting the finitely many singleton leaves at points of~$Z$, and
    the countably many leaves emanating from points of~$Y$. These omitted leaves
    have zero measure by Remark~\ref{rem:zero_measure_leaf}, so that $(\cT^s,
    \nu^s)$ and $(\cT^u, \nu^u)$ cover full measure in~$\Sigma$ and hence are
    turbulations. Clearly~$F$ preserves~$\mu$ and satisfies $F(\cT^s, \nu^s) =
    (\cT^s, \lambda\nu^s)$ and $F(\cT^u, \nu^u) = (\cT^u, \lambda^{-1}\nu^u)$. It
    remains therefore to show that the turbulations are measurable pseudo-Anosov
    turbulations (Definitions~\ref{defn:mpa}). The streamlines are dense by
    assumption, so it remains to show that the turbulations are transverse; that
    they are tame; and that they are full.

    By Definition~\ref{defn:charts}, there are regular chart domains $D_x$ about
    each point $x\in\Sigma\setminus(Y\cup Z)$ together with ambient and stream
    measure preserving homeomorphisms $\Phi_x\colon D_x\to M_x$ to regular models.
    In particular, the stable and unstable turbulations are transverse. For
    fullness, the stable and unstable leaf segments of $D_x$ which are not
    contained in singular leaves form a tartan $R_x = (R^s_x, R^u_x)$, whose
    compatibility with~$\mu$ is immediate using~$\Phi_x$. Moreover we have
    $\mu(R_x^\pitchfork) = \mu(D_x)>0$ by Remark~\ref{rem:zero_measure_leaf}. Since
    $\Sigma\setminus (Y\cup Z)$ is separable, it is covered by countably many chart
    domains $D_i=D_{x_i}$: the corresponding tartans~$R_i$ provide a countable
    collection of compatible tartans with $\mu\left(\bigcup_i
    R_i^\pitchfork\right)=1$. The second condition in the Definition of fullness is
    immediate, since every non-empty open subset~$U$ of $\Sigma$ contains a regular
    point~$x$, and the chart domain~$D_x$ can be restricted to a subset of~$U$.

    It remains to show that the turbulations are tame. Let~$d$ be a metric
    on~$\Sigma$ compatible with its topology, and suppose for a contradiction that
    there is some~\mbox{$\epsilon>0$} and a sequence $([x_i, y_i])$ of stream arcs
    --- of either turbulation --- with measures converging to zero, such that
    $d(x_i, y_i)\ge\epsilon$ for all~$i$. By taking a subsequence, we can assume
    that $x_i\to x^*\in\Sigma$ and $y_i\to y^*\in\Sigma$.

    Consider first the case $x^*\not\in Z$, so that there is a chart domain~$D$
    about~$x^*$ and an ambient and stream-measure preserving
    homeomorphism~$\Phi\colon D\to R$ or $\Phi\colon D\to Q$ to a regular or
    pronged model. Let $\eta>0$ be such that the leaf through~$x^*$ has segments of
    measure~$\eta$ on each side of~$x^*$ contained in~$D$ (or on each prong
    emanating from~$x^*$ if $x^*\in Y$). For~$i$ sufficiently large, we have that
    $x_i\in D$, and the streamline through~$x_i$ has measure at least $\eta/2$ on
    each side of~$x_i$ contained in~$D$. Since the measure of $[x_i, y_i]$
    converges to~$0$, we have $[x_i, y_i]\subset D$ for sufficiently large~$i$.
    Then $\Phi([x_i, y_i])$ is a horizontal or vertical segment with Lebesgue
    measure converging to zero, so that $\Phi(y_i)\to \Phi(x^*)$, and hence $y_i\to
    x^*$, contradicting $d(x_i, y_i)\ge\epsilon$.

    Therefore we must have $x^*\in Z$ and, analogously, $y^*\in Z$. We obtain
    the required contradiction by showing that there is a positive lower bound
    on the measure of any stream arc which connects sufficiently small
    neighborhoods of two distinct points of~$Z$.

    Pick $\xi>0$ such that distinct elements $z_1, z_2$ of~$Z$ have $d(z_1,
    z_2)>5\xi$. Let $U=\bigcup_{z\in Z}B(z, \xi)$, a union of $|Z|$ disjoint open
    disks. Cover $\Sigma\setminus U$ with a finite collection $(D_i)$ of (regular
    or pronged) chart domains of diameter less than~$\xi$. Any stream arc with
    endpoints in different components of~$U$ must intersect some $D_i$ which is
    contained entirely in $\Sigma\setminus U$, and therefore must contain an entire
    leaf segment (or perhaps two prongs of an $n$-od leaf segment in the pronged
    case). However, in each chart domain there is a minimum measure of such leaf
    segments (coming from the models of Definitions~\ref{defn:regular_model}
    and~\ref{defn:pronged_model}). The required lower bound on the stream measure
    of any stream arc which connects two different components of~$U$ follows.

\end{proof}

\section{Background}
    \label{sec:background}

In this section we review the material which forms the background of this work, and
state some key results from old and recent papers. Section~\ref{sec:tent} concerns
tent maps and their study using symbolic dynamics. In Section~\ref{sec:il} we
discuss inverse limits of tent maps. Finally, in Section~\ref{sec:measures}, we
summarize key properties of the absolutely continuous invariant measures of tent
maps, and state some results from~\cite{typical} which will play a key r\^ole in
this paper.

\subsection{Tent maps} \label{sec:tent}

\begin{defn}[Tent map, see Figure~\ref{fig:tent}] \label{defn:tent}
    Let $\lambda\in (\sqrt{2}, 2)$. 
    The \emph{(core) tent map
    \mbox{$f_\lambda\colon[a_\lambda, b_\lambda]\to [a_\lambda, b_\lambda]$} of
    slope~$\lambda$}  is the restriction of the map $T_\lambda\colon[0,1]\to
    [0,1]$ defined by $T_\lambda(x)=\min(\lambda x, \lambda(1-x))$ to the
    interval $I_\lambda := [a_\lambda, b_\lambda] := [T_\lambda^2(1/2),
    T_\lambda(1/2)]$. We write $c=1/2$, the turning point of $f_\lambda$.
\end{defn}

Throughout the paper we work with a single tent map of fixed slope and
    drop the subscripts~$\lambda$, writing simply $f\colon I\to I$, where
    $I=[a,b]$.

\begin{figure}[htbp]
        \begin{center}
            \includegraphics[width=0.4\textwidth]{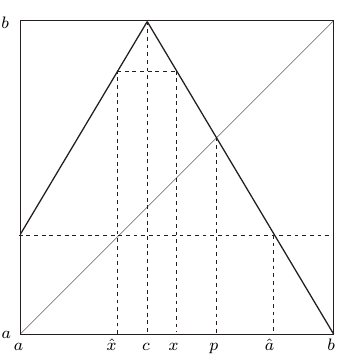}
        \end{center}
    \caption{A tent map $f\colon I\to I$}
    \label{fig:tent}
\end{figure}

\begin{notn}[$\PC$] \label{notn:PC}
    The post-critical set is denoted $\PC := \{f^r(c)\,:\,r\ge 1\}$.
\end{notn}

 Notice that $f(c)=b$ and $f(b)=a$, so that $a, b\in\PC$.

\begin{notn}[$\hx$, see Figure~\ref{fig:tent}]
    Let $\ha\in[c, b]$ satisfy $f(\ha)=f(a)$. For each $x\in [a,
    \ha]$ with $x\not=c$, we write $\hx$ for the unique element of~$I$ with
    $\hx\not=x$ and $f(\hx)=f(x)$.
\end{notn}

The fundamental idea of the symbolic approach to the dynamics of tent maps is to
code orbits of~$f$ with sequences of~$0$s and~$1$s, according as successive points
along the orbit lie to the left or to the right of the turning point. This leaves
open the question of which symbol to use when the orbit passes through the turning
point. We take a hybrid approach, leaving the choice open when the turning point is
not periodic, and making a choice based on the parity of the orbit of the turning
point when it is periodic. This approach has the advantage of giving clean
necessary and sufficient admissibility conditions for a sequence to be an
itinerary~\cite{itin-invlim}.

\begin{defn}[Itinerary $j(x)$, kneading sequence~$\kappa(f)$] \mbox{}

    If~$c$ is a periodic point of~$f$, of period~$n$, write $\ve(f)=0$
    (respectively $\ve(f)=1$) if an even (respectively odd) number of
    the points $\{f^r(c)\,:\, 1\le r < n\}$ lie in $(c, b]$. The
    \emph{itinerary} $j(x)\in\ssp$ of a point $x\in I$ is then defined by
    \[
        j(x)_r = 
        \begin{cases}
         0 & \text{ if }f^r(x)\in [a,c),\\
         1 & \text{ if }f^r(x)\in (c,b],\\
         \ve(f) & \text{ if }f^r(x)=c
        \end{cases}
        \qquad\text{ for each }r\in\N
    \]
    (we adopt the convention that $0\in\N$).

    If~$c$ is not a periodic point of~$f$, we say that a sequence $s\in\ssp$ is
    \emph{an itinerary of $x\in I$} if $f^r(x)\in [a,c]$ whenever $s_r=0$, and
    $f^r(x)\in[c,b]$ whenever $s_r=1$. Therefore each~$x\in I$ has exactly
    two itineraries if~$c\in\orb(x,f):=\{f^r(x)\,:\,r\in\N\}$, and a unique
    itinerary otherwise.

    With an abuse of notation, we write $j(x)=s$ to mean that~$s$ is an itinerary
    of~$x$. With this abusive notation we have that $j(x)=s \implies j(f(x)) =
    \sigma(s)$, where~$\sigma\colon\{0,1\}^\N\to\{0,1\}^\N$ is the shift map.

    Since $f$ is uniformly expanding on each of its branches, any sequence $s\in\{0,1\}^\N$ is the itinerary of at most one point~$x\in I$.

    The \emph{kneading sequence} $\kappa(f)\in\ssp$ of~$f$ is defined to
    be the itinerary of~$b$ (which is well defined, since if $c$ is not a periodic point then $c\not\in\orb(b,f)$).

\end{defn}

\begin{defn}[Unimodal order] \mbox{}

    The \emph{unimodal order} (also known as the \emph{parity-lexicographical
    order}) is a total order $\preceq$ on $\ssp$, defined as follows: if
    $s$ and $t$ are distinct elements of $\ssp$, let $r\in\N$ be least
    such that $s_r\not=t_r$. Then $s\preceq t$ if and only if $\sum_{i=0}^r
    s_i$ is even.
\end{defn}

\begin{rem} \label{rem:uni_order_reflects}
    The unimodal order is defined precisely to reflect the usual order of
    points on the interval: if $s$ and~$t$ are itineraries of $x$ and~$y$, then
    $x<y \implies s\prec t$; and, conversely, if $s\prec t$ then either $x<y$,
    or $s$ and~$t$ are the two itineraries of $x=y$ in the case where~$c$ is
    not periodic.
\end{rem}

\begin{rem} \label{rem:sqrt2_conditions}%
    As in Definition~\ref{defn:tent}, we only consider tent maps~$f$ whose
    slope~$\lambda$ satisfies $\sqrt{2}<\lambda<2$, that is, whose topological
    entropy $h(f) = \log\lambda$ satisfies $\frac{\log 2}{2} < h(f) <
    \log 2$. The condition $\lambda > \sqrt{2}$ is equivalent to each of the
    following:
    \begin{itemize}
        \item $f$ is not renormalizable;

        \item $\kappa(f) \succ 10(11)^\infty$;

        \item $\kappa(f) = 10(11)^k0\dots$ for some $k\ge 0$.
    \end{itemize} 
    It is also equivalent (see Theorem~\ref{thm:height-range})
    to the statement that $q(f)<1/2$, where $q(f)$ is the \emph{height} of~$f$ as given by Definition~\ref{defn:height}.

    The condition $\lambda<2$ is to avoid cluttering statements with exceptions for
    the special case~$\lambda=2$ (in which the natural extension of~$f$ is
    semi-conjugate to a generalized pseudo-Anosov map called the \emph{tight
    horseshoe map}, see for example~\cite{gpa}).
\end{rem}

The following lemma is dependent on the assumption that $\lambda>\sqrt2$.

\begin{lem} \label{lem:order_fa_p_ha}
    $f(a) < p < \ha$, where~$p$ is the fixed point of~$f$ (see
    Figure~\ref{fig:tent}).
\end{lem}
\begin{proof}
    By Remark~\ref{rem:sqrt2_conditions} we have $j(a) = 0(11)^k0\dots$ for some~$k\ge 0$, so that $j(\ha)= 1(11)^k0\dots$ and $j(f(a))= (11)^k0\dots$.
    Comparing these itineraries in the unimodal order with $j(p)=1^\infty$ and
    using Remark~\ref{rem:uni_order_reflects} gives the result.
\end{proof}
    
\subsection{Inverse limits} \label{sec:il}

Inverse limits of tent maps have been intensively studied, both for their intrinsic
interest as complicated topological spaces, and in dynamical systems as models for
attractors~\cite{jerana,BBD,Ingram,BM2,BorSt,BB,B2,raines}. 

\begin{defns}[Inverse limit~$\hI$, projections~$\pi_r$, natural
    extension~$\hf$] \mbox{} \label{defns:il}

    The \emph{inverse limit~$\hI = \hI_\lambda$} of the tent map $f\colon I\to I$
    is the space
    \[
        \hI = \{\bx\in I^\N\,:\, f(x_{r+1}) = x_r \text{ for all }r\in\N\}
        \subset I^\N,
    \] 
    endowed with the metric $d(\bx, \by) = \sum_{r=0}^\infty |x_r-y_r|/2^r$,
    which induces its natural topology as a subset of the product space $I^\N$.

    We denote elements of $\hI$ with angle brackets, $\bx = \thrn{x}$.

    For each $r\in\N$, we denote by $\pi_r\colon \hI\to I$ the projection onto
    the $r^\text{th}$ coordinate, $\pi_r(\bx)=x_r$.

    The \emph{natural extension} of~$f$ is the homeomorphism
    $\hf\colon\hI\to\hI$ defined by
    \[
        \hf(\thrn{x}) = \thr{f(x_0), x_0, x_1, x_2, \dots}.
    \]
\end{defns}

    Clearly each $\pi_r$ is a semi-conjugacy from $\hf\colon\hI\to\hI$ to
    $f\colon I\to I$. It is straightforward to see that any semi-conjugacy from
    a homeomorphism~$F\colon X\to X$ to~$f$ factors through each~$\pi_r$.

    Since $\pi_r = \pi_0\circ \hf^{-r}$, we have that, for all $\bx\in \hI$,
    \begin{equation}
    \label{eq:usefulhi}
        \bx = \thr{\pi_0(\bx), \pi_0(\hf^{-1}(\bx)), \pi_0(\hf^{-2}(\bx)),
        \dots}.
    \end{equation}

The next result is an important consequence of the condition that
$\lambda\in(\sqrt2, 2)$.

\begin{lem} \label{lem:many_choices}%
     For every $\bx\in\hI$, there are infinitely many~$r\in\N$ with $x_r\in[f(a), b)$: that is, infinitely many $r\in\N$ such that $x_r$ has two preimages.
\end{lem}
\begin{proof}
     Let~$p>c$ be the fixed point of~$f$. If $x_r \le p$ then $x_{r+1}$, being a
     preimage of~$x_r$, satisfies either $x_{r+1}<c$ or $x_{r+1}\ge p$ (see
     Figure~\ref{fig:tent}). Since $\lambda<2$, there is an upper bound on the
     number of consecutive entries of~$\bx$ which are smaller than~$c$: hence
     $x_r\ge p$ for infinitely many $r\in\N$.

     Since $p>f(a)$ by Lemma~\ref{lem:order_fa_p_ha}, there are infinitely many~$r$
     with $x_r\in (f(a), b]$. However if $x_r=b$ for some $r\ge 2$ then
     $x_{r-2}=f(a)$, and the result follows.
\end{proof}

\begin{defns}[$\pi_0$-fibers $\fib{x}$, cylinder sets $\lbrack x, y_1, \dots,
    y_r\rbrack $] \label{defs:cylinder}
    For each $x\in I$ we define the \emph{$\pi_0$-fiber~$\fib{x}$ of~$\hI$
    above~$x$} by $\fib{x} = \pi_0^{-1}(x)$.

    More generally, if $f(y_1)=x$ and $f(y_i) = y_{i-1}$ for $2\le i\le r$, we
    define the \emph{cylinder set} $\fib{x, y_1, \dots, y_r} \subset \fib{x}$
    by
    \[
        \fib{x, y_1, \dots, y_r} = \{\bx\in \fib{x}\,:\, x_i = y_i \text{ for }
        1\le i
        \le r\}.
    \]

    Note that the diameter of $\fib{x, y_1, \dots, y_r}$ is bounded above by $1/2^r$, so that the cylinder subsets of~$\fib{x}$ form a basis for its topology.
\end{defns}

\begin{defn}[Arc]
    An \emph{(open, half-open, or closed) arc} in~$\hI$ is a subset of~$\hI$
    which is a continuous injective image of an (open, half-open, or closed)
    interval.

    An open arc in~$\hI$ is therefore synonymous with an immersed line in~$\hI$ (Definition~\ref{defn:immersed}).
\end{defn}

We will need the following well-known facts about the topology of~$\hI$ (the former is theorem~2.2 of~\cite{raines}, the latter is folklore).
\begin{lem} \label{lem:basic_top} \mbox{}
    \begin{enumerate}[(a)]
        \item If $\bx\in\hI$ is not in the $\omega$-limit set
        $\omega(\fib{c},\hf)$ of the critical fiber, then it has a neighborhood
        homeomorphic to the product of an open interval and a Cantor set.
        \item Every path component of~$\hI$ is either a point or an arc.
    \end{enumerate}
\end{lem}

It is worth noting how badly Lemma~\ref{lem:basic_top}(a) can fail for points
in the $\omega$-limit set of the critical fiber. Barge, Brucks, and
Diamond~\cite{BBD} show that there is a dense~$G_\delta$ set of
parameters~$\lambda$ for which every open subset of $\hI_\lambda$ contains a
homeomorphic copy of $\hI_{\lambda'}$ for every $\lambda'\in(\sqrt 2, 2)$.

Nevertheless, the fibers of tent maps inverse limits are regular (in the
case~$\lambda>\sqrt{2}$).

\begin{lem} \label{lem:cantor_fibre}%
    $\fib{x}$ is a Cantor set for all $x\in I$.
\end{lem}
\begin{proof}
     $\fib{x}$ is certainly compact and totally disconnected. To see that it is
    perfect, let $\bx\in\fib{x}$ and~$R>0$. By Lemma~\ref{lem:many_choices} there
    is some $r>R$ for which $x_r$ has two preimages $x_{r+1}$ and $\hx_{r+1}$. Let
    $\bx'\in\fib{x}$ have $x'_i = x_i$ for $i\le r$ and $x'_{r+1} = \hx_{r+1}$.
    Then $0<d(\bx, \bx') < 1/2^R$.
\end{proof}

We can introduce symbolic dynamics on $\hI$ just as we did for forwards orbits of
tent maps.

\begin{defn}[Itineraries on~$\hI$]
    Let $\bx\in\hI$. An element $s$ of $\ssp$ is said to be an \emph{itinerary}
    of~$\bx$ if $s_r = 0 \implies x_{r+1}\le c$, and $s_r = 1\implies x_{r+1}
    \ge c$. In an abuse of notation, we write $J(\bx)=s$ to mean that $s$ is an
    itinerary of~$\bx$.
\end{defn}

\begin{rem} \label{rem:itin_determines} \mbox{}
    \begin{enumerate}[(a)]
    \item
        An element~$\bx$ of~$\fib{x}$ is determined by its itinerary $J(\bx)=s$,
        since the $x_r$ are determined inductively by $x_0 = x$; and $x_{r+1}$ is
        the unique preimage of $x_r$ which lies in $[a, c]$ if $s_r=0$, or in $[c,
        b]$ if $s_r=1$.

    \item 
        If $x\not\in\PC$, then every element of $\fib{x}$ has a unique itinerary,
        and if $c$ is not a periodic point of~$f$ then every element of $\hI$ has
        at most two itineraries. More generally, for any $\bx\in\hI$, the set of
        itineraries of $\bx$ is closed in $\ssp$.
    \end{enumerate}
\end{rem}

By Remark~\ref{rem:itin_determines}(b), if $x\not\in\PC$ then the fiber~$\fib{x}$
can be totally ordered by the unimodal order on the itineraries of its elements. In
particular, this makes it possible to define consecutive elements of the fiber.

\begin{defn}[Consecutive elements of~$\fib{x}$] \label{defn:consecutive}
    Let $x\not\in\PC$, so that every $\bx\in\fib{x}$ has a unique itinerary
    $J(\bx)$. We say that distinct points $\bx,
    \bx'\in\fib{x}$ are \emph{consecutive} if there is no point of $\fib{x}$
    whose itinerary lies strictly between $J(\bx)$ and $J(\bx')$ in the
    unimodal order.
\end{defn}

More generally, we can consider all of the itineraries which are realized in a fiber.

\begin{defn}[$\cK_x$, $L_x\colon\cK_x \to\fib{x}$] 
    \label{defn:admiss-itin-fiber}
    For each $x\in I$, write   
    \[
        \cK_x = \{s\in\ssp\,:\, s \text{ is an itinerary of some 
            $\bx\in\fib{x}$}\}.
    \]
    By Remark~\ref{rem:itin_determines}(a), there is a function
    $L_x\colon\cK_x\to\fib{x}$ defined by $L_x(J(\bx))=\bx$ for all
    $\bx\in\fib{x}$.
\end{defn}

\begin{lem} \label{lem:cpctKx} 
    For each $x\in I$, 
        $\cK_x$ is compact and  $L_x\colon \cK_x\to\fib{x}$ is continuous. 

    More generally, if $\cV = \{(x,s)\in I\times\{0,1\}^\N\,:\, s\in \cK_x\}$,
    then the function $L\colon\cV\to\hI$ defined by $L(x,s)=L_x(s)$ is continuous.
\end{lem}

\begin{proof}
    Compactness of $\cK_x$ is straightforward: if $s\upi\to s$ are itineraries of
    $\bx\upi\in\fib{x}$ with $\bx\upi\to\bx$, then $s$ is an itinerary of~$\bx$.

    For continuity of $L_x$, note that if $\bx,
    \by\in\fib{x}$ and $J(\bx)_r = J(\by)_r$ for all $r\le R$, then $x_r=y_r$
    for all $r\le R$.

    For the final statement, note that if $\bx, \by\in\hI$ with
    $|x_0-y_0|<\epsilon$ and $J(\bx)_r = J(\by)_r$ for all $r\le R$, then
    $|x_r-y_r|<\epsilon/\lambda^r$ for all $r\le R$.
\end{proof}

\begin{rem} \label{rem:cont-itin}
    It follows that if $x\not\in\PC$, so that $J|_{\fib{x}}\colon \fib{x}\to\cK_x$
    is well-defined, then this map is continuous, being the inverse of the
    continuous bijection $L_x$.
\end{rem}

\subsection{Invariant measures} \label{sec:measures}

    Recall (for example~\cite{baladi,henkthesis,DGP,HG,Keller,LY,Ry}) that each
    tent map~$f=f_\lambda$ has a unique invariant Borel probability measure~$\mu =
    \mu_\lambda$ which is absolutely continuous with respect to Lebesgue
    measure~$m$, and $d\mu = \varphi\, dm$, for some $\varphi\in L^1(m)$ defined on
    $[a,b]\setminus \PC$ which is bounded away from zero, of bounded variation, and
    satisfies, for all $x\not\in\PC$ and $r\in\N$,
    \begin{equation}
    \label{eq:pf}
        \varphi(x) = \sum_{f^r(y)=x}\frac{\varphi(y)}{\lambda^r}
    \end{equation}
    Moreover, $\mu$ is ergodic.

    This measure~$\mu$ induces an ergodic $\hf$-invariant OU probability
     measure~$\hmu$ on~$\hI$ characterized by $(\pi_r)_*(\hmu) = \mu$ for
     all~$r\in\N$~\cite{para}.

The following result is from~\cite{typical}.

\begin{lem} \label{lem:alpha-measure}
    For each $x\not\in\PC$, there is a Borel measure $\alpha_x$ on~$\hI$, which is
    supported on~$\fib{x}$ and is given on cylinder subsets of~$\fib{x}$ by
    \[
        \alpha_x(\fib{x, y_1, \dots, y_r}) = \frac{\varphi(y_r)}{\lambda^r}.
    \]
    In particular, $\alpha_x(\fib{x}) = \varphi(x)$.
\end{lem}

\begin{rems} \mbox{} \label{rems:alpha-x}
    \begin{enumerate}[(a)]
    \item
        $\alpha_x$ is OU: since the cylinder subsets form a basis of $\fib{x}$,
        $\alpha_x$ is positive on non-empty open subsets of $\fib{x}$; and since
        any point~$\bx$ is contained in $\fib{x_0, \dots, x_r}$ for all~$r$,
        $\alpha_x$ is non-atomic.

    \item
        Since $\hf(\fib{x, y_1, \dots, y_r}) = \fib{f(x), x, y_1, \dots, y_r}$, we
        have $\alpha_{f(x)}(\hf(A)) = \lambda^{-1} \alpha_x(A)$ for every Borel
        subset~$A$ of~$\fib{x}$.

    \item
        If~$c$ is not pre-periodic, then~$\varphi$ can be extended inductively
        over~$\PC$ by
        \[
            \varphi(f^i(c)) = \sum_{f(y)=f^i(c)} \frac{\varphi(y)}{\lambda},
        \] 
        so that it satisfies equation~\eqref{eq:pf} for all~$x\in I$; and therefore
        the measures $\alpha_x$ can be defined for all~$x\in I$ in such a way that
        they satisfy~(a) and~(b).
    \end{enumerate}
\end{rems}

    Remarks~\ref{rems:alpha-x}(b) says that~$\hf$ contracts $\pi_0$-fibers
    uniformly by a factor~$\lambda$, with respect to the measures~$\alpha_x$.

    The images of these fibers under the semi-conjugacy $g\colon\hI\to\Sigma = S^2$
    --- which are almost all arcs --- will form the stable foliation or turbulation
    of the generalized or measurable pseudo-Anosov map semi-conjugate to $\hf$,
    with the measure on leaves/streamlines coming from the~$\alpha_x$. The unstable
    foliation/turbulation will come from the path components of~$\hI$, together
    with Lebesgue measure~$m$ on their $\pi_0$-images, which is expanded by~$\hf$
    by a factor~$\lambda$.

    The remaining definitions and results in this section are from~\cite{typical}:
    Theorem~\ref{thm:+ve-0-box} is lemma~8.3, Theorem~\ref{thm:holonomy_alpha} is
    theorem~7.1, Theorem~\ref{thm:disintegrate} is theorem~8.1,
    Theorem~\ref{thm:typical} is theorem~1.1(a), and Lemma~\ref{lem:decompose-glr}
    is lemma~3.5. The reader may find it helpful to interpret these in light of the
    genesis of the invariant turbulations just described. The $0$-boxes of
    Definition~\ref{defn:0-box} will provide the unstable fibers of a tartan, whose
    stable fibers come from the $\pi_0$-fibers of~$\hI$;
    Theorem~\ref{thm:+ve-0-box} will provide a tartan~$R$ with
    $\mu(R^\pitchfork)>0$, which will enable us to show fullness of the
    turbulations; Theorem~\ref{thm:holonomy_alpha} translates to unstable holonomy
    invariance of these tartans (stable holonomy invariance is straightforward);
    Theorem~\ref{thm:disintegrate} will be used to show that they are compatible
    with~$\mu$; and Theorem~\ref{thm:typical}, together with
    Lemma~\ref{lem:decompose-glr}, ensures that an unstable turbulation can be
    built from well-behaved path components.

\begin{defn}[$0$-flat arc]
    An arc~$\Gamma$ in~$\hI$ is \emph{$0$-flat (over the subinterval $J$ of $
    I$)} if $\pi_0$ is injective on~$\Gamma$ (and maps it onto~$J$).
\end{defn}

\begin{rem} \label{rem:0-flat-itin}
    Alternatively, the arc~$\Gamma$ is $0$-flat if and only if
    $c\not\in\Int(\pi_r(\Gamma))$ for all $r\ge 1$. In particular, the points 
    of a $0$-flat arc all share a common itinerary.
\end{rem}

\begin{defn}[$0$-box] \label{defn:0-box}
    Let~$J$ be a subinterval of~$I$. A \emph{$0$-box~$B$ over~$J$} is a Borel
    disjoint union of $0$-flat arcs over~$J$.
\end{defn}

\begin{thm}
    \label{thm:+ve-0-box}
    There exists a $0$-box~$B$ with $\hmu(B)>0$.
\end{thm}

\begin{thm}[Holonomy invariance]
    \label{thm:holonomy_alpha}
    Let~$B$ be a $0$-box over a subinterval~$J$. Then
    \[
        \alpha_x(B)  = \alpha_y(B) \quad \text{ for all $x,y\in
        J\setminus\PC$}.
    \]
\end{thm} 

\begin{thm}
    \label{thm:disintegrate}
    For every Borel subset~$E$ of~$\hI$, we have
    \[
        \hmu(E) = \int_I\alpha_x(E) dm(x).
    \]
\end{thm}

\begin{defn}[Globally leaf regular] \label{defn:glob_leaf_reg}
    A point $\bx\in\hI$ is said to be \emph{globally leaf regular} if its path
    component~$\Gamma$ is an open arc, and each component of
    $\Gamma\setminus\{\bx\}$ is dense in~$\hI$.
\end{defn}

\begin{thm}[Typical path components]
    \label{thm:typical}
    $\hmu$-almost every point of~$\hI$ is globally leaf regular.
\end{thm}

Globally leaf regular path components (or, more generally, open arcs) can be
decomposed into $0$-flat arcs.

\begin{lem}[$0$-flat decomposition of globally leaf regular component]
    \label{lem:decompose-glr}
    Let~$\Gamma\subset\hI$ be an open arc. Then there is a countable (finite,
    infinite, or bi-infinite) sequence $(\Gamma_i)$ of $0$-flat arcs in~$\hI$,
    unique up to re-indexing by an order-preserving or order-reversing
    bijection, such that
    \begin{enumerate}[(a)]
        \item $\Gamma = \bigcup_i \Gamma_i$
        \item Each $\Gamma_i$ is disjoint from all other~$\Gamma_j$, except
        that it intersects $\Gamma_{i-1}$ and $\Gamma_{i+1}$, if they exist, at
        its endpoints.
        \item $\bx\in\Gamma$ is an endpoint of some $\Gamma_i$ if and only
        if $x_r=c$ for some~$r$.
    \end{enumerate}

\end{lem}

\begin{rem} \label{rem:glr-zero-measure}
    It follows from Lemma~\ref{lem:decompose-glr} and
    Theorem~\ref{thm:disintegrate} that if $\Gamma$ is the path component of a
    globally leaf regular point, then $\hmu(\Gamma)=0$. This is because each
    $\Gamma_i$ of Lemma~\ref{lem:decompose-glr} intersects each~$\pi_0$-fiber in at
    most one point, so that $\alpha_x(\Gamma_i)=0$ for all~$x$.
\end{rem}

\section{Extreme elements of \texorpdfstring{$\pi_0$}{pi-0}-fibers and the
outside map}
    \label{sec:extreme-outside}

If $x\not\in\PC$ then every element of~$\fib{x}$ has a well-defined itinerary, so
that $\fib{x}$ can be totally ordered by the unimodal order and, being compact, has
minimum and maximum elements. When $x\in\PC$ there is no such order on the fiber,
but there are nevertheless minimum and maximum elements of~$\cK_x$, the set of
itineraries realized on~$\fib{x}$ (recall Definition~\ref{defn:admiss-itin-fiber}),
which correspond to unique `minimum' and `maximum' elements of the fiber itself.
When $x=a$ or $x=b$, it turns out that these elements coincide.

These extreme elements of fibers are important because the identifications induced
by the map $g\colon \hI\to\Sigma$ which semi-conjugates $\hf$ to the sphere
homeomorphism $F\colon\Sigma\to\Sigma$ take place along their orbits (the
identifications will be described in Lemma~\ref{lem:rat-equiv-class}). In this
section we study the action of~$\hf$ on extreme elements, which we model with a
circle map called the \emph{outside map} associated to~$f$. The
construction~\cite{prime} of the sphere homeomorphism $F\colon\Sigma\to\Sigma$
depends crucially on the dynamics of this circle map, and in particular on its
rotation number, which is called the \emph{height} $q(f)$ of~$f$.

\subsection{Extreme elements}

Recall (Definition~\ref{defn:admiss-itin-fiber}) that
    \[
        \cK_x = \{s\in\ssp\,:\,s \text{ is an itinerary of some $\bx\in\fib{x}$}\}
    \] 
    for each~$x\in I$, and that $L_x\colon \cK_x\to\fib{x}$ associates to each
    $s\in\cK_x$ the unique~$\bx\in\fib{x}$ with that itinerary. Since $\cK_x$ is
    compact (Lemma~\ref{lem:cpctKx}) and every non-empty subset of~$\ssp$ has an
    infimum and supremum, $\cK_x$ has a maximum element~$\cU_x$ and a minimum
    element~$\cL_x$ with respect to the unimodal order.

\begin{defn}[Upper, lower, extreme elements, $\be(x_u)$, $\be(x_\ell)$]
    \label{defn:upper_lower_extreme}
    For each $x\in I$, we write $\be(x_u) = L_x(\cU_x)$ and $\be(x_\ell) =
    L_x(\cL_x)$ for the \emph{upper} and \emph{lower} elements of~$\fib{x}$, having
    as itineraries $\cU_x$ and $\cL_x$ respectively, and refer to them both as
    \emph{extreme} elements. (The apparently idiosyncratic notation will become
    more natural in Section~\ref{sec:outside}, where a map~$\be$ onto the set of
    extreme elements will be defined.)
\end{defn}

The following lemma describes the action of $\hf$ on extreme elements.
\begin{lem}\label{lem:action_on_extremes}
    Let $x\in I$.
    \begin{enumerate}[(a)]
        \item If $x<f(a)$, so that $x$ has a unique $f$-preimage $x^1>c$, then
        \[
            \be(x_\ell) = \hf(\be(x^1_u)) \quad\text{and}\quad
            \be(x_u) = \hf(\be(x^1_\ell)).
        \]

        \item If $f(a)\le x < b$, so that $x$ has preimages $x^0<c$ and $x^1>c$,
        then
         \[
            \be(x_\ell) = \hf(\be(x^0_\ell)) \quad\text{and}\quad
            \be(x_u) = \hf(\be(x^1_\ell)).
        \]

        \item If $x=b$, so that $x$ has unique preimage~$c$, then
        \[
            \be(b_\ell) = \be(b_u) = \hf(\be(c_\ell)).
        \]
    \end{enumerate}
\end{lem}
\begin{proof}
    \begin{enumerate}[(a)]
        \item We have $\fib{x} = \hf(\fib{x^1})$; and if $\bx\in\fib{x^1}$, then
        the itineraries of $\hf(\bx)\in\fib{x}$ are exactly the itineraries of
        $\bx$ with a symbol~$1$ prepended, so that $\hf|_{\fib{x^1}}$ reverses the
        order on itineraries.

        \item We have $\fib{x} = \hf(\fib{x^0}) \sqcup \hf(\fib{x^1})$; if
        $\bx\in\fib{x^0}$ (respectively $\bx\in\fib{x^1}$) then the itineraries
        of $\hf(\bx)\in\fib{x}$ are exactly those of $\bx$ with a symbol~$0$
        (respectively~$1$) prepended. Therefore $\hf|_{\fib{x^0}}$ preserves
        the order on itineraries and $\hf|_{\fib{x^1}}$ reverses it; and all
        elements of $\hf(\fib{x^1})$ have larger itineraries than all elements
        of $\hf(\fib{x^0})$.

        \item We have $\fib{b} = \hf(\fib{c})$; and if $\bx\in\fib{c}$, then for each itinerary of~$\bx$ there are two itineraries of $\hf(\bx)\in\fib{b}$, one with the symbol~$0$ prepended and one with the symbol~$1$ prepended. 
    \end{enumerate}
\end{proof}

Recall from Definition~\ref{defn:consecutive} that, provided $x\not\in\PC$, we can
define the notion of consecutive elements of~$\fib{x}$. These will play an
important r\^ole in the theory, since typically two elements of a fiber are
identified by the semi-conjugacy~$g\colon\hI\to\Sigma$ if and only if they are
consecutive. The proof of Lemma~\ref{lem:action_on_extremes} also elucidates the
action of $\hf$ on pairs of consecutive points, as described by the next lemma.

\begin{lem}\label{lem:preserve_consecutive}
    Let $x\in I\setminus\PC$.
    \begin{enumerate}[(a)]
        \item Let $r>0$ and suppose that $f^r(x)\not\in\PC$. If $\bx$ and
        $\bx'$ are consecutive in~$\fib{x}$, then $\hf^r(\bx)$ and
        $\hf^r(\bx')$ are consecutive in~$\fib{f^r(x)}$.

        \item If $\bx$ and $\bx'$ are consecutive in~$\fib{x}$ and $x_1=x'_1$,
        then $\hf^{-1}(\bx)$ and $\hf^{-1}(\bx')$ are consecutive
        in~$\fib{x_1}$.

        \item If $\bx$ and $\bx'$ are consecutive in~$\fib{x}$ and
        $x_1\not=x'_1$, then $\hf^{-1}(\bx) = \be(x_{1,u})$ and
        $\hf^{-1}(\bx') = \be(x'_{1,u})$.

        \item If $x\in [a, \ha] \setminus\{c\}$, then $\hf(\be(x_u))$ and
        $\hf(\be(\hx_u))$ are consecutive in $\fib{f(x)}$.
    \end{enumerate}
\end{lem}

The extreme elements of a fiber~$\fib{x}$ coincide if and only if $x$ is an
endpoint of~$I$:

\begin{lem}
    $\be(a_\ell)=\be(a_u)$, and $\be(b_\ell)=\be(b_u)$. On the other hand, if $x\in
    I\setminus\{a,b\}$ then $\be(x_\ell)\not=\be(x_u)$.
\end{lem}

\begin{proof}
    $\be(b_\ell)=\be(b_u)$ by Lemma~\ref{lem:action_on_extremes}(c), and it follows
    from Lemma~\ref{lem:action_on_extremes}(a) that $\be(a_\ell)=\be(a_u) =
    \hf(\be(b_u) = \be(b_\ell))$.

    Suppose then that $x\in(a,b)$. By Lemma~\ref{lem:many_choices}, there is some
    $y\in[f(a), b)$ and $r\in\N$ such that $f^r(y)=x$ and $f^i(y)\in(a, f(a))$ for
    $1\le i < r$. Then $\be(y_\ell)\not=\be(y_u)$ by
    Lemma~\ref{lem:action_on_extremes}(b); and $\{\be(x_\ell), \be(x_u)\} =
    \{\hf^r(\be(y_\ell)), \hf^r(\be(y_u))\}$ by
    Lemma~\ref{lem:action_on_extremes}(a), so that
    \mbox{$\be(x_\ell)\not=\be(x_u)$} as required.
\end{proof}
    
\subsection{The outside map} \label{sec:outside}
    
We model the action of $\hf$ on the set of extreme elements, as described
in Lemma~\ref{lem:action_on_extremes}, with a circle map (see
Figure~\ref{fig:outside}).

Let~$S$ be the circle obtained by gluing together two copies of~$I$ at
their endpoints. We denote the points of~$S$ by $x_\ell$ and $x_u$, for
$x\in I$, depending on whether they belong to the `lower' or `upper' copy
of~$I$. Therefore $a_\ell=a_u$ and $b_\ell=b_u$, and we also denote these
points of~$S$ with the symbols~$a$ and~$b$ respectively. We write
$\tau\colon S\to I$ for the projection $\tau(x_\ell)=\tau(x_u)=x$. When we
use interval notation in~$S$, we regard the interval as an arc in~$S$ which
goes in the positive sense from its first endpoint to its second endpoint,
in the model of Figure~\ref{fig:outside}. Thus, for example, $[a, b]$
contains $x_\ell$ for all~$x$, while $[b, a]$ contains~$x_u$ for all~$x$.

\begin{figure}[htbp]
    \begin{center}
        \includegraphics[width=0.6\textwidth]{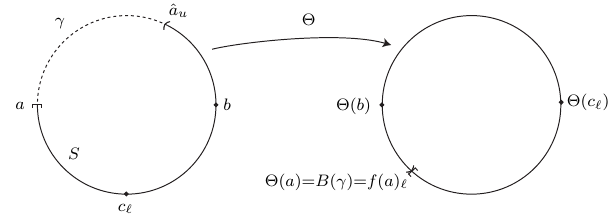}
    \end{center}
\caption{The circle~$S$ and the map $\tB\colon [a, \re_u)\to S$}
\label{fig:outside}
\end{figure}

Definition~\ref{defn:upper_lower_extreme} provides an injective function $\be\colon
S\to\hI$ with image the set of extreme elements of~$\hI$. This map is clearly
discontinuous, since $\hI$ does not contain a circle: its discontinuity set is
discussed in Lemma~\ref{lem:tB-facts}(d) and
Remark~\ref{rem:height-intervals-discont}.

Similarly, we define a symbolic version $\cE\colon S\setminus\{a,b\}\to
\ssp$ by $\cE(x_\ell)=\cL_x$ and $\cE(x_u)=\cU_x$, so that $\cE(y)$ is an itinerary
of $\be(y)$ for each $y\in S\setminus\{a, b\}$ (recall that $\cL_x$ and $\cU_x$ are
the minimum and maximum itineraries in the fiber~$\fib{x}$).

If $\tB\colon[a, \re_u)\to S$ is defined (see Figure~\ref{fig:outside}) by
\begin{equation}
    \label{eq:tB}
    \begin{array}{rcll} 
      \tB(x_\ell) &=& f(x)_\ell & \text{ if }x\in[a,c]\\
      \tB(x_\ell) &=& f(x)_u & \text{ if }x\in[c,b], \quad\text{ and}\\ 
      \tB(x_u) &=& f(x)_\ell & \text{ if }x\in(\re,b],
    \end{array}
\end{equation}
then it follows from Lemma~\ref{lem:action_on_extremes} --- and of course this
is the point of the definition --- that
\begin{equation}
    \label{eq:model_action}
    \hf(\be(y)) = \be(\tB(y)) \qquad \text{for all }y\in [a, \re_u)
\end{equation}
(if $y\in [\re_u, a)$ then $\hf(\be(y))$ is not an extreme point by
 Lemma~\ref{lem:action_on_extremes}). 

 $\tB$ is a continuous bijection, whose inverse has a one-sided discontinuity at
 $f(a)_\ell$ --- it is discontinuous as we approach $f(a)_\ell$ in the positive
 sense --- and is continuous elsewhere.

$\tB$ can be extended to a continuous monotone circle map $B\colon S\to S$
of degree~1 by setting $B(y)=f(a)_\ell$ for all $y\in\gamma := [\re_u, a]$.
$B$ is called the \emph{outside map} associated to the tent map~$f$
(Figure~\ref{fig:outside_2}). We write $\ingam:=(\ha_u, a)$ for the interior
of~$\gamma$.

\begin{figure}[htbp]
        \begin{center}
            \includegraphics[width=0.45\textwidth]{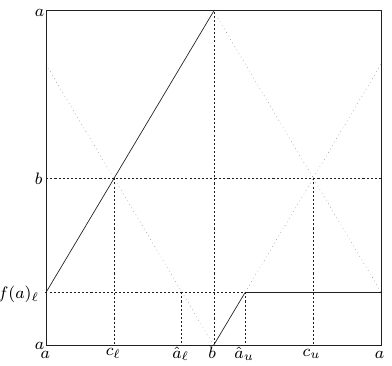}
        \end{center}
    \caption{The outside map $B\colon S\to S$. $\tB$ has the same graph, except
    that it is undefined on the interval $[\ha_u, a)$, on which $B$ takes the
    constant value $f(a)_\ell$.}
    \label{fig:outside_2}
\end{figure}

The next lemma gathers some straightforward facts about $\tB$.
\begin{lem} \label{lem:tB-facts} \mbox{}
    \begin{enumerate}[(a)]
        \item If $\tB^r(y)$ exists for some $r>0$, then $\tau(\tB^r(y)) =
        f^r(\tau(y))$.

        \item $\tB^{-r}(y)$ exists for all $y\in S$ and all $r\in\N$, and
    $\be(\tB^{-r}(y)) = \hf^{-r}(\be(y))$

    \item An explicit formula for $\be(y)$ is given by
        \[
            \be(y) =  \thr{\tau(y), \,\,\tau(\tB^{-1}(y)),
            \,\,\tau(\tB^{-2}(y)),\,\,\dots}.
        \]

    \item $\be$ is discontinuous at $y\in S$ if and only if $y\in\orb(f(a)_\ell,
        \tB)$, and these discontinuities are one-sided (as we approach~$y$ in the
        positive sense).

    \item $\cE\colon S\setminus\{a,b\} \to \ssp$ is locally constant at $y\in
        S\setminus\{a,b\}$ if and only if $y\not\in\overline{\orb(f(a)_\ell,
        \tB)}$.
    \end{enumerate}
\end{lem}
\begin{proof}
    (a) and (b) are immediate from~\eqref{eq:tB} and~\eqref{eq:model_action}
    respectively. 
    \begin{enumerate}[(a)]
        \addtocounter{enumi}{2}
        \item We have $\be(y) = \thr{\pi_0(\be(y)),\,\,
            \pi_0(\hf^{-1}(\be(y))) ,\,\, \pi_0(\hf^{-2}(\be(y))), \,\,\dots}$
            by~\eqref{eq:usefulhi}, and~(b) gives $\hf^{-r}(\be(y)) =
            \be(\tB^{-r}(y))\in\fib{\tau(\tB^{-r}(y))}$, so that
            $\pi_0(\hf^{-r}(\be(y))) = \tau(\tB^{-r}(y))$.
        \item follows from~(c) because $\tB^{-1}$ has a one-sided discontinuity at
            $f(a)_\ell$ and is continuous elsewhere.
        \item Since $c_u$ is not in the domain of $\tB$, it follows
            from~(c) that $\be(y)_r = c$ for some $r\ge 1$ if and only if
            $y\in\orb(\tB(c_\ell), \tB) = \orb(b, \tB) = \{a, b\}
            \cup \orb(f(a)_\ell, \tB)$.
    \end{enumerate}
\end{proof}

From parts~(d) and~(e) of this Lemma, we see that the $\tB$-orbit of $f(a)_\ell$
plays an important r\^ole in understanding the structure of the extreme elements of
$\hI$. In particular, there is a strong distinction between the cases where this
orbit is finite and those where it is infinite. Note carefully that the orbit can
be finite not only when it is periodic or preperiodic, but also, more commonly,
when $\tB^r(f(a)_\ell)\in [\ha_u, a)$ is outside the domain of $\tB$ for some~$r$
(see Theorem~\ref{thm:height-intervals-rational}(c)).

\subsection{Dynamics of the outside map}

The outside map was defined in~\cite{gpa}, and a closely related map was introduced
in~\cite{henkthesis}. In this section we state those of its dynamical properties
which we will use and make some related definitions.
Theorem~\ref{thm:height-range}, parts~(a) and~(b) of
Theorem~\ref{thm:height-intervals-rational}, and part~(a) of
Theorem~\ref{thm:height-irrational} can be found in~\cite{HS}; while the remaining
parts of Theorems~\ref{thm:height-intervals-rational}
and~\ref{thm:height-irrational} are proved in~\cite{prime}.

\begin{defn}[Height]\label{defn:height}\mbox{}
    The \emph{height} $q_\lambda =q(f_\lambda)$ of $f$ is the Poincar\'e rotation number of the outside map $B\colon S\to S$.
\end{defn}

\begin{thm} \label{thm:height-range}%
    $q_\lambda$ is decreasing in~$\lambda$, and $q_\lambda\in(0,1/2)$ for
    all~$\lambda$. (In fact, for general tent maps, $q_\lambda = 0$ if and only
    if $\lambda=2$, and $q_\lambda=1/2$ if and only if $\lambda\in(1,
    \sqrt{2}]$.)
\end{thm}

\begin{thm}[The rational height case]  \label{thm:height-intervals-rational}%
    Let $m/n\in(0,1/2)$ be rational.
    \begin{enumerate}[(a)]
        \item There is a non-trivial closed interval $\Lambda_{m/n} =
        [\lambda_{m/n}^-, \lambda_{m/n}^+]
        \subset(\sqrt{2}, 2)$ of parameters such that $q_\lambda = m/n$ if and
        only if $\lambda\in\Lambda_{m/n}$.

        \item The critical orbit of~$f_\lambda$ is periodic (with period~$n$)
        if $\lambda = \lambda_{m/n}^-$, and pre-periodic (to period~$n$) if
        $\lambda=\lambda_{m/n}^+$.

        \item For each $\lambda\in\Lambda_{m/n}$, the smallest positive
        integer~$r$ with $\tB^r(a)\in\gamma= [\ha_u, a]$ is $r=n$. We have
        $\tB^n(a) = a$ if $\lambda =
        \lambda_{m/n}^-$; $\tB^n(a) = \ha_u$ if $\lambda = \lambda_{m/n}^+$;
        and $\tB^n(a)\in (\ha_u, a)$ if $\lambda\in\Int(\Lambda_{m/n})$.
        Moreover, there is a unique $\lambda =
        \lambda_{m/n}^\ast\in\Lambda_{m/n}$ for which $\tB^n(a) = c_u$.

        \item For each $\lambda\in\Lambda_{m/n}$, the set
        $S\setminus\bigcup_{r\ge 0}B^{-r}(\ingam)$ of points whose $B$-orbits
        never enter $\ingam = (\ha_u, a)$ is a period~$n$ orbit of $B$, which
        attracts the $\tB^{-1}$-orbit of every point of~$S$.
    \end{enumerate}
\end{thm}
    
\noindent Note that, by~\eqref{eq:model_action}, we have $\tB^n(a) =
f^n(a)_u$ whenever $\lambda\in\Lambda_{m/n}$.

\begin{thm}[The irrational height case] \label{thm:height-irrational}%
    Let $q\in(0,1/2)$ be irrational.
    \begin{enumerate}[(a)]
        \item There is a parameter
        $\lambda_q\in(\sqrt{2}, 2)$ such that $q_\lambda = q$ if and only if
        $\lambda = \lambda_q$.

        \item If $\lambda=\lambda_q$ then $f$ is post-critically infinite.
        Moreover $\orb(f(a)_\ell, \tB)$ is disjoint from \mbox{$\gamma=[\ha_u,
        a]$}, and the closure of this orbit is a Cantor set
        $\cC_S=S\setminus\bigcup_{r\ge 0}\tB^{-r}(\ingam)$, the set of points
        whose~$\tB$-orbit never enters~$\ingam$.        
        In particular~$\cC_S$ contains $a$ and $\ha_u$, since
        $\tB(a)=f(a)_\ell$ has an orbit which never enters~$\gamma$; and $\tB$
        is not defined at~$\ha_u$.
    \end{enumerate}
\end{thm}

The values of $\lambda_{m/n}^\pm$, $\lambda_{m/n}^\ast$, and $\lambda_q$ can be
specified in terms of the kneading sequences of the corresponding tent maps (see
for example~\cite{prime}), but this will not be necessary here.

\begin{rem} \label{rem:height-intervals-discont} 
    Using Lemma~\ref{lem:tB-facts}(d) and~(e), it follows from
    Theorem~\ref{thm:height-intervals-rational}(c) that if $q_\lambda = m/n$ then
    $\be\colon S\to \hI$ has exactly~$n$ discontinuities, and $\cE\colon
    S\setminus\{a,b\}\to \ssp$ is constant on each component of the complement of
    these discontinuities. (If $\lambda =
    \lambda_{m/n}^-$ then the $\tB$-orbit of $f(a)_\ell$ is periodic, while if
    $\lambda>\lambda_{m/n}^-$ then it is a finite orbit segment.)
\end{rem}

We can use Theorems~\ref{thm:height-intervals-rational}
and~\ref{thm:height-irrational} to classify tent maps: first, according as their
height is rational or irrational; and second, in the rational case, according as
the parameter~$\lambda$ is an endpoint of the rational height interval, the special
interior value for which $\tB^n(a)=c_u$, or a general interior value.

\begin{defns}[Type of tent map: irrational, rational, endpoint, NBT, general]
    \label{defn:tent_type}
    We say that~$f_\lambda$ is of \emph{irrational} or \emph{rational} type
    according as $q_\lambda$ is irrational or rational. In the rational case
    with $q_\lambda = m/n$, we say that~$f_\lambda$ is of \emph{endpoint} type
    if $\lambda = \lambda_{m/n}^{\pm}$; of \emph{NBT} type if
    $\lambda=\lambda_{m/n}^\ast$; and of \emph{general} type otherwise.
\end{defns}

\section{The rational post-critically infinite case}
    \label{sec:rat_PCI}

Throughout this section we assume that~$f$ is of rational type with height $q=m/n$,
and that the critical orbit is not periodic or pre-periodic. We will show that the
corresponding sphere homeomorphism $F\colon\Sigma\to\Sigma$ is a measurable
pseudo-Anosov map. (In the case where~$f$ is post-critically finite, $F$ is a
generalized pseudo-Anosov map, as proved in~\cite{gpa}: see
Section~\ref{sec:rat-PCF}.)

The path that we take is as follows:
\begin{itemize}
    \item The map $g\colon\hI\to\Sigma$ which semi-conjugates $\hf$ to $F$ sends
    most of the $\pi_0$-fibers $\fib{x}$ in~$\hI$ --- those for which~$x$ is not in
    the grand orbit of~$c$ --- to arcs in~$\Sigma$. In Section~\ref{sec:semi-conj}
    we describe~$g$ (recalling results from~\cite{prime}) and prove this result
    (Theorem~\ref{thm:fiber_arc}).

    \item The streamlines of the stable turbulation are unions of these arcs,
    joined at their endpoints. The stable measure comes from the
    measures~$\alpha_x$ of Lemma~\ref{lem:alpha-measure} supported on the
    $\pi_0$-fibers. In Section~\ref{sec:stab-fol} we define $(\cT^s,
    \nu^s)$ (Definition~\ref{defn:stab_fol}) and show that~$F$ contracts $\nu^s$ by
    a factor~$\lambda$ (Lemma~\ref{lem:stable-action}) and that the turbulation is
    tame (Lemma~\ref{lem:stable-tame}).

    \item The streamlines of the unstable turbulation are the $g$-images of the
    globally leaf regular path components of~$\hI$
    (Definition~\ref{defn:glob_leaf_reg}): recall from Theorem~\ref{thm:typical}
    that the union of these has full measure --- we omit countably many of them, on
    which the identifications induced by~$g$ are non-trivial. The unstable measure
    comes from Lebesgue measure on the components of the $0$-flat decompositions
    (Lemma~\ref{lem:decompose-glr}) of these path components. In
    Section~\ref{sec:unstab-fol} we define $(\cT^u,\nu^u)$
    (Definition~\ref{defn:unstab-fol}) and show that $F$ expands $\nu^u$ by a
    factor~$\lambda$ (Lemma~\ref{lem:unstab-action}) and that the turbulation is
    tame (Lemma~\ref{lem:unstable-tame}).

    \item In Section~\ref{sec:mpa-complete} we complete the proof that~$F$ is
    a measurable pseudo-Anosov map, by proving that the two measured turbulations
    are transverse and full.
\end{itemize}

\subsection{The semi-conjugacy \texorpdfstring{$g\colon\hI\to\Sigma$}{g}}
\label{sec:semi-conj}

We now describe the identifications on $\hI$ realized by the semiconjugacy
$g\colon\hI\to\Sigma$, and show that, under these identifications, the
$\pi_0$-fibers of points $y\in I$ which are not on the grand orbit
\[
    \GO(c) = \{x\in I\,:\,f^m(x)=f^n(c) \text{ for some }m, n\ge 0\}     
 \] 
of~$c$ become arcs in~$\Sigma$.

\begin{notn}[$\Pi$, $\bP$]
    We write $\Pi = S\setminus \bigcup_{r\ge 0}B^{-r}(\ingam)$, the set of points
    whose $B$-orbits never enter $\ingam = (\ha_u, a)$, and $\bP=\be(\Pi)\subset
    \hI$.
\end{notn}
\begin{lem} \label{lem:rat-don't-enter-gamma}
    $\Pi$ is a period~$n$ orbit of~$B$ which is disjoint from~$\gamma$, and $\bP$
    is a period~$n$ orbit of~$\hf$ contained in the set of extreme elements
    of~$\hI$.
\end{lem}
\begin{proof}
    Theorem~\ref{thm:height-intervals-rational}(d) states that~$\Pi$ is a
    period~$n$ orbit of~$B$. By definition it is disjoint from $\ingam$, so we need
    only show that the endpoints~$a$ and $\ha_u$ of $\gamma$ are not in~$\Pi$.
    Since~$f$ is not post-critically finite we have $\lambda\in\Int(\Lambda_{m/n})$
    by Theorem~\ref{thm:height-intervals-rational}(b), and hence $B^n(a) =
    \tB^n(a)\in\ingam$ by Theorem~\ref{thm:height-intervals-rational}(c), so that
    $a\not\in\Pi$. Since $B(a)=B(\ha_u)=f(a)_\ell$, we have $\ha_u\not\in\Pi$ also.

    It follows from~\eqref{eq:model_action} that $\bP=\be(\Pi)$ is a period~$n$
    orbit of $\hf$, which is clearly contained in the image of $\be$, the set of
    extreme elements of~$\hI$.
\end{proof}

\begin{notn}[$z$, $\cZ$]
    We write $z=f^n(a)$, so that $\tB^n(a)=z_u$, and $\cZ = \{z, \hz\}$.
\end{notn}

The following key lemma is a translation into the language of this paper of
remark~5.17(e) of~\cite{prime}.

\begin{lem} \label{lem:rat-equiv-class}
    The semi-conjugacy $g\colon\hI\to\Sigma$ is the quotient map of the equivalence
    relation on~$\hI$ with the following non-trivial equivalence classes:
    \begin{enumerate}[(EI)]
        \item $\{
            \hf^r\left(\be(x_u)\right), \hf^r\left(\be(\hx_u)\right)
            \}$,
            for each $x\in(a,c)\setminus \cZ$ and each $r\in\Z$.

        \item $\{
            \hf^r\left(\be(\hz_u)\right),
            \hf^{r+n}\left(\be(a)\right),
            \hf^{r+n}\left(\be(\ha_u)\right)
        \}$, 
        for each $r\in\Z$.

        \item The period~$n$ orbit~$\bP$.
    \end{enumerate}
\end{lem}

We first use this result to understand the identifications induced by~$g$ on
individual fibers of~$\hI$. Note that the condition $y\in I\setminus\GO(c)$ implies
that $y\not\in\PC$, so that the fiber~$\fib{y}$ is totally ordered, and in
particular consecutive elements are well defined.

\begin{lem}[Identifications within fibers] \label{lem:idents_in_fibers}
    Let $y\in I \setminus\GO(c)$. 
    \begin{enumerate}[(a)]
        \item Non-extreme elements of~$\fib{y}$ are only identified with other
        (non-extreme) elements of~$\fib{y}$.

        \item Distinct non-extreme elements $\bx$ and $\bx'$ of~$\fib{y}$ are
        identified if and only if they are consecutive.

        \item The extreme elements $\be(y_u)$ and $\be(y_\ell)$ of~$\fib{y}$ are
        not identified with each other: that is, $g(\be(y_u))\not=g(\be(y_\ell))$.
    \end{enumerate}
\end{lem}
\begin{proof}
    \begin{enumerate}[(a)]
        \item Since $y\not\in\GO(c)$, equivalence classes of type~(EII) are
        disjoint from~$\fib{y}$. Moreover, if $r\ge 0$ then for any
        $x\in(a,c)\setminus\cZ$ we have $\hf^{-r}(\be(x_u)) = \be(\tB^{-r}(x_u))$
        and $\hf^{-r}(\be(\hx_u)) = \be(\tB^{-r}(\hx_u))$ by
        Lemma~\ref{lem:tB-facts}(b), so these are extreme elements, as are the
        points of~$\bP$. Therefore the non-trivial
        equivalence classes which contain non-extreme elements of $\fib{y}$ are
        precisely $\{\hf^r\left(\be(x_u)\right), \hf^r\left(\be(\hx_u)\right) \}$
        for those $x\in(a,c)\setminus \cZ$ and $r > 0$ which satisfy $f^r(x)=y$;
        and, since $f^r(x)=f^r(\hx)$, such equivalence classes consist of two
        points of~$\fib{y}$.

        \item Suppose that $\bx$ and $\bx'$ are consecutive. Let $r>0$ be least such that
        $x_r\not=x_r'$. By Lemma~\ref{lem:preserve_consecutive}(b) and~(c), we have
        $\hf^{-r}(\bx)=\be(x_{r,u})$ and $\hf^{-r}(\bx') = \be(x'_{r,u})$. Since
        $x'_r = \widehat{x_r}$ (both have image $x_{r-1}$), and these are not equal
        to $a$, $c$, $z$, or $\hz$ (since $y\not\in\GO(c)$), it follows
        from~(EI) that $g(\bx)=g(\bx')$ as required.

        For the converse, if $g(\bx)=g(\bx')$ then by the proof of~(a) there are
        $x\in(a,c)\setminus \cZ$ and $r>0$ such that $\{\bx, \bx'\} =
        \{\hf^r(\be(x_u)), \hf^r(\be(\hx_u))\}$. These are consecutive by
        Lemma~\ref{lem:preserve_consecutive}(d) and~(a).

        \item The two extreme elements cannot both belong to the periodic
        orbit~$\bP$, since if $y<\ha$ then $y_u\in\ingam$ (so that $y_u\not\in\Pi$,
        and hence $\be(y_u)\not\in\bP$); while if $y>\ha$ then $f(y)<f(a)<\ha$ by
        Lemma~\ref{lem:order_fa_p_ha}, and hence $B(y_\ell)=f(y)_u\in\ingam$ (so
        that $y_\ell\not\in\Pi$, and hence $\be(y_\ell)\not\in\bP$).

        It follows from Lemma~\ref{lem:rat-equiv-class} and $y\not\in\GO(c)$ that
        we could only have $g(\be(y_u)) = g(\be(y_\ell))$ if there were some
        $x\in(a,c)$ and $r\in\Z$ with $\{\be(y_u), \be(y_\ell)\}$ equal to
        $\{\hf^r\left(\be(x_u)\right), \hf^r\left(\be(\hx_u)\right)\}$. However if
        $r>0$ then $\hf^r\left(\be(x_u)\right)$ and $\hf^r\left(\be(\hx_u)\right)$
        are non-extreme elements of $\fib{f^r(x)}$. If $r = -s\le 0$, on the other
        hand, then we would have $\{y_u, y_\ell\} = \{\Theta^{-s}(x_u),
        \Theta^{-s}(\hx_u)\}$ as in the proof of~(a), so that~$s$ would
        be both least with $B^s(y_u)\in\ingam$, and least with
        $B^s(y_\ell)\in\ingam$. This is impossible, since, as above, if $y<\ha$
        then $y_u\in\ingam$ but $y_\ell\not\in\ingam$; while if $y>\ha$ then
        $B(y_\ell)\in\ingam$ but $B(y_u) = f(y)_\ell\not\in\ingam$.
    \end{enumerate}
\end{proof}

\begin{thm}
    \label{thm:fiber_arc}
    Let $y\in I\setminus\GO(c)$. Then $g(\fib{y})$ is an arc in~$\Sigma$.
\end{thm}
\begin{proof}
    Let $H\colon\ssp\to [0,1]$ be the order-preserving homeomorphism onto the
    middle thirds Cantor set defined by
    \[
        H(s) = \sum_{r=0}^\infty \frac{2\varepsilon_r}{3^r}, \quad \text{ where
        }
        \varepsilon_r = \left(\sum_{i=0}^r s_i\right) \bmod{2}.
    \]
    Then $H\circ J\colon\fib{y}\to[0,1]$ is an order-preserving embedding of the
    Cantor set~$\fib{y}$ into the interval. Therefore, by
    Lemma~\ref{lem:idents_in_fibers}(b) and~(c), $g(\fib{y})$ is homeomorphic to
    the space obtained by collapsing the complementary intervals of a Cantor subset
    of the interval: that is, to an arc.
\end{proof}

\begin{notn}[$\cA_y$] \label{notn:cay}
    For each $y\in I\setminus\GO(c)$, we write $\cA_y$ for the arc
    $g(\fib{y})\subset\Sigma$.
\end{notn}

\begin{lem} \label{lem:cont-inverse}
    For $y\in I \setminus\GO(c)$, let $\cB_y$ be the subset of $\cA_y$ consisting
    of points with a unique $g$-preimage. Then $g^{-1}\colon\cB_y\to\fib{y}$ is
    continuous.
\end{lem}

\begin{proof} 
    Suppose for a contradiction that there is a convergent sequence $x_n\to x$
    in~$\cB_y$ with $g^{-1}(x_n)\not\to g^{-1}(x)$. Since~$\fib{y}$ is compact, we
    can assume by taking a subsequence that $g^{-1}(x_n)\to\bx \not=g^{-1}(x)$.
    Since~$g$ is continuous, it follows that $x_n\to g(\bx)$, so that $x=g(\bx)$.
    The distinct elements $\bx$ and $g^{-1}(x)$ of~$\fib{y}$ are therefore both
    mapped to~$x$ by~$g$, contradicting the assumption that $x\in\cB_y$.
\end{proof}

\subsection{The stable turbulation}
\label{sec:stab-fol}

The streamlines of the stable turbulation will be formed from unions of the arcs
$\cA_y$, which are joined at their endpoints by the identifications of
Lemma~\ref{lem:rat-equiv-class}. In order to ensure that the endpoints are
identified in pairs, we need to exclude the identifications of types~(EII) and
(EIII) in Lemma~\ref{lem:rat-equiv-class}. Type~(EII) identifications are excluded
in any case, since $y\not\in\GO(c)$; to deal with type~(EIII), we will also require
that $y$ is not in the grand orbit of $\pi_0(\bP)$.

\begin{notn}[$P$, $V$, $\mu$] \label{notn:PV}
    Write $P=\pi_0(\bP)$ and let $V= I \setminus (\GO(c) \cup \GO(P))$, a totally
    $f$-invariant subset of~$I$ with countable complement.

    Denote by~$\mu$ the OU ergodic $F$-invariant probability measure
    $g_*(\hmu)$ on~$\Sigma$, where $\hmu$ is the measure on~$\hI$ defined in
    Section~\ref{sec:measures}. No confusion will arise from the same symbol's
    having been used for the absolutely continuous $f$-invariant measure on~$I$.
\end{notn}

If $y\in V$ then, by definition of~$\bP$, there is some $r\ge 0$ with $B^r(y_u)\in
(\ha_u, a)$: say \mbox{$B^r(y_u)=x_u$}. Then $x\not\in\cZ$, since $y\not\in\GO(c)$,
and hence, by Lemma~\ref{lem:rat-equiv-class}, $\be(y_u) =
\be(B^{-r}(x_u))$ is identified with $\be(B^{-r}(\hx_u))$, and only with this
point.

That is, for every $y\in V$ there is a unique $x\in V\setminus\{y\}$ such that
$g(\be(y_u))$ is an endpoint of $\cA_x$; and, by an identical argument, there is a
unique $x'\in V\setminus\{y\}$ such that $g(\be(y_\ell))$ is an endpoint of
$\cA_{x'}$.

It follows that for every $y\in V$, there is a bi-infinite sequence
$(y_i)_{i\in\Z}$ in $V$ with $y_0=y$ and $y_i\not=y_{i+1}$ for each~$i$, unique up
to reversal, such that each $\cA_{y_i}$ shares one endpoint with $\cA_{y_{i-1}}$
and the other with~$\cA_{y_{i+1}}$.

\begin{notn}[$\ell_y$]
    For each $y\in V$ we write $\ell_y = \bigcup_{i\in\Z} \cA_{y_i}$, which is
    either an immersed line, or, if the sequence~$(y_i)$ repeats, a simple closed
    curve. We shall see in Lemma~\ref{lem:stable-dense} that the latter is
    impossible.
\end{notn}

The $\ell_y$ will be the streamlines of the stable turbulation, and we now define
the stream measures.

\begin{defn}[$\nu^s_{\ell_y}$] \label{defn:def-stab-meas}
    We define a Borel measure $\nu^s_{\ell_y}$ on $\ell_y$ by 
    \[
        \nu^s_{\ell_y} = \sum_{i\in K} g_*(\alpha_{y_i}),
    \]
    where the $\alpha_{y_i}$ are the measures of Lemma~\ref{lem:alpha-measure},
    supported on $\fib{y_i}$; and where $K=\Z$ if the $y_i$ are all distinct,
    and otherwise $K=\{0, \dots, r-1\}$ if $r>0$ is least with $y_r=y_0$.
    Notice from Lemma~\ref{lem:alpha-measure} that $\nu^s_{\ell_y}(\cA_{y_i}) =
    \varphi(y_i)$.
\end{defn}

The next lemma ensures that the streamlines and stream measures satisfy the
conditions of Definitions~\ref{defn:turbulation}. 

\begin{lem} \mbox{}
    \begin{enumerate}[(a)]
        \item $\mu\left(\bigcup_{y\in V}\ell_y\right)=1$.

        \item The measures $\nu^s_{\ell_y}$ are OU and assign finite measure to
        closed arcs.
    \end{enumerate}
\end{lem}
\begin{proof}
    Write $X=\bigcup_{y\in V}\ell_y$, and note that $X=g(\pi_0^{-1}(V))$. Since~$V$
    has countable complement, Theorem~\ref{thm:disintegrate} gives
    \[
        \mu(X^s) = \hmu(\pi_0^{-1}(V)) 
            = \int_V\alpha_x(\fib{x})dm(x) 
            = \int_I\alpha_x(\fib{x})dm(x)
            = \hmu(\hI)
            = \mu(\Sigma).
    \]
    $\nu^s_{\ell_y}$ is OU since the same is true of each $\alpha_{y_i}$
    (Remarks~\ref{rems:alpha-x}(a)). Every closed arc in~$\ell_y$ is contained in a
    finite union of the $\cA_{y_i}$, and so has finite measure.
\end{proof}

Until we show in Lemma~\ref{lem:stable-dense} that all of the~$\ell_y$ are immersed
lines, we temporarily relax the definition of measured turbulation to allow
streamlines which are simple closed curves.

\begin{defn}[$\cT^s, \nu^s$] \label{defn:stab_fol}
        We define the stable turbulation $(\cT^s, \nu^s)$ on~$\Sigma$ to be the
        measured turbulation whose streamlines are the distinct
        immersed lines and simple closed curves~$\ell_y$, for $y\in V$, with
        measure $\nu^s_{\ell_y}$ on the streamline~$\ell_y$.
\end{defn}

\begin{lem} \label{lem:stable-action}
    $F(\cT^s, \nu^s) = (\cT^s, \lambda\nu^s)$.
\end{lem}

\begin{proof}
    $F$ sends streamlines to streamlines since $\hf$ sends $\pi_0$-fibers into
    $\pi_0$-fibers and respects the identifications
    of Lemma~\ref{lem:rat-equiv-class}.

    If~$A$ is a Borel subset of a streamline~$\ell_y$ then $\nu^s_{\ell_y}(A) =
    \sum_{i\in K} \alpha_{y_i}(g^{-1}(A))$, where $K=\Z$ or $K=\{0,\dots,
    r-1\}$ as in Definition~\ref{defn:def-stab-meas}. On the other hand
    \begin{eqnarray*}
        \nu^s_{\ell_{f(y)}}(F(A)) &=& \sum_{i\in K} \alpha_{f(y_i)}(g^{-1}(F(A))) \\
        &=& \sum_{i\in K} \alpha_{f(y_i)}(\hf(g^{-1}(A))) \\ &=& \lambda^{-1}
        \sum_{i\in K} \alpha_{y_i}(g^{-1}(A)) = \lambda^{-1}\nu^s_{\ell_y}(A),
    \end{eqnarray*}
    where the first equality comes from $F(\cA_{y_i}) \subset \cA_{f(y_i)}$;
    the second is because~$g$ is almost everywhere invertible; and the third
    comes from Remarks~\ref{rems:alpha-x}(b). 

    Thus $F_*(\nu^s_{\ell_y}) = \lambda\nu^s_{\ell_{f(y)}}$ as required.
\end{proof}

\begin{lem} \label{lem:stable-dense}
    Every streamline of $(\cT^s, \nu^s)$ is dense in~$\Sigma$. In particular, no
    streamlines are simple closed curves, so that $(\cT^s, \nu^s)$ is a measured
    turbulation.
\end{lem}

\begin{proof}
    Let $U\subset\Sigma$ be non-empty and open, and $\ell_y$ be a streamline: we
    shall show that they intersect.

    $g^{-1}(U)$ is open in~$\hI$, so $\pi_0(g^{-1}(U))$ contains an interval
    $J\subset I$ (one way to see this is from Theorem~\ref{thm:typical} and
    Lemma~\ref{lem:decompose-glr}: $g^{-1}(U)$ contains an open arc in~$\hI$, which
    has a $0$-flat decomposition). Tent maps with slope $\lambda>\sqrt{2}$ are
    locally eventually onto, so that there is some~$n>0$ with $f^n(J) = I$. In
    particular, there is some~$x\in J$ with $f^n(x)=f^n(y)$: thus
    $\hf^n(\fib{x})\subset
    \fib{f^n(y)}$ and $\hf^n(\fib{y})\subset
    \fib{f^n(y)}$, so that both $\fib{x}$ and $\fib{y}$ are contained in
    $\hf^{-n}(\fib{f^n(y)})$. Applying~$g$, both $\cA_x$ and $\cA_y$ are contained
    in $F^{-n}(\cA_{f^n(y)})$. Since $F^{-n}(\cA_{f^n(y)})$ is a stream arc
    containing $\cA_y$, the streamline which contains it is~$\ell_y$. Thus
    $\cA_x\subset\ell_y$: but since $x\in J \subset\pi_0(g^{-1}(U))$ we have
    \mbox{$\cA_x\cap U\not=\emptyset$}.
\end{proof}

Therefore $(\cT^s, \nu^s)$ is a measured turbulation with dense streamlines, which
satisfies $F(\cT^s, \nu^s) = (\cT^s, \lambda\nu^s)$. It remains to show that it is
tame (Definition~\ref{defn:tame}).

\begin{lem} \label{lem:stable-tame}
    $(\cT^s, \nu^s)$ is tame.
\end{lem}

\begin{proof}
    Let~$\rho$ be a metric on~$\Sigma$ inducing its topology, and let~$\epsilon>0$.
    Let~$\xi>0$ be small enough that $d(\bx, \by)<\xi \implies \rho(g(\bx), g(\by)) < \epsilon/2$, where~$d$ is the standard metric of Definitions~\ref{defns:il}.

    Let~$k>0$ be a lower bound for the density $\varphi$ of
    Section~\ref{sec:measures} on $V\subset I\setminus\PC$. Let~$r>0$ be an integer
    large enough that $1/2^r < \xi$, and set $\delta = k/\lambda^r$. We shall show
    that any stream arc $[u,v]$ with measure less than~$\delta$ satisfies
    $\rho(u,v)<\epsilon$, establishing tameness of the turbulation.

    Recall from Definitions~\ref{defs:cylinder} that, given our fixed integer
    $r>0$, each fiber~$\fib{x}$ of~$\hI$ is partitioned into a finite number of
    level~$r$ cylinder sets
    \[
        \fib{x, y_1, \dots, y_r} = \{\bx\in \fib{x}\,:\, x_i = y_i \text{ for }
            1\le i \le r\},
    \]
    where $f(y_1)=x$ and $f(y_i)=y_{i-1}$ for $2\le i\le r$. If $x\in V$, so that
    $x\not\in\PC$ and itineraries are well-defined on~$\fib{x},$ then each of these
    cylinder sets corresponds to a choice of the first~$r$ symbols of the itinerary
    of~$\bx$: in particular, given two distinct level~$r$ cylinder subsets $C_1$
    and $C_2$ of $\fib{x}$, either $J(\bx)\prec J(\bx')$ for every $\bx\in C_1$ and
    $\bx'\in C_2$, or $J(\bx')\prec J(\bx)$ for every $\bx\in C_1$ and $\bx'\in
    C_2$. It follows that the image $g(\fib{x, y_1, \dots, y_r})$ of a cylinder set
    is a stream arc, which we denote $\cA_{x, y_1, \dots, y_r}$ and call a
    \emph{level~$r$ cylinder arc}. The stream arc $\cA_x$ is a union of these
    level~$r$ cylinder arcs, which intersect only at their endpoints.

    By Lemma~\ref{lem:alpha-measure} the arc $\cA_{x, y_1, \dots, y_r}$ has measure
    $\varphi(y_r)/\lambda^r \ge k/\lambda^r=\delta$, so that any stream arc of
    measure less than~$\delta$ is contained in the union of at most two level~$r$
    cylinder arcs (possibly in different fibers of~$\hI$). Now the metric diameter
    of a level~$r$ cylinder set is bounded above by $\sum_{i=r+1}^\infty |b-a|/2^i
    < 1/2^r < \xi$, so that the metric diameter of a level~$r$ cylinder arc is
    bounded above by $\epsilon/2$. Therefore the endpoints $[u,v]$ of a stream arc
    contained in the union of at most two level~$r$ cylinder arcs satisfy $\rho(u,
    v)<\epsilon$ as required.

\end{proof}

\subsection{The unstable turbulation}
\label{sec:unstab-fol}

The streamlines of the unstable turbulation will be $g$-images of the globally leaf
regular path components (Definition~\ref{defn:glob_leaf_reg}) of~$\hI$. We will
omit the countably many path components which are involved in identifications ---
that is, which contain points~$\bx$ for which $g^{-1}(\{g(\bx)\}) \not=
\{\bx\}$.

Consider the identifications of type (EI) from Lemma~\ref{lem:rat-equiv-class}. By
Lemma~\ref{lem:tB-facts}(d), $\be\colon S\to\hI$ is discontinuous at~$y\in S$ if
and only if $y\in\orb(f(a)_\ell, \tB)$. Since there is exactly one point of this
orbit in $(\ha_u, a)$, namely $\tB^n(a)=z_u$, the set of points
$\{\be(x_u)\,:\,x\in(a, \ha)\}$ is contained in at most two path components
of~$\hI$. The union~$\Gamma_1$ of the bi-infinite $\hf$-orbits of these path
components is therefore an $\hf$-invariant set which contains all of the points
of~$\hI$ involved in identifications of type~(EI).

Likewise, the union $\Gamma_2$ of the bi-infinite $\hf$-orbits of the path
components of $\be(\hz_u)$, $\be(a)$, and $\be(\ha_u)$ is an $\hf$-invariant set
which contains all points involved in identifications of type~(EII); and the
union~$\Gamma_3$ of the path components of the points of~$\bP$ is $\hf$-invariant
and contains all points involved in identifications of type~(EIII).

Let
\begin{equation} \label{eq:Ku}
    K^u = \{\bx\in \hI\,:\,\bx \text{ is globally leaf regular}\} \setminus
    (\Gamma_1\cup\Gamma_2\cup\Gamma_3).
\end{equation}
Then $K^u$ is an $\hf$-invariant union of globally leaf regular path components
of~$\hI$ with $g^{-1}(\{g(\bx)\}) = \{\bx\}$ for all~$\bx\in K^u$.

Since $\Gamma_1\cup\Gamma_2\cup\Gamma_3$ contains only countably many of the path
components of~$\hI$, each of which has $\hmu$-measure zero by
Remark~\ref{rem:glr-zero-measure}, it follows from Theorem~\ref{thm:typical} that
$\hmu(K^u)=\hmu(\hI)$.

Let~$\Gamma$ be a path component of~$K^u$. Since~$\Gamma$ is an immersed line and
$g$ is injective on~$K^u$, its image $\ell=g(\Gamma)$ is also an immersed line.
Moreover, $\ell$ is dense in~$\Sigma$, since~$\Gamma$ is dense in~$\hI$ and~$g$ is
continuous and surjective. These dense lines~$\ell$ will be the streamlines of the
unstable turbulation.

We define an OU measure $\nu^u_\ell$ on~$\ell$ using the $0$-flat
decomposition $\Gamma=\bigcup_{i\in\Z} \Gamma_i$ of Lemma~\ref{lem:decompose-glr}:
if~$A$ is a Borel subset of~$\ell$, then
\begin{equation} \label{eq:unstab-meas}
    \nu^u_\ell(A) = \sum_{i\in\Z} m(\pi_0(g^{-1}(A) \cap \Gamma_i)),
\end{equation}
where~$m$ is Lebesgue measure on~$I$.

Every closed stream arc in~$\ell$ is contained in a finite union of the arcs
    $g(\Gamma_i)$ coming from the $0$-flat decomposition of~$\Gamma$, and so has
    finite measure by~\eqref{eq:unstab-meas}.

\begin{defn}[$\cT^u, \nu^u$] \label{defn:unstab-fol}
    We define the unstable turbulation $(\cT^u, \nu^u)$ to be the measured
    turbulation on~$\Sigma$ whose streamlines are the $g$-images of the path
    components of $K^u$, with measure $\nu^u_\ell$ on the streamline~$\ell$.
\end{defn}

\begin{lem}\label{lem:unstab-action}
    $F(\cT^u, \nu^u) = (\cT^u, \lambda^{-1}\nu^u)$.
\end{lem}
\begin{proof}
    Let~$\ell = g(\Gamma)$ be a streamline of $\cT^u$. Then $F(\ell) =
    g(\hf(\Gamma))$ is also a streamline.

    Let~$A$ be a Borel subset of~$\ell$. Write~$\Gamma=\bigcup_{i\in\Z}\Gamma_i$
    for the $0$-flat decomposition of Lemma~\ref{lem:decompose-glr}, and for
    each~$i\in\Z$ define measurable subsets of~$A$ by
    \begin{eqnarray*}
        A_i^L &=& \{x\in A\,:\, g^{-1}(x)\in\Gamma_i \text{ and
        }\pi_0(g^{-1}(x))<c\}    \text{ and}\\
         A_i^R &=& \{x\in A\,:\, g^{-1}(x)\in\Gamma_i \text{ and
        }\pi_0(g^{-1}(x))>c\}.
    \end{eqnarray*}
    It is enough to show that $\nu^u_{F(\ell)}(F(A_i^Q)) = \lambda
    \nu^u_\ell(A_i^Q)$ for each~$i$ and each $Q\in\{L,R\}$.

    Now $\nu^u_\ell(A_i^Q) = m(\pi_0(g^{-1}(A_i^Q)))$. Since $g^{-1}(A_i^Q)$ is
    contained in a $0$-flat arc disjoint from $\fib{c}$, its image
    $\hf(g^{-1}(A_i^Q)) = g^{-1}(F(A_i^Q))$ is contained in a $0$-flat arc.
    Therefore, if $\bigcup_{i\in\Z}\delta_i$ is the $0$-flat decomposition of
    $g^{-1}(F(\ell))$, there is some~$j$ for which
    $g^{-1}(F(A_i^Q))\subset~\delta_j$. It follows that
    \begin{eqnarray*}
        \nu^u_{F(\ell)}\big(F\big(A_i^Q\big)\big) &=&
            m\big(\pi_0\big(g^{-1}\big(F\big(A_i^Q\big)\big)\big)\big)
            \\ &=&
            m\big(\pi_0\big(\hf\big(g^{-1}\big(A_i^Q\big)\big)\big)\big)
            \\ &=&
            m\big(f\big(\pi_0\big(g^{-1}\big(A_i^Q\big)\big)\big)\big)
            \\ &=& \lambda m\big(\pi_0\big(g^{-1}\big(A_i^Q\big)\big)\big)
            \\ &=& \lambda
            \nu^u_\ell\big(A_i^Q\big) \quad\text{ as required.}
    \end{eqnarray*}
    Here the fourth equality follows from the fact that $\pi_0(g^{-1}(A_i^Q))$
    is contained in $[a,c]$ or in $[c,b]$, according as $Q=L$ or $Q=R$, on
    which~$f$ expands uniformly by a factor~$\lambda$.
\end{proof}

Therefore $(\cT^u, \nu^u)$ is a measured turbulation with dense streamlines, which
satisfies $F(\cT^u, \nu^u) = (\cT^u, \lambda^{-1}\nu^u)$. It remains to show that
it is tame (Definition~\ref{defn:tame}).

\begin{lem} \label{lem:unstable-tame}
    $(\cT^u, \nu^u)$ is tame.
\end{lem}

\begin{proof}
    If~$\Gamma$ is a $0$-flat arc in~$\hI$ over the interval~$J$, then
    $m(\pi_0(\Gamma)) = |J|$, while the metric diameter of~$\Gamma$ is
    $\frac{2\lambda}{2\lambda - 1}|J| < 2|J|$ (see for example section~4.1
    of~\cite{typical}).

    Let~$\rho$ be a metric on~$\Sigma$ inducing its topology and let $\epsilon>0$.
    Pick $\xi>0$ small enough that $d(\bx, \by)<\xi \implies \rho(g(\bx),
    g(\by))<\epsilon$.

    Let $\delta = \xi/2$, and consider a stream arc $[x,y]\subset\ell$ with
    $\nu_\ell^u([x,y])<\delta$. Then $g^{-1}([x,y])$ is a union of finitely many
    $0$-flat arcs which intersect only at their endpoints, arising from the
    $0$-flat decomposition of $g^{-1}(\ell)$. By~\eqref{eq:unstab-meas} and the
    remarks above, the metric diameter of~$g^{-1}([x,y])$ is less than $2\delta =
    \xi$, so that $\rho(x,y)<\epsilon$ as required.

\end{proof}

\subsection{\texorpdfstring{$F$}{F} is a measurable pseudo-Anosov map}
\label{sec:mpa-complete}

We have defined tame measured turbulations $(\cT^s, \nu^s)$ and $(\cT^u,
\nu^u)$ on~$\Sigma$, having dense streamlines, with the properties that $F(\cT^s,
\nu^s) = (\cT^s, \lambda\nu^s)$ and $F(\cT^u, \nu^u) = (\cT^u,
\lambda^{-1}\nu^u)$. In this section we show that these turbulations
satisfy the remaining conditions of Definition~\ref{defn:mpa}, thereby establishing
that~$F\colon\Sigma\to\Sigma$ is measurable pseudo-Anosov. That is:
\begin{itemize}
    \item the two turbulations are transverse (Lemma~\ref{lem:transverse}); and
    \item the turbulations are full: there is a countable collection of
    compatible tartans whose intersections cover a full measure subset of~$\Sigma$,
    and every non-empty open subset of~$\Sigma$ contains a positive measure
    compatible tartan (Lemma~\ref{lem:tartan_exists}).
\end{itemize}

\begin{lem} \label{lem:transverse}
    Every intersection of $\cT^s$ and $\cT^u$ is topologically transverse.
\end{lem}
\begin{proof}
    Let $g(\bx)$ be a point of intersection of a streamline $\ell=g(\Gamma)$
    of~$\cT^u$, and a streamline~$\ell'$ of $\cT^s$. The streamline~$\ell'$ is
    locally contained in the arc $\cA_{x_0}$ for some $x_0\in V$ ($g(\bx)$ cannot
    be an endpoint of~$\cA_{x_0}$, since points of~$\ell$ have a unique
    $g$-preimage). The arc $\cA_{x_0}$ separates any sufficiently small disk
    neighborhood of $g(\bx)$ into points $g(\by)$ with $y_0 < x_0$, and points
    $g(\by)$ with $y_0 > x_0$. It is therefore enough to show that $\pi_0|_\Gamma$
    is locally injective at~$\bx$. This is a consequence of
    Lemma~\ref{lem:decompose-glr}: since~$x_0\not\in\GO(c)$, we have $x_r\not=c$
    for all~$r$, and hence $\bx$ is in the interior of one of the
    components~$\Gamma_i$ of the $0$-flat decomposition of~$\Gamma$ by
    Remark~\ref{rem:0-flat-itin}.
\end{proof}

\begin{lem} \label{lem:tartan_exists}
    The pair $(\cT^s, \nu^s)$, $(\cT^u, \nu^u)$ is full.
\end{lem}
\begin{proof} 
    We will prove the existence of a positive measure regular compatible
    tartan~$R$. It follows from Lemma~\ref{lem:image_tartan} and
    Remark~\ref{lem:regular-propagate} that $F^i(R)$ is a regular compatible tartan
    for each $i\ge 0$. Since~$F$ is ergodic this countable collection of tartans
    has $\mu(\bigcup_{i\ge 0}F^i(R)^\pitchfork) = 1$, and fullness follows from
    Lemma~\ref{lem:part_b_from_a}.

    By Theorem~\ref{thm:+ve-0-box} there is a $0$-box~$B$ with $\hmu(B)>0$: recall
    that~$B$ is a disjoint union of $0$-flat arcs over some interval $J\subseteq
    I$. Without loss of generality we can assume that the arcs of~$B$ are all
    contained in the full measure subset $K^u$ of~\eqref{eq:Ku}, so that their
    images under~$g$ are contained in streamlines of~$\cT^u$, and~$g$ is injective
    on~$B$.

    Recall (Notation~\ref{notn:PV}) that $V=I\setminus(\GO(c)\cup\GO(P))$, and
    write $J'=J\cap V$, which has countable complement in~$J$. Let $R^u = g(B)$, a
    disjoint union of unstable stream arcs, each with finite stream measure~$m(J)$;
    and let
    $
        R^s = \bigcup_{y\in J'} \cA_y
    $
    (recall Notation~\ref{notn:cay}), a disjoint union of stable stream arcs with
    bounded stream measures $\nu_{\ell_y}^s(\cA_y) = \varphi(y)$.

    By definition, $\pi_0$ is a bijection from each arc of~$B$ onto~$J$, so
    that each of these arcs intersects each fiber~$\fib{y}$ (with $y\in J$)
    exactly once. Since $g^{-1}(x)$ is a point for each $x\in R^u$, it follows
    that every arc of $R^s$ intersects every arc of $R^u$ exactly once, and
    these intersections are either transverse or at endpoints
    by Lemma~\ref{lem:transverse}. Each unstable fiber~$\fu$ can be oriented in the
    direction of increasing $\pi_0\circ g^{-1}\colon \fu\to J$, and with this
    orientation crosses the stable fibers $\cA_y$ in order of increasing~$y$.
    Likewise, each stable fiber can be oriented in the direction of increasing
    itinerary, and with this orientation crosses the unstable fibers in order
    of increasing itinerary (see Remark~\ref{rem:0-flat-itin}).

    The open topological disk $\Sigma\setminus\{g(\be(c^u))\}$ contains~$R^s\cup
    R^u$, so $R=(R^s, R^u)$ is a tartan (Definition~\ref{defn:tartan}). It has positive measure by Theorem~\ref{thm:disintegrate}:
    \[
        \mu(R^\pitchfork) = \hmu(B\cap \pi_0^{-1}(J')) = \int_{J'}
                \alpha_x(B)\,dm(x) = \int_J \alpha_x(B)\, dm(x) = \hmu(B) > 0.
    \] 
    It thus remains to show that~$R$ is compatible and regular. 

    Given $y, z\in J'$, there is a holonomy map $h_{y,z}\colon \fib{y}\cap B \to
    \fib{z}\cap B$, which takes the point of $\fib{y}\cap\Gamma$ to the point of
    $\fib{z}\cap\Gamma$ for each $0$-flat arc~$\Gamma$ of~$B$. That is, since each
    $0$-flat arc has constant itinerary (Remark~\ref{rem:0-flat-itin}),
    $h_{y,z}(\bx) = L(z, J(\bx))$, where $L\colon\cV\to\hI$ is the continuous
    function of Lemma~\ref{lem:cpctKx}. In particular, $h_{y,z}$ is a homeomorphism
    (with inverse $h_{z, y}$). Moreover, if $X\subseteq \fib{y}\cap B$ is Borel,
    then
    \begin{equation} \label{eq:hol-inv}
        \alpha_y(X) = \alpha_z(h_{y,z}(X)),
    \end{equation}
    by applying Theorem~\ref{thm:holonomy_alpha} to the $0$-box consisting of those
    $0$-flat arcs of~$B$ which pass through~$X$.

    Now pick arbitrary stable and unstable fibers $\fs = \cA_y = g(\fib{y})$ and
    $\fu=g(\Gamma)$ of~$R$, where $y\in J'$ and $\Gamma$ is one of the $0$-flat
    arcs of~$B$. Define 
        \[
            \psi\colon (\fib{y}\cap B)\times (\Gamma\cap
                \pi_0^{-1}(J')) \to B \cap \pi_0^{-1}(J)
        \] 
        by $\psi(\bx, \bx')=h_{y, \pi_0(\bx')}(\bx) = L(\pi_0(\bx'), J(\bx))$,
        which is continuous by Lemma~\ref{lem:cpctKx} and
        Remark~\ref{rem:cont-itin}, and so a homeomorphism with inverse
        $\psi^{-1}(\bx)=(L(y, J(\bx)), L(\pi_0(\bx), s))$, where~$s$ is the
        constant itinerary of~$\Gamma$. The definition is made in order that
        \[
            \begin{CD}
                (\fib{y}\cap B)\times (\Gamma\cap \pi_0^{-1}(J')) @>\psi>>
                            B \cap \pi_0^{-1}(J) \\
                @VVG\times GV @ VVGV \\
                E^\fs \times E^\fu @>\psi^{\fs, \fu}>> R^\pitchfork
            \end{CD}
        \]
        commutes, where we have written $G = g|_B$, a homeomorphism by
        Lemma~\ref{lem:cont-inverse}. Therefore $\psi^{\fs, \fu}$ is a
        homeomorphism, and in particular bi-measurable. Since $G_*\hmu = \mu$ and
        $(G\times G)_* (\alpha_y\times M) = \nu_{\fs, \fu}$ by
        Definition~\ref{defn:def-stab-meas} and~\eqref{eq:unstab-meas} (where we
        have written~$M=(\pi_0|_\Gamma^{-1})_*m$), it suffices for compatibility to
        show that $\psi_*(\alpha_y\times M) = \hmu$.

        Let~$E$ be a Borel subset of $B\cap\pi_0^{-1}(J')$. Then
        \begin{eqnarray*}
            (\alpha_y\times M)(\psi^{-1}(E)) &=& \int_{\Gamma\cap \pi_0^{-1}(J)}
              \alpha_y\left(\{
                    \bx\in\fib{y}\cap B\,:\, \psi(\bx,z) \in E
                \}\right)dM(z) 
            \\
            &=& \int_{J'} \alpha_y \left(\{
                        \bx\in \fib{y}\cap B\,:\, h_{y,z}(\bx)\in E
                    \}\right) dm(z)
            \\
            &=& \int_{J'} \alpha_y(h_{z,y}(\fib{z}\cap E))\, dm(z)
            \\
            &=& \int_{J'} \alpha_z(E)\, dm(z) = \hmu(E)
        \end{eqnarray*}
        as required, where we have used, in turn, Fubini's theorem, the definition
        of~$\psi$, $h_{z,y}=h_{y,z}^{-1}$,~\eqref{eq:hol-inv}, and
        Theorem~\ref{thm:disintegrate}.

        Finally, for regularity, let $x = g(\bx)\in R^\pitchfork$ and~$U$ be a
        neighborhood of~$x$. Write $V=g^{-1}(U)$, a neighborhood of~$\bx$. By
        Lemma~\ref{lem:cpctKx}, $W=L^{-1}(V)$ is a neighborhood of $(x_0, J(\bx))$
        in $\{(x,s)\in I\times\{0,1\}^\N\,:\,s\in\cK_x\}$: that is, there is
        some~$\epsilon>0$ and $r\in\N$ such that if $\bx'\in\hI$ has
        $|x'_0-x_0|<\epsilon$, and $J(\bx)_i = J(\bx')_i$ for $0\le i\le r$, then
        $\bx'\in V$.

        Since~$J$ is continuous on~$\fib{x_0}$ by Remark~\ref{rem:cont-itin}, there
        is some~$\eta>0$ such that if $\by\in\fib{x_0}$ and
        $\alpha_{x_0}(\{\bz\in\fib{x_0}\,:\,J(\bz)\text{ is between }J(\bx)\text{
        and }J(\by)\})<\eta$, then $J(\by)_i = J(\bx)_i$ for $0\le i\le r$.

        Taking $\delta = \min(\epsilon, \eta)$, it follows that if $y\in
        \fs(x)\cap R^\pitchfork$ and $z\in \fu(x)\cap R^\pitchfork$ with
        $\nu_{\fs(x)}([x,y]_s)<\delta$ and $\nu_{\fu(x)}([x,z]_u)<\delta$, then the
        stream arcs $[x,y]_s$, $[y, \fs(z)\pitchfork\fu(y)]_u$,
        $[\fs(z)\pitchfork\fu(y), z]_s$, and $[z,x]_u$ are all contained in~$U$.
        This establishes regularity of~$R$.
\end{proof}

This completes the proof that $F\colon\Sigma\to\Sigma$ is a measurable
pseudo-Anosov map whenever $f$ has rational height and is not post-critically
finite.

\section{The irrational case} 
    \label{sec:irrat}

In this section we assume that~$f$ is of irrational type. We will show that the
corresponding sphere homeomorphism $F\colon\Sigma\to\Sigma$ is a generalized
pseudo-Anosov map, having a single bi-infinite orbit of $1$-pronged singularities
which is homoclinic to a point~$\infty = g(\cC)$ of~$\Sigma$, where $\cC$ is a
Cantor set in~$\hI$.

\subsection{The semi-conjugacy \texorpdfstring{$g\colon\hI\to\Sigma$}{g}}
\label{sec:smi-conj-irrat}
Here we describe the identifications on~$\hI$ induced by the semi-conjugacy
$g\colon\hI\to\Sigma$.

Recall from Theorem~\ref{thm:height-irrational} that $\orb(f(a)_\ell, \tB)$ is
disjoint from~$\gamma$, and that
\[
    \cC_S=\overline{\orb(f(a)_\ell, \tB)} = S\setminus \bigcup_{r\ge
    0}\tB^{-r}(\ingam)   
\]
is a Cantor set which contains $a$ and $\ha_u$. The following lemma is a
translation of remark~5.17(a) of~\cite{prime} into the language of this paper.

\begin{lem}\label{lem:irrat-equiv}
    The semi-conjugacy $g\colon\hI\to\Sigma$ is the quotient map of the equivalence
    relation on~$\hI$ with the following non-trivial equivalence classes:
    \begin{enumerate}[(EI)]
        \item $\{\hf^r(\be(x_u)), \hf^r(\be(\hx_u))\}$ for each $x\in(a,c)$ and
        each $r\in\Z$.

        \item The subset $\cC = \overline{\be(\cC_S)}$ of $\hI$, which is a
        Cantor set.
    \end{enumerate}
\end{lem}

We write $\infty = g(\cC) \in\Sigma$.

\subsection{The equivalence class \texorpdfstring{$\cC$}{C}}
We start by investigating the infinite equivalence class
$\cC=\overline{\be(\cC_S)}$ of (EII). The key observation is that, although~$a$ is
clearly an endpoint of the $\tB$-invariant Cantor set~$\cC_S$, its image $\tB(a) =
f(a)_\ell$ is not (Lemma~\ref{lem:sequences-in-cS}): this is because $\tB$
collapses the gap $\ingam$ adjacent to~$a$. Since $\be$ is discontinuous along the
orbit of $f(a)_\ell$ by Lemma~\ref{lem:tB-facts}(d), the image $\be(\cC_S)$ is not
compact: taking its closure can be seen (Theorem~\ref{thm:C-is-Cantor}) as adding a
gap adjacent to each $\tB^r(f(a)_\ell)$, whose other endpoint corresponds to the
non-extreme element $\hf^{r+1}(\be(\ha_u))$ of~$\hI$.

\begin{notn}[$\toplus$, $\tominus$]
    If $(y_i)$ is a sequence in~$S$, we will write $y_i\toplus y$ (respectively
    $y_i\tominus y$) if $y_i\to y$ in the positive (respectively negative)
    direction: that is, if for every $z\in S$ distinct from~$y$ we have
    $y_i\in[z,y]$ (respectively $y_i\in[y,z]$) for sufficiently large~$i$.
\end{notn}

\begin{lem} \label{lem:sequences-in-cS}
    For every $r\ge 0$, there are sequences $(y_i)$ and $(y_i')$ in~$\cC_S$
    with \mbox{$y_i\toplus\tB^r(f(a)_\ell)$} and
    $y_i'\tominus\tB^r(f(a)_\ell)$.
\end{lem}
\begin{proof}
    Since $a\in\cC_S=\overline{\orb(f(a)_\ell, \tB)}$ and $\orb(f(a)_\ell, \tB)
    \cap \gamma=\emptyset$ by Theorem~\ref{thm:height-irrational}, there is a
    sequence $(y_i)$ in $\orb(f(a)_\ell, \tB)\subset\cC_S$ with $y_i\tominus a$.
    Applying $\tB^{r+1}$, and using the fact that $\tB$ is defined and continuous
    in a neighborhood of each point of $\orb(f(a)_\ell, \tB)$ gives the required
    sequences converging negatively to points of this orbit. By an identical
    argument, there is a sequence $(y_i')$ in $\cC_S$ with $y_i'\toplus \ha_u$,
    which gives the required sequences converging positively to points of
    $\orb(f(a)_\ell, \tB)$.
\end{proof}

\begin{thm} \label{thm:C-is-Cantor}
    The Cantor set $\cC$ is $\hf$-invariant, and is given by
    \begin{equation} \label{eq:C-expression}
        \cC = \be(\cC_S) \cup \{\hf^r(\be(\ha_u))\,:\,r\ge 1\}.
    \end{equation} 
    Moreover $\orb(\be(f(a)_\ell), \hf)$ is dense in~$\cC$.
\end{thm}

\begin{proof}
    $\be(\cC_S)\subset\cC$, so for~\eqref{eq:C-expression} we need to show that
    $\cC\setminus\be(\cC_s) = \{\hf^r(\be(\ha_u))\,:\,r\ge 1\}$.

    Suppose that~$\bx\in\cC$, so that there is a sequence $(y_i)$ in~$\cC_S$ with
    $\be(y_i)\to\bx$. Taking a subsequence if necessary, assume $y_i\to y\in\cC_S$.
    By Lemma~\ref{lem:tB-facts}(d), $\be$ is discontinuous at~$y$ if and only if
    $y\in\orb(f(a)_\ell, \tB)$; and even at points of this orbit, $\be$ is
    continuous as we approach~$y$ in the negative direction. Therefore, if either
    $y\not\in\orb(f(a)_\ell, \tB)$, or there is a subsequence $y_{r_i}\tominus y$,
    then $\bx=\be(y)\in\be(\cC_S)$. On the other hand,
    by~\ref{lem:sequences-in-cS}, there is for each~$r\ge 1$ a sequence $(y_i)$
    in~$\cC_S$ with $y_i\toplus \tB^{r-1}(f(a)_\ell)$. Hence
    $\cC\setminus\be(\cC_s)$ is exactly the set of limits of $(\be(y_i))$ for such
    sequences~$(y_i)$.

    Suppose then that $y_i\toplus\tB^{r-1}(f(a)_\ell)$. Now
    $\be(y_i)=\thr{\tau(y_i), \tau(\tB^{-1}(y_i)),\dots}$ by
    Lemma~\ref{lem:tB-facts}(c). Therefore
    \begin{eqnarray*}
        \be(y_i) 
        &\to& 
            \thr{\tau(\tB^{r-1}(f(a)_\ell)),
            \tau(\tB^{r-2}(f(a)_\ell)), \dots, \tau(f(a)_\ell), \tau(\ha_u),
            \tau(\tB^{-1}(\ha_u)), \dots} \\
        &=&
            \thr{f^r(a), f^{r-1}(a), \dots, f(a), \tau(\ha_u),
            \tau(\tB^{-1}(\ha_u)), \dots } \\
        &=&
            \hf^r\left(\thr{\tau(\ha_u), \tau(\tB^{-1}(\ha_u)), \dots}\right)
        \\ 
        &=& 
            \hf^r(\be(\ha_u)),
    \end{eqnarray*}
    where the first equality uses Lemma~\ref{lem:tB-facts}(a), the third uses
    Lemma~\ref{lem:tB-facts}(c), and the appearance of~$\ha_u$ rather than~$a$ in
    the first line is because $y_i$ converges in the positive direction. This
    establishes~\eqref{eq:C-expression}.

    Since $\hf^r(\be(f(a)_\ell)) = \be(\tB^r(f(a)_\ell))\in\be(\cC_S)$
    by~\eqref{eq:model_action}, we have $\orb(\be(f(a)_\ell), \hf)\subset\cC$.
    We finish by showing that this orbit is dense in~$\cC$ so that, in
    particular, $\cC$ is $\hf$-invariant.

    It is enough to show denseness in $\be(\cC_S)$. Since $\cC_S =
    \overline{\orb(f(a)_\ell, \tB)}$, every $y\in\cC_S$ either lies on
    $\orb(f(a)_\ell, \tB)$, or is the limit of a subsequence
    $\tB^{r_i}(f(a)_\ell)\to y$ of this orbit. In the former case we have
    $\be(y)=\hf^r(\be(f(a)_\ell))$ for some~$r$ by~\eqref{eq:model_action}; and
    in the latter case, $\hf^{r_i}(\be(f(a)_\ell))=\be(\tB^{r_i}(f(a)_\ell))
    \to \be(y)$ by~\eqref{eq:model_action} and Lemma~\ref{lem:tB-facts}(d).
\end{proof}

We finish the section with some key facts about the dynamics of~$\hf$ as it
pertains to~$\cC$, and the identifications on the extreme points of~$\hI$.

\begin{lem} \label{lem:irrat-key}\mbox{}
    \begin{enumerate}[(a)]
        \item The $\omega$-limit set of the critical fiber~$\fib{c}$ is equal to
        $\cC$.

        \item The $\alpha$-limit set of the set of extreme elements of~$\hI$ has
            $\alpha(\be(S), \hf) \subset \cC$.

        \item The $\hf$-orbit of $\be(c_u)$ is homoclinic to~$\cC$, and the
            $F$-orbit of $g(\be(c_u))$ is homoclinic to~$\infty$.

        \item There is no~$x\in(a,b)$ for which both $\be(x_u)$ and $\be(x_\ell)$
            lie in~$\cC$.

        \item If $x\in(a,c)$ then $g(\be(x_\ell))\not=g(\be(\hx_\ell))$.
    \end{enumerate}
\end{lem}

\begin{proof} \mbox{}
    \begin{enumerate}[(a)]
        \item Note first that $\omega(\fib{c}, \hf) = \omega(\be(c_\ell),
            \hf)$, as $\diam(\hf^r(\fib{c})) = \diam(\fib{c})/2^r$. Now the first
            points of the $\tB$-orbit of $c_\ell$ are $c_\ell\mapsto b\mapsto
            a\mapsto f(a)_\ell$, which are contained in $[a, \ha_u)$, so that
            $\hf^3(\be(c_\ell)) = \be(f(a)_\ell)$ by~\eqref{eq:model_action}. The
            result follows since the $\hf$-orbit of $\be(f(a)_\ell)$ is a dense
            subset of the Cantor set~$\cC$ by Theorem~\ref{thm:C-is-Cantor}.

        \item Suppose that
            $\bx\in\alpha(\be(S), \hf)$, so that there is some $y\in S$ and a
            sequence $r_i\to\infty$ with $\hf^{-r_i}(\be(y)) =
            \be(\tB^{-r_i}(y))\to \bx$. Assume, taking a subsequence if necessary,
            that $\tB^{-r_i}(y)\to y^\ast\in S$. Then $y^\ast\in\cC_S$, since
            otherwise we would have $\tB^r(y^\ast)\in\ingam$ for some~$r\ge 0$, so
            that $\tB^r(\tB^{-r_i}(y))\in\ingam$ for sufficiently large~$i$, a
            contradiction as $\ingam$ is disjoint from the range of~$\tB^{-1}$.

            If $y^\ast\not\in\orb(f(a)_\ell, \tB)$ then $\bx=\be(y^\ast)$ by
            Lemma~\ref{lem:tB-facts}(d); while if \mbox{$y^\ast=\tB^r(f(a)_\ell)$},
            then taking a subsequence we can assume that either
            $\tB^{-r_i}(y)\toplus y^\ast$, so that \mbox{$\bx =
            \hf^{r+1}(\be(\ha_u))$}; or $\tB^{-r_i}(y)\tominus y^\ast$, so that
            $\bx=\be(y^\ast)$. In each case we have $\bx\in\cC$
            by~\eqref{eq:C-expression}.

        \item Immediate from parts~(a) and~(b).

        \item if \mbox{$x<\ha$} then $x_u\in\ingam$, so that
            $\be(x_u)\not\in\cC$; while if $x\ge \ha$ then $f(x)\le f(\ha) = f(a) <
            \ha$ (since \mbox{$j(f(a))=(11)^k0\dots$} and $j(\ha)= 1(11)^k0\dots$
            for some $k\ge 0$ by Remark~\ref{rem:sqrt2_conditions}), so that
            \mbox{$\tB(x_l) = f(x)_u\in\ingam$}, and $\be(x_\ell)\not\in\cC$.

        \item The argument is similar to that of~(d). We have $\tB(x_\ell) =
            f(x)_\ell$ and $\tB(\hx_\ell) = f(x)_u$ by~\eqref{eq:tB}, and
            \begin{itemize}
                \item if $f(x)<\ha$ then $f(x)_u\in\ingam$ while
                $f(x)_\ell\not\in\ingam$: therefore $\hx_\ell\not\in\cC$, and
                $\{x_\ell, \hx_\ell\} \not= \{\hf^{-r}(z_u), \hf^{-r}(\hz_u)\}$ for
                any $z\in(a,c)$;

                \item if $f(x)=\ha$ then $\hx_\ell\in\cC$, but $\tB^2(x_\ell) =
                f(a)_u\in\ingam$, so $x_\ell\not\in\cC$; and

                \item if $f(x)>\ha$ then $f^2(x)<\ha$, and
                $\tB^2(x_\ell)=f^2(x)_u\in\ingam$ while $\tB^2(\hx_\ell) =
                f^2(x)_\ell\not\in\ingam$, so that $g(x_\ell)\not=g(\hx_\ell)$ as
                in the first case.
            \end{itemize}
    \end{enumerate}
\end{proof}

\subsection{The invariant foliations}
The generalized pseudo-Anosov $F$-invariant foliations are constructed in a very
similar way to the measurable pseudo-Anosov turbulations in the rational case:
leaves of the stable foliation are locally $g$-images of the fibers~$\fib{x}$, with
measure coming from $g_*(\alpha_x)$; while leaves of the unstable foliation are
$g$-images of connected components of~$\hI\setminus\cC$, with measure induced by
Lebesgue measure on $0$-flat arcs.

These foliations have a single bad singularity at~$\infty$, where uncountably many
stable and unstable leaves meet: in the language of
Definitions~\ref{defn:pa-gpa-fols}, we have $Z=\{\infty\}$. They also have
$1$-pronged singularities along the bi-infinite orbit of $g(\be(c_u))$: the leaves
emanating from these singularities are finite, terminating at~$\infty$. All other
points are regular.

In Section~\ref{sec:irrat_stab} we construct the stable measured foliation $(\cF^s,
\nu^s)$, and show that $F(\cF^s, \nu^s) = (\cF^s, \lambda\nu^s)$; we then deal with
the unstable foliation in Section~\ref{sec:irrat_unstab}. Finally, in
Section~\ref{sec:irrat_charts}, we exhibit the charts required by
Definitions~\ref{defn:pa-gpa-fols} to establish that these are generalized
pseudo-Anosov foliations, completing the proof that $F\colon\Sigma\to\Sigma$ is a
generalized pseudo-Anosov map.

\begin{rem} \label{rem:id_summary}
    It will be helpful to summarize the essential features of the
    identifications described in Lemma~\ref{lem:irrat-equiv}.
    \begin{enumerate}[(a)]
        \item For $r\le 0$, (EI) identifies the two
        extreme points $\be(\tB^r(x_u))$ and $\be(\tB^r(\hx_u))$.

        \item For $r>0$, (EI) identifies two non-extreme
        points of the same fiber~$\fib{f^r(x)}$.

        \item (EII) collapses to~$\infty$ an uncountable
        collection of extreme points of distinct (Lemma~\ref{lem:irrat-key}(d))
         fibers together with non-extreme points $\hf^r(\be(\ha_u))$ of the
         post-critical fibers $\fib{f^r(a)}$ for $r\ge 1$.
    \end{enumerate}
\end{rem}

\subsubsection{The stable foliation} \label{sec:irrat_stab}

In the rational case, the images in~$\Sigma$ of fibers above the orbit of~$c$ are
in general not arcs, because of the 3-element identifications in
Lemma~\ref{lem:rat-equiv-class}. By contrast, in the irrational case, where there
are no such identifications, \emph{every} $\pi_0$-fiber yields an arc $\cA_x =
g(\fib{x})$. Fibers above the postcritical set still need to be treated by a
separate argument, since the itinerary, and hence the notion of consecutive points,
is not well defined on such fibers, but this argument is a straightforward
induction. The reader should note carefully in
Theorem~\ref{thm:irrational-fiber-arc} that, although $\be(\tB^r(f(a)_\ell))$ is by
definition an extreme element of $\fib{f^{r+1}(a)}$ for $r\ge 0$, the point
\mbox{$\infty = g(\be(\tB^r(f(a)_\ell)))$} is not an endpoint of the arc
$g(\fib{f^{r+1}(a)})$. This situation arises because there is no natural order
on $\fib{f^{r+1}(a)}$, due to the lack of well-defined itineraries.

\begin{thm} \label{thm:irrational-fiber-arc}
    Let~$x\in I$. Then $\cA_x := g(\fib{x})$ is an arc in~$\Sigma$. Moreover
    \begin{enumerate}[(a)]
        \item The endpoints of~$\cA_x$ are $g(\be(x_\ell))$ and $g(\be(x_u))$,
        unless $x\in\orb(b,f)$, in which case, writing $x=f^r(c)$ for $r\ge 1$,
        one endpoint of~$\cA_x$ is $g(\hf^r(\be(c_u)))$ and the other is either
        $g(\be(x_\ell))$ or $g(\be(x_u))$.

        \item $\infty=g(\cC)$ is an interior point of~$\cA_x$ if and only if
        $x\in\orb(f(a), f)$.
    \end{enumerate}
\end{thm}

\begin{proof}
    Suppose first that $x\not\in\orb(b,f)$. If $\bx$ and $\bx'$ are distinct
    elements of~$\fib{x}$, then $g(\bx)=g(\bx')$ if and only if $\bx$ and $\bx'$
    are consecutive. The proof of this is essentially identical to that of
    Lemma~\ref{lem:idents_in_fibers}(b), the only additional point to note being
    that if $g(\bx)=g(\bx')$ then, by Remark~\ref{rem:id_summary}(c), it is not possible
    that $\bx, \bx'\in\cC$.

    That $\cA_x$ is an arc for $x\not\in\orb(b,f)$, with endpoints $g(\be(x_u))$
    and $g(\be(x_\ell))$, follows exactly as in the proof of
    Theorem~\ref{thm:fiber_arc}. $\infty$ is not an interior point of~$\cA_x$
    by Remark~\ref{rem:id_summary}(c).

    Since $\cA_b = F(\cA_c)$ is a homeomorphic image of $\cA_c$, it is also an
    arc, with endpoints $F(g(\be(c_\ell))) = g(\hf(\be(c_\ell))) = g(\be(b)) =
    \infty$ and $F(g(\be(c_u))) = g(\hf(\be(c_u)))$. (That $g(\be(b))=\infty$
    follows because $\tB^2(b)=f(a)_\ell$, so $\orb(b, \tB)$ is disjoint from
    $\ingam$: that is, $b\in\cC_S$, so $\be(b)\in\cC$.) Likewise $\cA_a =
    F(\cA_b)$ is an arc with endpoints $\infty$ and $g(\hf^2(\be(c_u)))$.

    The result for $x\in\orb(f(a), f) = \{f^r(a)\,:\,r\ge 1\}$ follows by induction
    on~$r$. For the base case $r=1$, we have that $\cA_{f(a)} = F(\cA_a) \cup
    F(\cA_{\ha}) = F(\cA_a\cup\cA_{\ha})$ is the union of the $F$-images of the two
    arcs $\cA_a$ and $\cA_{\ha}$ (the latter is an arc as $\ha\not\in\orb(b,f)$,
    since $f(\ha)=f(a)$ and $f$ is injective on $\orb(b,f)$). Now $\cA_a$ and
    $\cA_{\ha}$ intersect at their endpoints $g(\be(a))=g(\be(\ha_u)) = \infty$.
    They have no other points of intersection by Remark~\ref{rem:id_summary}
    (noting that~$\fib{a}$ has only one extreme point~$\be(a)$). Therefore
    $\cA_{f(a)}$ is an arc with~$\infty$ in its interior, with endpoints
    $g(\hf^3(\be(c_u)))$ and $g(\hf(\be(\ha_\ell))) = g(\be(f(a)_u))$ as required.

    Suppose then that~$r>1$. 

    If $f^r(a) < f(a)$ then, by the inductive hypothesis,
    $\cA_{f^r(a)}=F(\cA_{f^{r-1}(a)})$ is an arc with~$\infty$ in its interior.
    Since $f^r(a)<f(a)$, both extremes of $\fib{f^{r-1}(a)}$ map to extremes of
    $\fib{f^r(a)}$, so that the endpoints of $\cA_{f^r(a)}$ are
    $F(g(\be(f^{r-1}(a)_{u/\ell}))) = g(\be(f^r(a)_{\ell/u}))$ and
    $F(g(\hf^{r-1}(\be(c_u)))) = g(\hf^r(\be(c_u)))$ as required.

    On the other hand, if $f(a) < f^r(a) < b$ then $\cA_{f^r(a)}$ is the union of
    $F(\cA_{f^{r-1}(a)})$ and $F(\cA_{\widehat{f^{r-1}(a)}})$. The former is an arc
    with~$\infty$ in its interior by the inductive hypothesis, and the latter is
    also an arc, since $\widehat{f^{r-1}(a)}\not\in\orb(b,f)$ as~$f$ is injective
    on $\orb(b,f)$. Now $\cA_{f^{r-1}(a)}$ and $\cA_{\widehat{f^{r-1}(a)}}$
    intersect at their endpoints $g(\be({f^{r-1}(a)}_u)) =
    g(\be(\widehat{f^{r-1}(a)}_u))$ by~(EI) of Lemma~\ref{lem:irrat-equiv}. This is
    their only intersection: their other endpoints are not identified
    by Lemma~\ref{lem:irrat-key}(e); the identifications~(EI) only
    identify extremes of fibers and points of the same fiber; and
    $\infty\not\in\cA_{\widehat{f^{r-1}(a)}}$ (if we had
    $\widehat{f^{r-1}(a)}_\ell\in\cC$ as well as $f^{r-1}(a)_\ell\in\cC$, then both
    extremes of $\fib{f^r(a)}$ would be in~$\cC$,
    contradicting Lemma~\ref{lem:irrat-key}(d)). Therefore
    $\cA_{f^r(a)}$ is an arc with~$\infty$ in its interior, and with endpoints
    $F(g(\be(\widehat{f^{r-1}(a)}_\ell))) = g(\be(f^r(a)_{u/\ell}))$ and
    $F(g(\hf^{r-1}(\be(c_u)))) = g(\hf^r(\be(c_u)))$ as required.
\end{proof}

\begin{defn}[$\cI_x$, $\cJ_x$] \label{defn:ix_jx}
    Let $x\in I$.

    If $x=f^r(a)$ for some $r\ge 1$, we define $\cJ_x$ to be the component of
    $\cA_x\setminus\{\infty\}$ which contains $g(\hf^r(\be(c_u)))$, and $\cI_x$
    to be the other component.

    Otherwise, we define $\cI_x = \cA_x\setminus\{\infty\}$.

    Each $\cI_x$ and $\cJ_x$ is therefore homeomorphic to a half-open or closed
    interval, according as $\infty\in\cA_x$ or $\infty\not\in\cA_x$.
\end{defn}

There are three types of extreme elements of~$\hI$:
\begin{itemize}
    \item points of~$\cC$, which are identified with all other points of~$\cC$; 
    \item points of the form $\hf^r(\be(c_u))$ for $r\le 0$, which are not
    identified with any other points; and
    \item points of the form $\hf^r(\be(x_u))$ for $r\le 0$ and $x\in(a,
    \ha) \setminus\{c\}$, which are identified with $\hf^r(\be(\hx_u))$ and
    no other points.
\end{itemize}
Therefore each endpoint of each $\cI_x$ is either $\infty$, or a point of the form
$\hf^r(\be(c_u))$, or is an endpoint of exactly one other $\cI_{x'}$.

We define an equivalence relation $\sim$ on $I$ by declaring that $x\sim y$ if and
only if there is a finite sequence $x=x_1,\dots,x_k=y$ such that $\cI_{x_i}$ and
$\cI_{x_{i+1}}$ share an endpoint for $1\le i < k$. We denote by~$\ec{x}$ the
equivalence class containing~$x$.

\begin{defn}[$(\cF^s, \nu^s)$] \label{defn:stabfol_irrat}
    The stable measured foliation $(\cF^s, \nu^s)$ on~$\Sigma$ has
    leaves~$\ell$ of the following forms:
    \begin{itemize}
        \item $\ell=\bigcup_{y\in\ec{x}} \cI_y$ for $x\in I$, with measure
        $\displaystyle{
            \nu^s_\ell = \sum_{y\in \ec{x}}g_*(\alpha_y),
            }
        $
        where $\alpha_{y}$ is the measure of Remarks~\ref{rems:alpha-x}(c)
            on~$\fib{y}$ (restricted to $\cI_y$);
        \item $\ell = \cJ_{f^r(a)}$ for $r\ge 1$, with measure $\nu^s_\ell = g_*(\alpha_{f^r(a)})$; and
        \item A singleton leaf $\ell = {\infty}$.
    \end{itemize}

\end{defn}

\begin{lem}
    $F(\cF^s, \nu^s) = (\cF^s, \lambda\nu^s)$.
\end{lem}

\begin{proof}
    We have $F(\cI_x)\subset \cI_{f(x)}$ for all~$x\in I$, and $F(\cJ_{f^r(a)}) =
    \cJ_{f^{r+1}(a)}$ for all $r\ge 1$, so that $F(\ell)$ is a leaf whenever~$\ell$
    is. The proof that~$F$ contracts the stable measure by a factor~$\lambda$ is
    similar to that of Lemma~\ref{lem:stable-action}: if $A$ is a Borel subset of
    $\ell=\bigcup_{y\in\ec{x}}\cI_y$, then, since $\ec{f(x)} = f(\ec{x})$, we have
    \begin{eqnarray*}
        \nu_{F(\ell)}^s(F(A)) &=& \sum_{y\in\ec{f(x)}}\alpha_y(g^{-1}(F(A))) \\
            &=& \sum_{y\in f(\ec{x})}\alpha_y(g^{-1}(F(A))) \\
            &=& \sum_{y\in\ec{x}}\alpha_{f(y)}(g^{-1}(F(A))) \\
            &=& \sum_{y\in\ec{x}}\alpha_{f(y)}(\hf(g^{-1}(A))) \\
            &=& \lambda^{-1} \sum_{y\in\ec{x}} \alpha_y(g^{-1}(A)) 
                \qquad\text{by Remarks~\ref{rems:alpha-x}(b)} \\
            &=& \lambda^{-1} \nu_\ell^s(A).
    \end{eqnarray*}
    An analogous but more straightforward argument applies to the leaves
    $\cJ_{f^r(a)}$.
\end{proof}
    
\subsubsection{The unstable foliation} \label{sec:irrat_unstab}

Each path component of~$\hI$ is either a point or an arc by
Lemma~\ref{lem:basic_top}(b). Any path component of $\hI\setminus\cC$ is therefore
a path connected subset of a point or an arc, and so is itself a point or an arc.
Since $\cC=\omega(\fib{c}, \hf)$ by Lemma~\ref{lem:irrat-key}(a), it follows from
Lemma~\ref{lem:basic_top}(a) that in fact each path component of $\hI\setminus\cC$
is an open arc.

$\be$ is continuous on~$\gamma$ by Lemma~\ref{lem:tB-facts}(d) and
Theorem~\ref{thm:height-irrational}, so that $\{\be(x_u)\,:\,x\in[a, \ha]\}$ is an
arc in~$\hI$ which intersects~$\cC$ exactly at its endpoints $\be(a)$ and
$\be(\ha_u)$. It follows that \mbox{$S_0:=\{\be(x_u)\,:\,x\in (a,\ha)\}$} is a path
component of $\hI\setminus\cC$.

Since $g^{-1}(\{g(\be(x_u))\}) =
\{\be(x_u), \be(\hx_u)\}$ for $x\in(a, \ha)$, the image $g(S_0)$ of $S_0$ in~$\Sigma$ is the
half-open interval $\cS_0:=\{g(\be(x_u))\,:\,x\in(a, c]\}$ with endpoints
$g(\be(c_u))$ and~$\infty$.

\begin{defn}[$S_r$, the spikes~$\cS_r$]
    For each $r\in\Z$, we write $S_r = \hf^r(S_0)$, a path component of
    $\hI\setminus\cC$; and $\cS_r = F^r(\cS_0) = g(S_r)$, a half-open interval
    in~$\Sigma$ with one endpoint on the orbit of $g(\be(c_u))$ and the other
    at~$\infty$. We refer to the $\cS_r$ as \emph{spikes}.
\end{defn}

Every $\bx\in\hI\setminus(\cC\cup\bigcup_{r\in\Z}S_r)$ has $g^{-1}(\{g(\bx)\}) =
\{\bx\}$. Therefore, each path component $\Gamma$ of $\hI\setminus\cC$ other than
the $S_r$ has image $g(\Gamma)$ an immersed line; and these images, together with
$\infty$ and the spikes~$\cS_r$, partition~$\Sigma$.

\begin{defn}[$(\cF^u, \nu^u)$] \label{defn:unstab_fol_irrat}
    The unstable measured foliation~$(\cF^u, \nu^u)$ on~$\Sigma$ has
    leaves~$\ell=\{\infty\}$ and $\ell = g(\Gamma)$ for each path
    component~$\Gamma$ of $\hI\setminus\cC$.

    For $\ell=g(\Gamma)$ and~$A$ a Borel subset of~$\ell$,
    \begin{equation} \label{eq:spike_or_not}
        \nu^u_\ell(A) = 
        \begin{cases}
            \frac{1}{2} \sum_i m (\pi_0(g^{-1}(A) \cap \Gamma_i)) & \text{ if ~$\ell$ is a spike,} \\
            \sum_i m (\pi_0(g^{-1}(A) \cap \Gamma_i)) & \text{ otherwise,}
        \end{cases}
    \end{equation}
    where $\Gamma=\bigcup_i\Gamma_i$ is the $0$-flat decomposition of
        Lemma~\ref{lem:decompose-glr}.   
\end{defn}

\begin{lem}
    $F(\cF^u, \nu^u) = (\cF^u, \lambda^{-1}\nu^u)$.
\end{lem}
\begin{proof}
    Since $\hf$ preserves~$\cC$ and permutes the path components of
    $\hI\setminus\cC$, mapping $S_r$ to $S_{r+1}$, the leaves of $F(\cF^u,
    \nu^u)$ coincide with the leaves of $(\cF^u, \nu^u)$.

    The proof that $F_*(\nu^u_\ell) = \lambda^{-1}\nu^u_{F(\ell)}$ for each~$\ell$
    is formally identical to that of Lemma~\ref{lem:unstab-action}, using that
    $F(\ell)$ is a spike if and only if~$\ell$ is.
\end{proof}

\subsubsection{Charts for the foliations} \label{sec:irrat_charts}

We have defined a pair $(\cF^s, \nu^s)$, $(\cF^u, \nu^u)$ of $F$-invariant measured
foliations whose leaves are respectively contracted and expanded by a
factor~$\lambda$ under the action of~$F$. It remains to show that these are
generalized pseudo-Anosov foliations, by constructing the charts of
Definitions~\ref{defn:charts}. We do this first around points $g(\bx)$ for which
$x_0$ is not in the closure of the post-critical set $\PC = \orb(b, f)$, and then
use the action of~$F$ to transfer these charts to general points
of~$\Sigma\setminus\{\infty\}$.

\begin{notn}[$Y$]
    We write $Y=I\setminus\overline{\PC}$.
\end{notn}

\begin{lem} \label{lem:make_0_box}
    Let~$K$ be an interval contained in~$Y$. Then $\pi_0^{-1}(K)$ is a disjoint
    union of $0$-flat arcs $\gamma\colon K\to\hI$ over~$K$.
\end{lem}
\begin{proof}
    Since $K\subset Y$ we have in particular that $b$, $a$, and $f(a)$ are not
    in~$K$, so that either $K\subset(a, f(a))$, or $K\subset(f(a), b)$. In the
    former case $f^{-1}(K)\subset Y$ is a single interval mapped homeomorphically
    by~$f$ onto~$K$; and in the latter case $f^{-1}(K)\subset Y$ is a union of two
    disjoint intervals each mapped homeomorphically onto~$K$.

    Given $\bx\in\pi_0^{-1}(K)$, we can define a sequence~$(K_r)$ of intervals
    in~$Y$ inductively by taking $K_0=K$, and $K_r$ to be the component of
    $f^{-1}(K_{r-1})$ which contains~$x_r$. By the previous paragraph, $f$ maps
    each $K_r$ homeomorphically onto $K_{r-1}$.

    The function $\gamma\colon K\to \hI$ given by $\gamma(y) = \thr{y, y_1, y_2,
    \dots}$, where $y_r\in K_r$ for each~$r$, therefore defines a $0$-flat arc
    over~$K$, passing through~$\bx$, on which every point has the same itinerary
    as~$\bx$. It follows that $\pi_0^{-1}(K)$ is a disjoint union of such $0$-flat
    arcs over~$K$, that is, a $0$-box over~$K$.
\end{proof}

If $\bx\in\hI$ satisfies $x_r\in\overline{\PC}$ for all~$r$, then
$\bx\in\omega(\fib{c}, \hf)$, so that $\bx\in\cC$ by Lemma~\ref{lem:irrat-key}(a).
For any other elements of~$\hI$, since $\overline{\PC}$ is forward $f$-invariant,
there is a unique transition point in the thread between elements of
$\overline{\PC}$ and elements of $Y$.

\begin{defn}[$r(\bx)$]
    If $\bx\not\in\cC$, let $r(\bx)\ge 0$ be least such that $x_{r(\bx)}
    \not\in\overline{\PC}$. Then $x_r\not\in\overline{\PC}$ for all
    $r\ge r(\bx)$.
\end{defn}

Suppose that $r(\bx)=0$. Let~$K$ be the component of~$Y$ containing~$x_0$. Define
$\psi\colon \pi_0^{-1}(K)\to\R$ by
\begin{equation} \label{eq:psi}
    \psi(\by) = 
        \alpha_{y_0}\left( 
            \{ \bz\in\fib{y_0}\,:\, J(\bz)\preceq J(\by) \} 
        \right).
\end{equation}
Then $\psi$ is constant on each $0$-flat arc of $\pi_0^{-1}(K)$ by
Theorem~\ref{thm:holonomy_alpha}. Moreover, for each $x\in K$ and each $\by,
\by'\in\fib{x}$, we have $\psi(\by)=\psi(\by')$ if and only if $\by$ and
$\by'$ are consecutive in~$\fib{x}$, since $\alpha_x$ is OU by
Remarks~\ref{rems:alpha-x}(a).

Suppose in addition that~$\bx$ is not an extreme element of its $\pi_0$-fiber. Let
$\psi_\ast = \max_{\by\in\pi_0^{-1}(K)}\psi(\by)$, and $U_\bx =
\psi^{-1}(0, \psi_{\ast})$, a neighborhood of~$\bx$. Let $R_\bx = K\times(0,
\psi_\ast)$, and define $\Psi_\bx\colon U_\bx \to R_\bx$ by $\Psi_\bx(\by) = (y_0,
\psi(\by))$. Then $\Psi_\bx(\by) = \Psi_\bx(\by')$ if and only if $g(\by)=g(\by')$.
Writing $N_\bx = g(U_\bx)$, a neighborhood of $g(\bx)$, $\Psi_\bx$ therefore
induces a homeomorphism $\Phi_\bx\colon N_\bx \to R_\bx$ with $\Psi_\bx = \Phi_\bx
\circ g$.

\begin{lem}
    \label{lem:interior_regular}
    Suppose that $r(\bx)=0$ and $\bx$ is not an extreme element of its
    $\pi_0$-fiber. Then $\Phi_\bx\colon N_\bx\to R_\bx$ is a regular chart for the
    invariant foliations of~$F$ at $g(\bx)$.
\end{lem}

\begin{proof}
    According to Definitions~\ref{defn:charts} we need to show that $\Phi_\bx$ is
    measure-preserving; and that for each unstable (respectively stable)
    leaf~$\ell$, and each path component~$L$ of $\ell\cap N_\bx$, $\Phi_\bx|_{L}$
    is a measure-preserving homeomorphism onto $K\times\{y\}$ for some $y\in (0,
    \phi_\ast)$ (respectively $\{x\} \times (0, \phi_\ast)$ for some $x\in K$).
    \begin{enumerate}[(a)]
        \item Let~$A$ be a Borel subset of~$N_\bx$. Then
            \begin{eqnarray*}
                \mu(A) = \hmu(g^{-1}(A)) &=& \int_K \alpha_x(g^{-1}(A)) \,dm(x) \\
                &=& \int_K \alpha_x(g^{-1}(A) \cap \fib{x}) \,dm(x) \\ &=&
                \int_K m \left(\psi(g^{-1}(A) \cap \fib{x})\right) \,dm(x) \\
                &=& \int_K m\left(
                        \{y\in(0, \psi_\ast)\,:\,(x,y)\in\Psi_\bx(g^{-1}(A))\}
                    \right) \,dm(x) \\
                &=& \int_K m\left(
                        \{y\in(0, \psi_\ast)\,:\,(x,y)\in\Phi_\bx(A)\}
                    \right) \,dm(x) \\
                &=& M(\Phi_\bx(A)),
            \end{eqnarray*}
            where~$M$ is two-dimensional Lebesgue measure. Here the equality in
            the first line is given by Theorem~\ref{thm:disintegrate}. 

        \item 
            Let~$\ell=g(\Gamma)$ be a non-trivial leaf of~$\cF^u$, where
            $\Gamma$ is a path component of $\hI\setminus\cC$. Then $\Gamma
            \cap U_\bx$ is a countable union of $0$-flat arcs over~$K$, so that
            each path component~$L$ of $\ell\cap N_\bx$ is of the form
            $L=g(\gamma(K))$, where $\gamma\colon K\to \hI$ is one of the
            $0$-flat arcs of Lemma~\ref{lem:make_0_box}. Then
            $\Phi_\bx(g(\gamma(x))) = (x, C)$ for some constant~$C$, so
            that~$\Phi_\bx|_L$ is a homeomorphism onto the unstable leaf
            segment $K\times \{C\}$ of~$R_\bx$.

            If $A=g(\gamma(B))$ is a Borel subset of~$L$, then $g^{-1}(A) =
            \gamma(B)$ if~$\ell$ is not a spike; while $g^{-1}(A) = \gamma(B) \cup
            \gamma'(B)$ if~$\ell$ is a spike, where $\gamma'\colon K\to \hI$ is the
            $0$-flat arc consecutive to~$\gamma$. In either case,
            $\nu^u_\ell(A) = m(B)$ by~\eqref{eq:spike_or_not}, so that
            $\nu^u_\ell(A) = m(\Phi_\bx|_L(A))$ as required.

        \item 
            Let~$\ell$ be a non-trivial leaf of~$\cF^s$. 

            If~$\ell=\cJ_{f^r(a)}$ for some~$r\ge 1$, then $\ell\subset
            g(\fib{f^r(a)}) \subset g(\pi_0^{-1}(\PC))$, so $\ell\cap N_\bx =
            \emptyset$. We can therefore assume that $\ell =
            \bigcup_{z\in\ec{y}}\cI_z$ for some~$y\in I$, a countable union of
            sets of the form $\cI_z$. Since any such set which intersects
            $N_\bx$ must have $\cI_z = \cA_z$ (as $N_\bx$ is disjoint from
            $g(\pi_0^{-1}(\PC))$), each component~$L$ of $\ell\cap N_\bx$ is of
            the form $\cA_x \cap N_\bx$ for some $x\in K$: that is, it is the
            $g$-image of a fiber~$\fib{x}$ less its extremes, so that
            $\Phi_\bx|_L$ is a homeomorphism onto the stable leaf segment
            $\{x\} \times (0, \phi_\ast)$ of~$R_{\bx}$.

            For each $y\in(0, \phi_\ast)$ we have
            $\nu^s_\ell(\Phi_\bx|_L^{-1}(\{x\} \times(0, y))) = y$
            by~\eqref{eq:psi} and Definition~\ref{defn:stabfol_irrat}, so that
            $\Phi_\bx|_L$ is measure-preserving as required.

    \end{enumerate}
\end{proof}

Suppose, on the other hand, that~$r(\bx)=0$ but that~$\bx$ is an extreme element of
its fiber, so that $\bx = \hf^{-r}(\be(x_u))$ for some $r\ge 0$ and $x\in(a, \ha)$.

Consider first the case $x\not=c$. Then $g(\bx)=g(\bx')$, where $\bx' =
\hf^{-r}(\be(\hx_u))$. Write $z=\pi_0(\hf^{-r}(\be(c_u)))$, and choose $\epsilon>0$
so that both $K=(x_0-\epsilon, x_0+\epsilon)$ and $K'=(x'_0-\epsilon,
x'_0+\epsilon)$ are contained in~$Y$ and don't contain~$z$ (so, in particular, are
disjoint). Suppose that $\bx$ and $\bx'$ are lower elements of their
$\pi_0$-fibers; only minor modifications are needed to treat the case where they
are upper elements.

Define maps $\psi\colon \pi_0^{-1}(K)\to\R$ and $\psi'\colon
\pi_0^{-1}(K')\to\R$ by~\eqref{eq:psi}; let $U =
\psi^{-1}([0,\psi_\ast))$ and $U' = \psi'^{-1}([0, \psi'_\ast))$, where
$\psi_\ast$ and $\psi'_\ast$ are the maximum values of $\psi$ and $\psi'$; and
define $U_\bx = U\cup U'$, and $\Psi_\bx\colon U_\bx \to R_\bx =
K\times(-\psi'_\ast, \psi_\ast)$ by
\[
    \Psi_\bx(\by) = 
    \begin{cases}
        (y_0, \psi(\by)) & \text{ if }\by \in U, \\ (2z - y_0, -\psi'(\by)) &
        \text{ if }\by \in U'.
    \end{cases}
\]
Then $\Psi_\bx(\by) = \Psi_\bx(\by')$ if and only if $g(\by)=g(\by')$, and
$\Psi_\bx$ induces a homeomorphism $\Phi_\bx\colon N_\bx\to R_\bx$, where $N_\bx =
g(U_\bx)$, a neighborhood of $g(\bx)$. The proof of the following is essentially
identical to that of Lemma~\ref{lem:interior_regular}.

\begin{lem} \label{lem:extreme_chart}
    Suppose that~$r(\bx)=0$, and that $\bx = \hf^{-r}(\be(x_u))$ for some $r\ge 0$
    and $x\in(a, \ha)$ with $x\not =c$. Then $\Phi_\bx\colon N_\bx\to R_\bx$ is a
    regular chart for the invariant foliations of~$F$ at $g(\bx)$.
\end{lem}

To complete the analysis of charts around~$\bx$ where~$r(\bx)=0$, it remains to
consider the case where $\bx = \hf^{-r}(\be(c_u))$ for some~$r\ge 0$: in this case
we will see that there is a $1$-pronged singularity at~$\bx$. Choose~$\epsilon>0$
so that $K:=(x_0-\epsilon, x_0+\epsilon) \subset Y$, and suppose that $\bx$ is a
lower element of its $\pi_0$-fiber; only minor modifications are needed if it is an
upper element. Define $\psi\colon\pi_0^{-1}(K)\to\R$ by~\eqref{eq:psi}, let $U_\bx
=
\psi^{-1}([0, \psi_\ast))$, where $\psi_\ast$ is the maximum value
of~$\psi$, set $R_\bx = (-\epsilon, \epsilon)\times[0, \psi_\ast)$, and
define $\Psi_\bx\colon U_\bx\to R_\bx$ by
\[
    \Psi_\bx(\by) = (y_0 - x_0, \psi(\by)).
\]
Then, for $\by, \by'\in U_\bx$, we have $g(\by) = g(\by')$ if and only if
either $\Psi_\bx(\by) = \Psi_\bx(\by')$, or $\Psi_\bx(\by) = (u, 0)$ and
$\Psi_\bx(\by) = (-u, 0)$ for some $u\in (-\epsilon, \epsilon)$. Therefore,
setting $Q_\bx$ to be the $1$-pronged model $Q_\bx = R_\bx \,/\, (u, 0)
\sim (-u, 0)$, and writing $N_\bx = g(U_\bx)$, $\Psi_\bx$ induces a
homeomorphism $\Phi_\bx\colon N_\bx\to Q_\bx$, which sends $g(\bx)$ to the
singular point of~$Q_\bx$. 

\begin{lem} \label{lem:crit_chart}
    Suppose that $r(\bx)=0$, and that $\bx = \hf^{-r}(\be(c_u))$ for some $r\ge 0$.
    Then $\Phi_\bx\colon N_\bx \to Q_\bx$  is a $1$-pronged chart for the invariant
    foliations of~$F$ at $g(\bx)$.
\end{lem}

\begin{proof}
    \begin{enumerate}[(a)]
        \item If~$A$ is a Borel subset of~$N_\bx$ then $\mu(A) =
        \hmu(g^{-1}(A)) = \hmu(g^{-1}(A)\cap \psi^{-1}(0, \psi_\ast))$. That this
        is equal to $M(\Phi_\bx(A))$ follows exactly as in (a) from the proof of
        Lemma~\ref{lem:interior_regular}.

        \item As in (b) from the proof of Lemma~\ref{lem:interior_regular}, if
        $\ell$ is a non-trivial leaf of~$\cF^u$ then each path component~$L$ of
        $\ell\cap N_\bx$ is of the form $L=g(\gamma(K))$, where $\gamma\colon
        K\to\hI$ is one of the $0$-flat arcs of Lemma~\ref{lem:make_0_box}, so that
        there is some constant~$C\in [0, \psi_\ast)$ such that
        $\Phi_\bx(g(\gamma(x)))$ is the equivalence class of $(x, C)$ in~$Q_\bx$.
        If~$C>0$ then this equivalence class is trivial, and the proof that
        $\Phi_\bx|_L$ is a measure preserving homeomorphism onto the corresponding
        stable leaf segment of~$Q_\bx$ is identical to that of (b) from the proof
        of Lemma~\ref{lem:interior_regular}.

        If $C=0$ then the leaf~$\ell$ is a spike. If~$A$ is a Borel subset
        of~$L$ then $A=g(\gamma(B)) = g(\gamma(B'))$, where $B\subset
        (x_0-\epsilon, x_0]$ and $B'\subset [x_0, x_0+\epsilon)$; and hence
        $\nu_l^u(A) = m(\Phi_\bx|_L(A)) = m(B)$ as required.

        \item As in (c) from the proof of Lemma~\ref{lem:interior_regular}, if
        $\ell$ is a non-trivial leaf of~$\cF^s$, then $\ell =
        \bigcup_{z\in\ec{y}}\cI_z$ for some $y\in I$, a countable union of sets of
        the form~$\cI_z$. Since any such set which intersects $N_\bx$ must have
        $\cI_z=\cA_z$, each component~$L$ of $\ell\cap N_\bx$ is either
        $\cA_{x_0}\cap N_\bx$, or $\left(\cA_{x_0-\delta} \cup
        \cA_{x_0+\delta}\right) \cap N_\bx$ for some $\delta\in(0,\epsilon)$.
        Therefore $\Phi_\bx|_L$ is a homeomorphism onto a stable leaf segment
        of~$Q_\bx$, which is measure-preserving as in the proof of
        Lemma~\ref{lem:interior_regular}.
    \end{enumerate}
\end{proof}

\begin{thm}
    \label{thm:irrat-gpa}
    Let $F\colon\Sigma\to\Sigma$ be the sphere homeomorphism obtained from a
    tent map~$f$ of irrational height. Then~$F$ is a generalized pseudo-Anosov
    map, whose invariant foliations $(\cF^u, \nu^u)$ and $(\cF^s, \nu^s)$ have
    $1$-pronged singularities along the bi-infinite orbit of $g(\be(c_u))$, and
    are regular at every other point of $\Sigma\setminus\{\infty\}$.
\end{thm}
\begin{proof}
    We have defined stable and unstable foliations satisfying $F(\cF^s, \nu^s)
    = (\cF^s, \lambda\nu^s)$ and $F(\cF^u, \nu^u) = (\cF^u, \lambda^{-1}
    \nu^u)$. It remains to show that these foliations have $1$-pronged charts
    along the orbit of $g(\be(c_u))$, and regular charts at every other point
    of $\Sigma\setminus\{\infty\}$.

    For the former, let $x=F^s(g(\be(c_u)))$ for $s\in\Z$, so that $x = g(\bx)$
    for a unique~$\bx\in\hI$, namely $\bx = \hf^s(\be(c_u))$, and write
    $r=r(\bx)$. Then if $\by = \hf^{-r}(\bx)$ we have $r(\by) = 0$, so that
    by~\ref{lem:crit_chart} there is a $1$-pronged chart $\Phi_{\by}\colon
    N_\by \to Q_\by$ for the invariant foliations at~$g(\by) = F^{-r}(x)$.

    Recall that $Q_\by = (-\epsilon, \epsilon) \times [0, \phi_\ast) \,/\, (u,
    0)\!\sim\! (-u, 0)$ for some $\epsilon$ and $\phi_\ast$, write $Q =
    (-\lambda^r\epsilon, \lambda^r\epsilon) \times [0, \lambda^{-r}\phi_\ast)
    \,/\, (u, 0) \!\sim\! (-u, 0)$, a $1$-pronged model, and define $\xi \colon
    Q_\by \to Q$ by $[(x,y)]\mapsto [(\lambda^r x, \lambda^{-r} y)]$. Let~$N =
    F^r(N_\by)$, and $\Phi = \xi \circ \Phi_\by \circ F^{-r}\colon N\to Q$.
    Then 
    \begin{itemize}
        \item $\Phi(x) = \xi \circ\Phi_\by(g(\by)) = \xi([(0,0)]) = [(0,0)]$ is
            the singular point of~$Q$.

        \item Since $F$, $\Phi_\by$, and $\xi$ are all measure-preserving
        homeomorphisms (with respect to the measure~$\mu$ on~$\Sigma$ and
        two-dimensional Lebesgue measure~$M$ on $Q_\by$ and $Q$), $\Phi$ is a
        measure-preserving homeomorphism.

        \item If $\ell\in\cF^s$ and $L$ is a path component of $\ell\cap N$,
        then $F^{-r}(L)$ is a path component of $F^{-r}(\ell)\cap N_\by$, so
        that $\Phi_\by|_{F^{-r}(L)}$ is a measure-preserving homeomorphism onto
        a stable leaf segment of~$Q_\by$. Since $F^{-r}$ scales the measure
        on~$L$ by a factor $\lambda^r$, and $\xi$ scales the measure on stable
        leaf segments by a factor $\lambda^{-r}$, $\Phi|_L$ is a
        measure-preserving homeomorphism onto a stable leaf segment of $Q$.
        Analogously, for each $\ell\in\cF^u$ and each path component~$L$ of
        $\ell\cap N$, $\Phi|_L$ is a measure-preserving homeomorphism onto an
        unstable leaf segment of~$Q$.
    \end{itemize}
    Therefore $\Phi\colon N\to Q$ is a $1$-pronged chart for the invariant
    foliations at~$x$, as required.

    The argument for other points~$x\in\Sigma\setminus\{\infty\}$ is analogous.
    Let $\bx\in g^{-1}(x)$, $r=r(\bx)$, and $\by = \hf^{-r}(\bx)$, which
    satisfies $r(\by)=0$. By~\ref{lem:interior_regular}
    and~\ref{lem:extreme_chart}, there is a regular chart $\Phi_\by\colon
    N_\by\to R_\by$ for the invariant foliations at $y=F^{-r}(x)$. Setting
    $N=F^r(N_\by)$ and $R=\xi(R_\by)$, where $\xi(x,y) = (\lambda^rx,
    \lambda^{-r}y)$, it follows that $\Phi = \xi\circ \Phi_\by\circ
    F^{-r}\colon N\to R$ is a regular chart for the invariant foliations
    at~$x$.

\end{proof}

\subsection{The rational post-critically finite case}
\label{sec:rat-PCF}

In the case where~$f$ is post-critically finite, and therefore has rational height
$q(f)=m/n$, analogous arguments can be used to show that the sphere
homeomorphism~$F$ is generalized pseudo-Anosov. We omit the details of these
because the result in this case was proved in~\cite{prime} (using quite different
methods), and only sketch some of the main features.

Let~$Q$ be the periodic orbit, of period~$k$, say, contained in $\PC$, and $R =
\PC\setminus Q$ be the pre-periodic part of the post-critical orbit.
Therefore~$Q$ and~$R$ are finite subsets of~$I$, with~$R$ being empty in the
case where~$b$ is itself a periodic point. The natural extension~$\hf$ has a
corresponding period~$k$ orbit~$\bQ$ lying in the fibers above~$Q$.

If $x\not\in\PC$, then $\cA_x := g(\fib{x})$ is an arc in~$\Sigma$, by an argument
similar to that of Section~\ref{sec:semi-conj}. On the other hand, $\cA_x$ is a
triod for $x\in R$; while for $x\in Q$, because $F^k(\cA_x)
\subset \cA_x$, it is a tree with infinitely many valence~1 and valence~3
vertices --- corresponding to 1-pronged and 3-pronged singularities of the
invariant foliations --- except in the~NBT case mentioned below.

Recall from Theorem~\ref{thm:height-intervals-rational} and
Definition~\ref{defn:tent_type} that~$f$ can be of three types:
\begin{itemize}
    \item \emph{Endpoint type}, when $\tB^n(a) = a$ or $\tB^n(a) = \ha_u$;
    \item \emph{NBT type}, when $\tB^n(a) = c_u$; and
    \item \emph{General type}, in all other cases.
\end{itemize}

The endpoint case is most similar to the irrational case. Suppose that $\tB^n(a) =
a$, so that, by Lemma~\ref{lem:tB-facts}(c), $\be(a) = \thr{a, b =
\tau(\tB^{-1}(a)), \dots, \tau(\tB^{-(n-1)}(a)), a, \dots}$ is a period~$n$
point of~$\hf$, whose orbit lies in the fibers above~$Q$. Then $\cC =
\{\hf^r(\be(a))\,:\, 0\le r < n\} \cup
\{\hf^r(\be(\ha_u))\,:\, r\in\Z\}$ is a countable set which plays the
r\^ole of the Cantor set~$\cC$ in the irrational case. Specifically, in close
analogy with Lemma~\ref{lem:irrat-equiv}, Remark~5.17(c) of~\cite{prime} states
that~$\cC$ is a non-trivial fiber of $g\colon \hI\to\Sigma$, and that the only
other non-trivial fibers are $\{\hf^r(\be(x_u)), \hf^r(\be(\hx_u))\}$ for each
$x\in(a,c)$ and $r\in\Z$. Writing $\infty = h(\cC)\in\Sigma$, it can be shown
that~$F$ is a generalized pseudo-Anosov map with $1$-pronged singularities along
the bi-infinite orbit of~$g(\be(c_u))$ and regular points elsewhere in
$\Sigma\setminus\{\infty\}$. The orbit of $g(\be(c_u))$ is homoclinic to~$\infty$.
$F^r(g(\be(c_u)))$ lies in $\cA_x$ for some $x\in Q$ when $r\ge 0$, and is the
$g$-image of an extreme point when $r \le 0$. When $\tB^n(a) =
\ha_u$, the situation is the same with the r\^oles of $\be(a)$ and $\be(\ha_u)$
exchanged (in this case the post-critical orbit is strictly pre-periodic, with
$|R|=2$).

In the general case, the identifications are given by (EI), (EII), and (EIII) of
Lemma~\ref{lem:rat-equiv-class}. The period~$n$ orbit~$\bP$ of (EIII) is collapsed
to a fixed point of~$F$; while the identifications~(EII) and the
singletons~$\hf^r(\be(c_u))$ give rise to bi-infinite heteroclinic orbits of
connected $3$-pronged and $1$-pronged singularities: these orbits converge to the
fixed point $g(\bP)$ under iteration of $F^{-1}$, and to the periodic
orbit~$g(\bQ)$ under iteration of~$F$. There are therefore regular or pronged
charts for the invariant foliations at every point of~$\Sigma$ except for the~$k+1$
points of $g(\bP) \cup g(\bQ)$.

In the NBT case, the element $\hf^r(\be(\hz_u))$ of~(EII) coincides with
$\be(c_u)$, so that the sequences of $1$-pronged and $3$-pronged singularities
cancel out. In this case $F\colon\Sigma\to\Sigma$ is pseudo-Anosov, with an
$n$-pronged singularity at~$g(\bP)$, and $1$-pronged singularities at the points
of~$g(\bQ)$ (which is a period~$n+2$ orbit, since $f^n(a)=c$ in this case).

\bibliographystyle{amsplain}
\bibliography{mparefs}

\end{document}